\documentclass[11pt, english, reqno]{amsart}
\usepackage{fullpage}
\usepackage[utf8]{inputenc}
\usepackage{microtype}
\usepackage{graphics}
\usepackage{amsmath,amssymb,stmaryrd,array,multirow,dsfont,mathrsfs}
\usepackage{amsthm}
\usepackage{enumerate}
\usepackage{xcolor}
\usepackage{color}
\usepackage{hyperref}
\usepackage{indentfirst}
\usepackage{pifont}
\graphicspath{{Images/}}
\xdefinecolor{darkgreen}{rgb}{0,0.4,0}
\newcommand{\beq}{\begin{equation}}
\newcommand{\eeq}{\end{equation}}
\newcommand{\beqs}{\begin{equation*}}
\newcommand{\eeqs}{\end{equation*}}
\newcommand{\ben}{\begin{eqnarray}}
\newcommand{\een}{\end{eqnarray}}
\newcommand{\beno}{\begin{eqnarray*}}
\newcommand{\eeno}{\end{eqnarray*}}
\renewcommand{\Re}{{\rm Re}\,}
\renewcommand{\Im}{{\rm Im}\,}
\renewcommand{\div}{{\rm div}\,}

\newcommand{\Id}{{\rm Id}\,}
\newcommand{\Supp}{{\rm Supp}\,}

\newcommand{\Rmnum}[1]{\uppercase\expandafter{\romannumeral #1} }
 \numberwithin{equation}{section}

\newtheorem{thm}{Theorem}[section]
\newtheorem{lem}[thm]{Lemma}
\newtheorem{prop}[thm]{Proposition}
\newtheorem{rmk}[thm]{Remark}
\newtheorem{cor}[thm]{Corollary}

\def\curl{\mathop{\rm curl}\nolimits}

\def \d {\mathrm {d}}

\def\cF{{\mathcal F}}

\def\cP{{\mathcal P}}

\let\f=\frac
\def \p {\partial}

\def\pa {\partial^{\alpha}}
\def\ep{\varepsilon}

\def \ltr{\langle t\rangle}
\def \ltsr{\langle t-s \rangle}
\def \lsr{\langle s\rangle}
\def \vep {\varepsilon}

\def \pab {\partial _{\xi}^{\alpha}\partial _{\eta}^{\beta}}

\def\lnr {\langle \nabla \rangle}
\def \lxr{\langle \xi \rangle}
\def \ler{\langle \eta \rangle}
\def \lxmer{\langle \xi-\eta \rangle}
\def \lxper{\langle \xi+\eta \rangle}
\def\lemsr{\langle \eta-\sigma \rangle}
\def\lsir {\langle \sigma \rangle}

\def \pt {\partial_{t}}
\def \vr {\varrho}

\def\R {\mathcal{R}}

\def \si {\sigma}

\def \hdb {\dot{\Delta}}
\def \epd {\varepsilon \Delta}
\def\tphi {\tilde{\phi}}

\def \na{\nabla}

\def \lp {\lambda_{+}}
\def \lm {\lambda_{-}}

\def \define {\triangleq}
\def \kpz {\kappa_{0}}
\def \izt {\int_{0}^{t}}

\def \fd {\frac{\delta}{5}}

\def \lesim {\lesssim}

\begin{document}
\title{Long-term regularity of two dimensional Navier-Stokes-Poisson equations}
\author{Changzhen Sun}

\address{Laboratoire de Math\'ematiques d'Orsay (UMR 8628), Univ. Paris-Sud, CNRS, Universit\'e Paris-Saclay, 91405 Orsay Cedex, France}
\email{changzhen.sun@universite-paris-saclay.fr}
\maketitle
\begin{abstract}
This article is devoted to the long-term regularity of the 2-d Navier-Stokes-Poisson system. We allow the initial density to be close to a constant and the potential part of the initial velocity to be small independently of the rescaled viscosity  parameter $\varepsilon$ while the rotational part of the initial velocity is assumed to be small compared to  $\varepsilon$.  
We then show that the lifespan of the system $T^{\varepsilon}$ satisfies $T^{\varepsilon}>\varepsilon^{-(1-\vartheta)}$,
where the small constant $\vartheta$ is the size of the initial perturbation in some suitable space. 
The normal form transformation and the classical parabolic energy estimates are the main ingredients of the proof.
\end{abstract}
%\tableofcontents

\section{Introduction}

In the present paper, we are concerned with the large time regularity of the scaled two dimensional Navier-Stokes-Poisson system, which is a hydrodynamical model of  
plasma describing the dynamics of  electrons and ions that interact with their self-consistent electric field. If the motion of ions is negligible, the dynamics of electrons can be described as the following Navier-Stokes-Poisson (NSP) system:
\beq \label{NSPO}
 \left\{
\begin{array}{l}
\displaystyle \pt \rho^{\varepsilon} +\div( \rho^{\varepsilon} u^\varepsilon)=0,\\
\displaystyle \pt ( \rho^{\varepsilon} u^{\varepsilon} )+\div(\rho^{\varepsilon}u^{\varepsilon}\otimes u^{\varepsilon} )-\varepsilon \mathcal{L} u^{\varepsilon}+
\nabla P(\rho^{\varepsilon})-\rho^{\varepsilon}\nabla \varphi ^{\varepsilon}=0 ,  \\
\displaystyle \Delta \varphi^{\varepsilon} =\rho ^{\varepsilon}-1,\\
\displaystyle u|_{t=0} =u_0^{\varepsilon} ,\rho|_{t=0}=\rho_0^{\varepsilon}.
\end{array}
\right.
\eeq
Here $(t,x)\in \mathbb {R}_{+}\times \mathbb{R}^2,$ the unkowns $\rho^{\varepsilon} (t,x)\in \mathbb{R}_{+}, u^{\varepsilon}\in \mathbb{R}^{2},\na\varphi^{\varepsilon}\in\mathbb{R}^{2}$ are the electron density, the  electron velocity and the  self-consistent electric field respectively. The thermal pressure $P(\rho^{\varepsilon})$ is assumed to follow a polytropic $\gamma$-law:  $P(\rho^{\varepsilon})=\f{(\rho^{\varepsilon})^{\gamma}}{\gamma},\gamma>1$, while the viscous term is under the form
 $ \mathcal{L}u^{\varepsilon}=\mu \Delta u^{\varepsilon}+(\mu+\lambda)\na \div u^{\varepsilon}$, where the Lam\'e coefficients  $\mu,\lambda$ are supposed to be constants which satisfy the condition:
 $\mu>0, 2\mu+\lambda>0.$ For the conciseness of the presentation, we shall assume $\mu=1,\lambda=0$ and $\gamma=2$,
 since there are no specific cancellations arising from this choice.
 Note also that the scaled parameter $\varepsilon$ in front of the diffusion term is the inverse of the Reynolds number and is assumed to be small in this paper, that is $\ep\in(0,1].$
 
There has been extensive studies concerning the global well-posedness of (NSP) under small and smooth perturbations of the constant equilibrium
 ($(\rho^{\ep},u^{\ep})=(1,0)$)
  when the scaled parameter $\ep=1$ and the spatial dimension $d=3$.
  We refer for example to  \cite{MR2609958} \cite{MR2917409}, where the global existence in $H^{N}$ for $N\geq 4$ is proved under the assumption that the initial perturbation is small in $H^{3}$. Besides, the smoothing effect of the diffusion term allows them to prove some time decay rate for the perturbation. In \cite{MR2609958}, the global existence of \eqref{NSPO} is obtained in hybrid Besov spaces when the initial perturbation is close to equilibrium in  critical $L^2$ norm by considering low  and high frequencies
    differently, inspired by the former work on compressible Navier-Stokes equations \cite{MR1779621}. This result was then generalized to the  critical $L^{p}$ spaces \cite{MR3695805,MR3926073,MR2914601}. %see also optimal time decay results in critical $L^p$ space.
    Although these works are done in spatial dimension $d=3$, global existence in $H^N$ ($N\geq 3$) for $d=2$ could still be proved by the same arguments as in \cite{MR2917409}.
    That is, by using the dissipation for $u$ provided by the diffusion term $\ep\Delta u^{\ep}$ and the damping for $\rho^{\ep}-1$ resulting from the coupling structure, one could control the nonlinear term as long as some lower order Sobolev norm of $(\rho^{\ep}-1,u^{\ep},\na\varphi^{\ep})$ is small. Nevertheless, if we consider the scaled equations, that is 
    $\varepsilon\in(0,1],$ %it is not hard to see that 
    the above strategy requires the initial perturbation $(\rho_0^{\ep}-1,u_0^{\ep},\na\varphi_0^{\ep})$ to be small proportional to $\varepsilon$ in some suitable Sobolev spaces.
    
  On the other hand, when $\ep=0$, the system \eqref{NSPO} reduces to the so-called Euler-Poisson (EP) equation. Regarding (EP),  Guo \cite{MR1637856}  constructs in dimension d=3
  the global smooth solutions close to the reference equilibrium $(1,0)$ under neutral ($\int_{\mathbb{R}^2}(\rho_{0}^{\ep}-1)\d x=0$), irrotational, small perturbation to the equilibrium. The good dispersive properties due to the presence of the electric field and the normal form transformation technique developed by Shatah \cite{MR803256}  %dealing with the quadratic nonlinearity, (change the quadratic nonlinear term to the cubic one), 
  are the main two ingredients of his proof. More recently, similar results was obtained in dimension $d=2$ by Ionescu-Pausader
  \cite{MR3024265} and Li-Wu \cite{MR3274788} independently and $d=1$ by Guo-Han-Zhang \cite{MR3595365}. See also the result about the large time regularity of 2-d (EP) on the torus \cite{MR3927088}.
  
 Nevertheless, since in practical physics, the Reynolds number is usually very high (that is $\ep$ very small), it is natural to ask global existence results that hold uniformly in $\ep$. 
 In \cite{rousset2019stability}, we successfully combine the parabolic energy estimate which works for (NSP) and normal form transformation used in the works for (EP) to prove a uniform stability result for 3-d (NSP) system. That is, %we allow the initial density to be close to constant $(say \vr=1)$ and the potential part of velocity $\mathcal{P}^{\perp}u^{\ep}$ small independent of $\ep$, and only the incompressible  part of velocity $\mathcal{P}u^{\ep}$ small proportional to $\ep$. Note that the assumption on $\mathcal{P}u^{\ep}$ is natural if one remember that global existence for EP is true only for irrotational initial data. get global smooth solutions for  perturbations
   we construct the global smooth solutions around the constant equilibria $(1,0)$ with a smallness assumption on the perturbation which is independent of $\ep$ except for the curl part of the velocity (recall that for $\ep= 0$  we have global smooth solutions only for irrotational data). In this paper, we aim to prove the analogous results in 2-d. However, due to the weaker dispersion in 2-d, the rotational part of the velocity is driven by a source term whose $L_x^2$ norm enjoys only at best the critical time decay  $(1+t)^{-1}$.  Consequently, the rotational part of the velocity has a logarithmic growth which prevents one from establishing the global existence.
   %it is thus unclear whether or not the global existence still holds as in the 3-d case. 
 Therefore, in this paper, we only devote ourselves to proving a large time existence result of 2-d (NSP) system.
   %Nevertheless, we shall prove the following long term regularity of 2-d (NSP) system.
     
   We shall denote  by   $\mathcal{P}=Id-\na \Delta^{-1} \div$ the Leray projector which projects a vector to its divergence-free (or rotational) part. Denote also $\mathcal{P}^{\perp}=Id-\mathcal{P},$ which projects a vector field to its curl-free (or potential) part.   
\begin{thm}\label{maintheorem}
 There exist two constants $\vartheta_0,C,$ such that for 
 any $\varepsilon\in(0,1],\mathcal{\vartheta}\in(0,\vartheta_0],$
 %which are independent of $\ep,$ such that 
 if  the following assumption holds:
 \beno
 \|(\rho_0^{\ep}-1,\cP^{\perp}u_{0}^{\ep},\na\varphi_0^{\ep})\|_{Y^4}\leq \f{1}{C}\vartheta, \qquad
 \|\cP u_0^{\ep}\|_{H^{3}}\leq \vartheta \ep,
 \eeno
 where the $Y^4$ norm is defined in \eqref{def of initial norm 0},
 then the system \eqref{NSPO} admits a solution 
 in $C([0,T),H^3)$ with  $T>\ep^{-(1-\vartheta)}$. %Here, we denote $\mathcal{P}=Id-\na\Delta^{-1}\div$ the Leray projector, $\mathcal{P}^{\perp}=Id-\mathcal{P}$, and 
\end{thm}

\begin{rmk}
We remark that the assumption we add on the rotational part of the initial velocity, that is, $\mathcal{P}u_0^{\ep}$ small proportional to $\ep$ is rather natural. Indeed, as the rotational part $\mathcal{P}u^{\ep}$ is driven by a source term of size $\ep$,  even if we assume it to vanish initially
, a rotational part of size $\ep$ is instantaneously created.
\end{rmk}

\begin{rmk}
As explained before, the rotational part of the velocity satisfies at the leading order a heat equation with a source term of size $\varepsilon$ and critical time decay in  $L_x^2,$ so that we can not even extend its existence time to $\varepsilon^{-1}.$ %only extend its existence time to $\varepsilon^{1-\vartheta}$
\end{rmk}
  
A natural attempt to prove Theorem \ref{maintheorem} is to consider the highly coupled equations for the potential part and rotational part of the velocity respectively. One expects to use the dispersive property from the coupled equations (see \eqref{eqU} below) for density and potential part of velocity to prove the time decay of $\rho^{\ep}-1$ and $\mathcal{P}^{\perp}u^{\ep},$ and to use the smoothing effect to prove the large time existence of of the rotational part $\mathcal{P}u.$ Nevertheless, there will be some difficulties stemming from the interactions between the rotational part and the potential part. To be more precise, on one hand, by the dispersive estimate, the $L_{x}^{\infty}$ norm of the solution to $\eqref{eqU}$ decays at best at rate $(1+t)^{-1}.$
On the other hand, since $\mathcal{P}u^{\ep}$ is governed by the heat equation with a source term $\ep\mathcal{P}\big[(\f{1}{\rho^{\ep}}-1)\Delta\mathcal{P}^{\perp}u^{\ep}\big],$ whose $L_x^2$ norm has at best critical time decay $(1+t)^{-1}$,
%we see that $\mathcal{P}u^{\ep}$
it is far from being bounded in $L_x^2,$ %(indeed, it has at least logarithmic growth since the source term has at best critical time decay $(1+t)^{-1}$), 
which in turn prevents us from proving the time decay of $(\rho^{\ep}-1, \mathcal{P}^{\perp}u^{\ep})$ in $L_x^{\infty}.$ %This explains why  we can not even prove a better lifespan. Although meticulous, this approach seems not helpful to get a better lifespan. %due to the growth of rotaional part $\mathcal{P}u^{\varepsilon}.$
Moreover, owing to the presence of the diffusion term, the eigenvalues $\lambda_{\pm}=\ep\Delta\pm i\sqrt{\lnr^2-(\ep\Delta)^2}$ of the linearized matrix for the dispersive part of the system are far from $\pm i\lnr$ $-$ the eigenvalues for (EP). It seems necessary to cut the frequency to isolate the dispersive effects and dissipation effects, which forces us to control the interactions between different frequencies. 
We prefer not to take this way since it is more sophisticated to treat the potential-rotational interactions and low-high frequencies interactions in the same time.

Another attempt is to write the solution of NSP $(\rho^{\ep}-1,\na\varphi^{\ep},u^{\ep})$ by that of EP $(\rho^{0},\na\varphi^{0},u^{0})$ plus a remainder, and try to control the remainder by the dissipation term.
However in that case the equation satisfied by the remainder has source term $\ep \Delta u^{0}$ which has size $\ep$ but without any decay, this forces the remainder to grow linearly (in $L_x^2$) and leads to
the time existence only at order $O(1).$ 
%with respect to time and prevent us from closing the estimate because of the quadratic nonlinearity.

We shall thus adopt the same strategy developed in \cite{rousset2019stability} where the uniform stability problem for 3d NSP is investigated. More precisely, we split the (NSP) into two viscous systems, with initial data $(\rho_0^{\varepsilon}-1,\na \varphi_0^{\varepsilon},\mathcal{P}^{\perp}u_0^{\varepsilon})$ and $(0,0,\mathcal{P}u_0^{\varepsilon})$ respectively. The first one  will have  global solutions under $\ep$-independent assumptions on the initial data $(\rho_0^{\varepsilon}-1,\na \varphi_0^{\varepsilon},\mathcal{P}^{\perp}u_0^{\varepsilon})$ and the solutions
will enjoy the same time decay as the  2-d (EP) system.
 The other is  the perturbation of the original system \eqref{NSPO} by the former one. The source term $\ep (\rho-1)\Delta u$ in this system  is small compared to
  $\ep$ and has critical decay in $L_x^2$. We can thus get the desired lifespan by merely energy estimates. More precisely, we write the solution $(\rho^{\varepsilon},u^{\varepsilon},\na\varphi^{\varepsilon})$
  of (NSP) as
  $$(\rho^{\varepsilon},u^{\varepsilon},\na\varphi^{\varepsilon})  =(\rho,u,\na\varphi)+(n,v,\na\psi),$$
 where
$(\rho,\na\phi,u)$ and $(n,\na\psi,v)$ are the solutions of the following systems 
 \beq \label{NSPlow}
 \left\{
\begin{array}{l}
\displaystyle \pt \rho +\div( \rho u)=0,\\
\displaystyle \pt u+u \cdot {\na u}-\varepsilon \mathcal{L} u+
\nabla \rho-\nabla \varphi=0 ,  \\
\displaystyle \Delta \varphi =\rho -1,\\
\displaystyle u|_{t=0}=u_0= \mathcal{P}^{\perp}u_0^{\varepsilon}, \rho|_{t=0}=\rho_0=\rho_0^{\varepsilon}.
\end{array}
\right.
\eeq
 \beq \label{NSPP0}
 \left\{
\begin{array}{l}
\displaystyle \pt n +\div( \rho v+nu+nv)=0,\\
\displaystyle \pt v+u\cdot {\na v}+v\cdot (\na u+\na v)-\varepsilon\mathcal{L}v +\na n -\na \psi
=\varepsilon(\f{1}{\rho+n}-1) (\mathcal{L}v+\mathcal{L}u),   \\
\displaystyle \Delta \psi=n,\\
\displaystyle v|_{t=0} =\mathcal{P}u_0^{\varepsilon}, n|_{t=0}=0.
\end{array}
\right.
\eeq
Note that we skip the $\varepsilon$ dependence of the solutions in our notation for the last two systems. We also point out that we choose this kind of splitting mainly to ensure that the smooth solutions of the first system remain irrotational which is crucial to establish the global existence.
%will release us from considering the potential-curl interactions so that essentially leads to the global existence.
As we shall see below,
 system \eqref{NSPlow} is a good viscous approximation of the Euler-Poisson system, in the sense that the linear part of this system  has the same decay properties  for low  frequencies as  the (EP) system, that is  for localized initial data, the $L_x^{p}$ norm of $\na(\rho-1,\na \phi,u)$ decay like $(1+t)^{-(1-\f{2}{p})}$ uniformly for $\ep\in(0,1]$.

To prove the global existence of \eqref{NSPlow}, we shall use the following norm for the initial data: 
\beq\label{def of initial norm 0}
\begin{aligned}
\|(\rho_0-1,u_0,\na\varphi_0)\|_{Y^{\si}}&\define \|(\rho_0-1,u_0,\na\varphi_0)^{L}\|_{W^{\si+4,1}}+\|x(\rho_0-1,u_0,\na\varphi_0)^{L}\|_{H^{\sigma+4+\delta}}\\
&\qquad+\|x(\rho_0-1,u_0,\na\varphi_0)^h\|_{L^2}+\|(\rho_0-1,u_0,\na\varphi_0)\|_{H^{11+2\si}}
\end{aligned}
\eeq
where $\si\geq 0$  is a positive parameter and $\delta=\f{1}{1000}.$

We now state our results for system \eqref{NSPlow} and \eqref{NSPP0}:
\begin{thm}\label{thmlow}
Let $\sigma\geq 0.$
There exist two constants $C_1>0$, $\vartheta_1>0$ %which are independent of $\ep$ 
such that for any $\varepsilon\in(0,1],$ any $\bar{\vartheta}\in(0,\vartheta_1]$ if $$\|(u_0,\rho_0-1,\na\varphi_0)\|_{Y^{\si}}\leq \bar{\vartheta},$$
then the system \eqref{NSPlow} admits a global solution $(u,\vr,\na\varphi)$ in $C([0,\infty),H^{\sigma+7})$, which enjoys the uniform (in $\ep$) time decay: for any $t>0,$
\beno
(1+t)\|(\rho-1,\na u,\na\varphi)(t)\|_{W^{\si,\infty}}+\|(\rho-1,\na u,\na\varphi)(t)\|_{H^{\sigma+7}}\leq C_1\bar{\vartheta}.
\eeno
\end{thm}

Once the above theorem proved, 
Theorem \ref{maintheorem} is an easy consequence of the following one:
\begin{thm}\label{thmper}
Let $(\rho,u,\na\phi)$ be the solutions constructed in the Theorem \ref{thmlow} with $\sigma=4$. There exists $C>1,\vartheta_0\in(0,\vartheta_1] $, such that for any $\varepsilon\in(0,1],\vartheta\in(0,\vartheta_0],$ if the following assumption holds: $$\|(\rho_0^{\varepsilon}-1,\mathcal{P}^{\perp}u_0^{\varepsilon},\na\varphi_0^{\ep})\|_{Y^{4}}\leq \f{\vartheta}{C},\quad \|\mathcal{P}u_{0}^{\ep}\|_{H^3}\leq \vartheta\ep,$$
then the system \eqref{NSPP0} admits a solution in $C([0,T),H^3)$ with $T>\ep^{-(1-\vartheta)}$.
\end{thm}

Since Theorem \ref{thmper} is quite easy to obtain by merely energy estimates, we shall only explain the difficulties and strategies for proving Theorem \ref{thmlow}.

 We introduce the new unknown $U=(\f{\lnr}{|\na|}(\rho-1),\f{\div}{|\na|} u)$. Using the curl-free condition, it suffices for us to consider the system:
 \begin{equation}\label{disperpart}
 \pt U+AU=F(U, U), \quad
A=\left(
  \begin{array}{cc}
    0&\lnr\\
    -\lnr&-2\vep \Delta\\
  \end{array}
\right).
 \end{equation}
where $F$ is a quadratic form  defined in \eqref{eqU}. %in with coefficients composed by singular integral operators which are bounded in $L^p (1<p<\infty),$ see 
Simple computations shows that the eigenvalues of $\hat{A}(\xi)$ are
$\lambda_{\pm}=-\ep|\xi|^{2}\pm \sqrt{1+|\xi|^{2}-\ep^2|\xi|^{4}}\define -\ep|\xi|^{2}\pm b(\xi)$.
To present our ideas about decay estimate, we thus consider the toy model (we change the nonlinear term for clarity since there will be no loss of derivative in energy estimate):
%$$\pt \beta -\epd \beta +i b(D)\beta =\beta^2$$.
   \beqs
   \left\{
   \begin{array}{c}
   \displaystyle\pt\beta-\epd \beta +i b(D)\beta=\beta^2 \\
   \displaystyle\beta|_{t=0}=\beta_{0}
   \end{array}
   \right.
   \eeqs
   where $\beta\in \mathbb{C}^{2}$.
 As indicated in the 3d case \cite{rousset2019stability}, we need to consider different frequencies to isolate the dispersive effects and dissipation effects. %The main observations are that 
 On one hand, when we focus on low frequency
 (say $\ep |\xi|^{2}\lesim \kpz$ where $\kpz$ is very small but independent of $\ep$), 
 $e^{\ep t\Delta}$ is useless since we want to get estimate uniformly for $\ep,$
 but we can expect that $e^{itb(D)}$  behave very like $e^{it\lnr}$, which shall provide us the dispersive estimate uniformly in $\ep\in(0,1]$. On the other hand, when we focus on the high frequency (say $\ep |\xi|^{2}\geq \kpz$), we have
that $\Re \lambda_{\pm}\leq -c(\kpz)$
where $c(\kpz)$ is independent of $\ep\in(0,1]$, the operator $e^{t\lambda_{\pm}(D)}$ can
provide exponential decay, we thus expect the solution has good decay even in $L_x^{2}$ norm.

In practice, we first try to introduce some norms which indicate the decay properties for both low and high frequency. In order to do this, let us choose a compactly supported function $\chi$, which equals to 1 on the unit ball $B_{1}(\mathbb{R}^{2})$ and vanishes outside $B_{2}(\mathbb{R}^{3}).$ Denote then $\chi^{L}(\xi)=\chi(\sqrt{\f{\ep}{\kpz}}\xi)$ , $\chi^{H}=1-\chi^{L}$
where $\kpz$ is a threshold to be chosen later.
We define the norm (the reason for evolving the weighted norm will be explained later):
$$\|\beta\|_{X_T}=\sup_{t\in [0,T)}\ltr\|\beta^{L}(t)\|_{W^{1,\infty}}+\|xe^{itb(D)}\beta^{L}\|_{H^4}%+\ltr^{-\delta}\|x\beta^{H}\|_{L^2}
+(1+t)\|\beta^{H}(t)\|_{H^9}+\ltr^{-\delta}\|\beta(t)\|_{H^{10}}+\|\beta(t)\|_{H^8}.$$
where $\beta^{L}=\chi^{L}(D)\beta, \beta^{H}=\chi^{H}(D)\beta.$
Note that in the definition of the norm, we have time decay of rate $(1+t)^{-1}$ rather than $e^{-ct}$ for high frequency due to the weak decay property for low frequency.
To get the a priori estimate, we need to consider several interactions between different frequencies. %namelylow$\times$ low $\rightarrow$ low, low $\times$ high $\rightarrow$ low, high$\times$ high $\rightarrow$ low,  low$\times$ low $\rightarrow$ high, low $\times$ high $\rightarrow$ high, high$\times$ high $\rightarrow$ high.
However, due to the slow decay provided by dispersive estimate for low frequency, the low frequency output of the interactions between the low frequency and the high frequency  is difficult to close.
%it is not enough to merely cut the frequency into two parts namely low and high frequency, as   we could expect the best the $L^2$ decay of high frequency is $(1+t)^{-1}$, which is not enough to close the decay estimate for low frequency.

%For the high frequency, we have by Duhamel principle and the good property for $e^{\lambda_{-}(D)t}\chi^{h}$ (see Lemma \ref{high frequency estiamte}),
%\beno
%&&\|\beta^H\|_{H^6}\lesim \|e^{\lambda_{-}(D)t}\beta_{0}^{H}\|_{H^6}+\izt \|e^{\lambda_{-}(D)(t-s)}\chi^{h}(D)\beta^{2}\|_{H^6}\d s\\
%&\lesim& e^{-ct}\|\beta_{0}^{H}\|_{H^6}+\int_{0}^{t}e^{-c(t-s)}\|\beta(s)\|_{H^6}\|\beta(s)\|_{L^{\infty}}\d s\lesim e^{-ct}\|\beta_{0}^{H}\|_{H^6}+(1+t)^{-1}\|\beta\|_{X}^2.
%\eeno
More precisely, in order to estimate the low frequency, by rewriting $\beta^{2}=(\beta^{L})^{2}+2\beta^{L}\beta^{H}+(\beta^{H})^2$, we need to estimate the term $\izt e^{\lambda_{-}(D)(t-s)}\chi^{L}(D)(\beta^{H}\beta^{L})(s)\d s$ which can be estimated as, by dispersive estimate for $e^{itb(D)}\chi^{L}$  (see Lemma \ref{lem dispersive})
\ben
&&\|\izt e^{\lambda_{-}(D)(t-s)}\chi^{L}(D)(\beta^{H}\beta^{L})(s)\d s\|_{L^{\infty}}\lesim \izt (1+t-s)^{-1}\|\beta^{H}\beta^{L}\|_{W^{2,1}}\d s\nonumber\\
&\qquad&\lesim
\izt (1+t-s)^{-1}(1+s)^{-1}\d s\|\beta\|_{X}^{2}\lesim (1+t)^{-\iota}\|U\|_{X}^{2}\label{eq0}
\een
where $0<\iota <1$. Unfortunately, the desired case
$\iota=1$ is not true. To overcome this difficulty,
we need more accurate splitting of frequencies.
We observe that one can indeed split the frequency into three parts, namely, lowest frequency: $\{\ep|\xi|^2\leq \kpz\}$, intermediate frequency $\{\f{\kpz}{2}\leq\ep|\xi|^2\leq 3\kpz\}$ and highest frequency $\{\ep|\xi|^2\geq \f{5}{2}\kpz\}$.
%where the highest frequency and the lowest frequency are far away from each other. 
In this way, due to the lack of interaction lowest$\times$ lowest $\rightarrow$ highest, we could expect the highest frequency enjoys faster decay. What is more, the intermediate frequency part has also good decay since on this region, we  have $e^{t\lambda_{\pm}(\xi)}\leq e^{-ct}$ for some $c>0$ independent of $\ep$. The lowest frequency is now manageable since for the lowest$\times$intermediate$\rightarrow$ lowest interaction we could use normal form transformation
(see details below)
by noticing that the intermediate frequency still lies in the region that dispersive property holds. To summarize, after some crude analysis, we expect the lowest frequency part enjoys the 
$L_x^{\infty}$ decay of $(1+t)^{-1}$, the intermediate part enjoy the $L_x^{p}$ $(2\leq p\leq 4)$ decay like $(1+t)^{-(2-\f{2}{p})}$, and the high frequency parts has $L_x^2$ decay like $(1+t)^{-2}$. We explain for instance the high frequency case. By choosing three smooth radial function $\chi^{l},\chi^{m},\chi^{h}$ which satisfies $\chi^{l}(\xi-\eta)\chi^{l}(\eta)\chi^{h}(\xi)=0$ (see the definition in Section 2) %In this case, we have:$(\beta^{l}\beta^{l})^{h}=0$
 and defining $\chi^{L}=\chi^{l}+\chi^{m}, \beta^{L}=\beta^{l}+\beta^{m}$, one can  write $(\beta^2)^{h}=\big(2\beta^{L}\beta^{h}+(\beta^{h})^2+(\beta^m)^{2}+2\beta^{l}\beta^{m}\big)^{h}$.
We could expect the worst part $\beta^l\beta^m$ enjoys $L^2$ decay of $(1+t)^{-2}$.
%5which is enough for us to close the $L^{\infty}$ decay estimate of $\beta^{L}$. Lastly, we note that this manner of cutting frequency is necessary  when one would like to get some estimate for $x\beta^{H}$.

We thus need to modify our norm to be (with $N>10$ to be chosen) :
\beq\label{normintr}
\begin{aligned}
\|\beta\|_{X_{T}}&\define\sup_{t\in[0,T)} \ltr \|\beta^{L}\|_{W^{1,\infty}}+\ltr\|\beta^{m}\|_{H^{N-1}}+\ltr^2\|\beta^{h}\|_{H^{N-1}}\\
&\qquad\quad+\|xe^{itb(D)\beta^L}\|_{H^4}+\ltr^{-\delta}\|\beta\|_{H^N}+\|\beta\|_{H^{N-2}}
\end{aligned}
\eeq
 We now explain  Low$\times$Low$\rightarrow$ Low estimate %when deal with $L_{x}^{\infty}$ estimatefor $\beta^L$. 
where only dispersive estimate is available. To overcome the difficulty of quadratic nonlinearity, the normal form transformation (or more generally 'space-time resonance' philosophy \cite{germain2010space}) should be enforced. To be more precise, we set  $\alpha=e^{itb(D)}\beta^L$ and write 
\ben\label{eq1}
\izt e^{-i(t-s)b(D)}e^{\ep(t-s)\Delta}\chi^L(D)(\beta^L)^2\d s=\cF^{-1}(e^{-itb(\xi)}\izt e^{is\phi_{1}(\xi,\eta)}e^{\ep(t-s)|\xi|^2}\chi^{L}(\xi)
\hat{\alpha}(\xi-\eta)\hat{\alpha}(\eta)\d\eta\d s)
\een
%One could easily check that 
where $\phi_1=b(\xi)-b(\xi-\eta)-b(\eta)<0$ 
 on the support of $\chi^L(\xi)\chi^L(\xi-\eta)\chi^L(\eta)$. Following the 'space-time resonance' philosophy, by identity $e^{is\phi_1}=\f{1}{i\phi_1}\p_{s}e^{is\phi_1}$, one integrate by parts in time so that \eqref{eq1} becomes 
 \ben\label{eq00}
 -\int_{0}^{t}e^{-i(t-s)b(D)}e^{\ep(t-s)\Delta}\chi^{L}(D)\big(\ep\Delta T_{\f{1}{i\phi_1}}(\beta^{L},\beta^{L})+T_{\f{1}{i\phi_1}}(\ep\Delta{\beta^{L}}+(\beta^2)^{L},\beta^{L})\big)\d s 
 \een
 plus boundary terms and symmetric terms which can be handled similarly. Here, $T_{\f{1}{i\phi_1}}$ is the bilinear operator defined by \eqref{defbilinear}. Note that we have also used the equation satisfied by $\alpha$: $$\pt\alpha=\epd \alpha+e^{itb(D)}(\beta^2)^L.$$
 
In view of \eqref{eq00}, besides the viscous terms, we need to estimate the typical term:
$$\izt e^{\lambda_{-}(D)(t-s)}\chi^{L}(D)(\beta^{L}\beta^L)^{L}\beta^{L}(s)\d s.$$  Nevertheless, the same problem like \eqref{eq0} emerges, since we could estimate $\|(\beta^{L}\beta^{L})^{L}\beta^{L}\|_{W^{2,1}}$ by $\|\beta^{L}\|_{H^2}^{2}\|\beta^L\|_{L^{\infty}}$ which has only the decay $(1+s)^{-1}$. Following %It has been pointed out in 
\cite{MR3024265},\cite{MR3274788}, the 'vector field-like' norm $e^{-itb(D)}xe^{itb(D)}\beta^L$ needs to be involved to detect some space resonance information of the phase function.
%that is why weighted norm is necessary in the a priori norm.
%{\color{red}explain more}
%The same problem also emerges when we deal with the $(\beta^{L})^2$ term, after using normal form transformations,(that is writing the solution by Duhamel principle in Fourier space, and using the fact that there is no 'time resonance' which allows us to integrate by parts in time and change the quadratic terms into cubic, see the next subsection for the proof of quadratic Klein-Gordon equation), we need to consider the term $\izt e^{\lambda_{-}(D)(t-s)}\chi^{L}(D)(\beta^{L}\beta^L)^{L}\beta^{L}(s)\d s$.Nevertheless, 

We now explain the extra difficulty due to the dissipation term $\epd \beta$. By noticing that  $e^{\ep (t-s)\Delta}\chi^{L}\epd$ is a multiplier in $L_x^2$ with norm $(1+t-s)^{-1}$,
we expect that $\epd \beta^{L}$ has $L^2$ decay like $(1+t)^{-1}$.  %What is more, when do the decay estimate of low frequency, we need to use more decay of $\|\epd \beta^L\|_{L^{p}}$ ($p>2$),by the elementary computation, one could expect the $L^p$ norm of $\epd \beta^L$ behaves like $(1+t)^{-\min(2-\f{2}{p},1+\f{2}{p})}$, we thus could   add $\ltr ^{-\f{3}{2}}\|\epd \beta^L\|_{L^{4}}$ into the a priori norm.
However, we shall still encounter the difficulty that $\|(\beta^L)^2\|_{L^2}$ have decay like $(1+t)^{-1},$ which forces us to use normal form transformation (or integrate by parts in time) again. This will increase the complexity of computations. The trick to simplify the arguments is that in the process of performing normal form transformations, we could introduce $\tilde{\alpha}=e^{-\ep t\Delta }e^{itb(D)}\beta^L$ as the intermediate profile.
By defining the complex phase function $\phi=i\phi_1+\ep(|\xi|^2-|\xi-\eta|^2-|\eta|^2)$ which does not vanish on the support of $\chi^L(\xi)\chi^L(\xi-\eta)\chi^L(\eta)$, we could integrate by parts in time as before to get:
\ben\label{eq2}
&&\izt e^{-i(t-s)b(D)}e^{\ep(t-s)\Delta}\chi^L(D)(\beta^L)^2\d s=\cF^{-1}(e^{-itb(\xi)}\izt e^{is\phi(\xi,\eta)}\chi^{L}(\xi)
\hat{\alpha}(\xi-\eta)\hat{\alpha}(\eta)\d\eta\d s)
\nonumber\\
&=&boundary \quad terms +\int_{0}^{t}e^{-i(t-s)b(D)}e^{\ep(t-s)\Delta}T_{\f{1}{\phi}}((\beta^2)^{L},\beta^{L})\d s+ symmetric \quad term.
\een
which allows us not to care about $\epd \beta^L$.
Note that there is no singularity on $\f{1}{\phi}$ since $i\phi_1$ never vanishes. Moreover, computations (see Section 3) show that the bilinear operator $T_{\f{1}{\phi}}$ enjoys the same good quasi-product estimates as $T_{\f{1}{\phi_1}}$. The strategy for dealing with this term shall then have  similarities with \cite{MR3024265}, \cite{MR3274788} where the global existence for 2-d (EP) is proved.

\textbf{ Organization of the paper}:
%In Section 2, we introduce some notations,
We first introduce some notations in Section 2. To prove Theorem \ref{thmlow},  some reformulations  and useful lemmas
(linear estimates, bilinear estimates) are presented in Section 3. The local existence in weighted space for system \eqref{NSPlow} shall be shown in Section 4. 
 Section 5 to Section 8 are dedicated to establish several a priori estimates. The conclusion for Theorem \ref{thmlow} are then made in Section 9. Theorem \ref{thmper} shall be proved in Section 10. Finally, in appendix, we sketch the proofs for part of low frequency estimates.

\section{Notations}
 $\bullet$ We denote $a_{+}(resp.a_{-})$ for a constant larger (resp.smaller) but arbitrarily closed to $a.$

$\bullet$ We choose three radial smooth functions $\chi_1,\chi_2,\chi_3:\mathbb{R}^{2}\rightarrow \mathbb{R}$ st.$\chi_1+\chi_2+\chi_3=1$ and
$\Supp\chi_1(\xi) \subset \{\xi\big||\xi|\leq 1\}$, $\Supp\chi_2(\xi)\subset\{\xi\big|\f{1}{2}\leq|\xi|\leq 3\}$, $\Supp\chi_3(\xi)\subset\{\xi\big||\xi|\geq \f{5}{2}\}$. Denote $\chi^{l}=\chi_1(\sqrt{\f{\ep}{\kpz}}\xi)$, $\chi^{m}=\chi_2(\sqrt{\f{\ep}{\kpz}}\xi)$,
$\chi^{h}=\chi_3(\sqrt{\f{\ep}{\kpz}}\xi)$,
$\chi^{L}=\chi^{l}+\chi^{m}$,
$\chi^{H}=\chi^{m}+\chi^{h}.$
We also write:
$f^{L}=\cF^{-1}(\chi^{L}(\xi)\cF{f}(\xi))$,
$f^{l}=\cF^{-1}(\chi^{l}(\xi)\cF{f}(\xi))$,
$f^{m}=\cF^{-1}(\chi^{m}(\xi)\cF{f}(\xi))$,
$f^{H}=\cF^{-1}(\chi^{H}(\xi)\cF{f}(\xi))$.

 $\bullet$ We define the bilinear operator $T_m(f,g)$ and trilinear operator $T_{\tilde{m}}(f,g,h)$
 \ben
 T_m(f,g)&\define&\cF^{-1}(\int m(\xi%-\eta
 ,\eta)\hat{f}(\xi-\eta)\hat{g}(\eta)\d \eta)\label{eqbilinear}\label{defbilinear}\\%\nonumber\\
 %&=&\f{1}{(2\pi)^2}\int m(\zeta,\eta)\hat{f}(\zeta)\hat{g}(\eta)e^{ix(\zeta+\eta)}\d \zeta\d \eta
  T_{\tilde{m}}(f,g,h)&\define&\cF^{-1}(\int m(\xi%-\eta
 ,\eta,\si)\hat{f}(\xi-\eta)\hat{g}(\eta-\si)\hat{h}(\eta)\d \eta)\label{trilinear}
 \een
 
 $\bullet$ We recall the classical  Littlewood-Paley decomposition: choose a cut-off function $ \Psi, 0\leq \Psi \leq 1, \Psi\equiv 1$ on $B_{3/2%\f{3}{2}
 }$ and vanish on $B_{5/3}%\f{5}{3}}
 ^{c}$. We set 
 \beq\label{jthdyadic}
 \Phi_j(x)=\Phi(\f{x}{2^{j}}),\quad \text{where} \quad \Phi(x)=\Psi(x)- \Psi(2x).
 \eeq
Note that $\Phi(x)$ supported on the annulus $\{\f{3}{4}\leq|x|\leq \f{5}{3}\}$ and  $1=\Psi(x)+\sum_{j\in\mathbb{N}^{*}}\Phi_{j}(x)$.
 Recall the homogeneous dyadic block:
 $\hdb_{k}f\define \cF^{-1}(\Phi_{k}(\xi)\hat{f}(\xi))$ ($k\in \mathbb{Z}$), inhomogeneous dyadic block:
 $\Delta_{-1}f\define \cF^{-1}(\Psi(\xi)\hat{f}(\xi)), \Delta_{l}f\define \cF^{-1}(\Phi_{k}(\xi)\hat{f}(\xi)),(l\in \mathbb{N})$, and $S_{k}=\sum_{-1\leq j\leq k-1}\Delta_j$.
 %$$

\section{Preliminaries}
Set $\vr=\rho-1$, system \eqref{NSPO} is equivalent to the following system:
 \beq \label{ANSP1}
 \left\{
\begin{array}{l}
\displaystyle \pt \vr +\div u+\div( \vr u)=0,\\
\displaystyle  \pt u+u \cdot {\na u}-\varepsilon  \mathcal{L} u+
\nabla \vr-\nabla \varphi=0 ,  \\
\displaystyle \Delta \varphi =\vr\\
\displaystyle u|_{t=0} =\mathcal{P}^{\perp}u_0^{\ep}, %\mathcal{P}^{\perp}u_0^{\varepsilon} ,
\vr|_{t=0}=\rho_0^{\varepsilon}-1
\end{array}
\right.
\eeq
We first remark that since $\curl (\mathcal{P}^{\perp}u_0^{\ep})=0$, standard energy estimates indicate that this curl-free property will propagate as long as smooth solution exists. Note also that by identity
$\Delta u=-\curl\curl u+\na\div u$, we have $\mathcal{L}u=\Delta u+\na\div u=2\Delta u$.\\
To symmetrize the system, we first introduce the new unkowns:
 \beqs
 \mathrm{a}=\f{\lnr}{|\na|} \vr,\quad \mathrm{c}=\f{\div}{|\na|} u ,\qquad U=(\mathrm{a},\mathrm{c})^{\top}
 \eeqs
 It is direct to see that $(\mathrm{a},\mathrm{c})$ satisfies the system:
  \beq \label{eqsym}
 \left\{
\begin{array}{l}
\displaystyle \pt \mathrm{a} + \lnr \mathrm{c}=\lnr\f{\div} {|\na|}\bigg(\big(\f{|\na|}{\lnr}\mathrm{a}
%|\na|/\lnr h
\big) \R \mathrm{c}\bigg)=\lnr \R\cdot\bigg(\big(\f{|\na|}{\lnr}\mathrm{a}\big) \R \mathrm{c}\bigg) \\%=\lnr \R^{*}\big(\f{|\na|}{\lnr}h\cdot \R c\big)
\displaystyle \pt \mathrm{c}-\lnr \mathrm{a}-2\vep \Delta \mathrm{c}=\f{1}{2}\f {\div}{|\na|} \na |\R \mathrm{c}|^2=-\f{1}{2}%\R \cdot \na 
|\na||\R \mathrm{c}|^2\\
\displaystyle \mathrm{a}|_{t=0}=\f{\lnr} {|\na|}\vr_0,\mathrm{c}|_{t=0}= \f {\div}{|\na|}u_0 \\
\end{array}
\right.
 \eeq
ie.
\begin{equation}\label{eqU}
\pt U+\left(
  \begin{array}{cc}
    0&\lnr\\
    -\lnr&-2\ep \Delta\\
  \end{array}
\right)
U=\left(
  \begin{array}{c}
    \lnr \R\cdot\big((\f{|\na|}{\lnr}\mathrm{a})\cdot \R \mathrm{c}\big) \\
-\f{1}{2}%\R\cdot 
  |\na| |\R \mathrm{c}|^2\\
  \end{array}
\right)\define\left(
  \begin{array}{c}
    F_1(\mathrm{a},\mathrm{c})\\
F_2(\mathrm{a},\mathrm{c})\\
  \end{array}
\right)= F(\mathrm{a},\mathrm{c})
%\triangleq B(V,V)
\end{equation}
where we denote $\R=\f{\na}{|\na|}$ the Riesz potential. %$\R^{*}=\f{\div}{|\na|}$ is the conjugate operator of $\R$. 
Note also that we have used the fact that $u=\R \mathrm{c}$ which is a consequence of $\curl u=0$. \\
Define
\ben\label{def of A}
A=\left(
  \begin{array}{cc}
    0&\lnr\\
    -\lnr&-2\vep \Delta\\
  \end{array}
\right)
\een
   By elementary computation, we get that the eigenvalues of $-\hat{A}(\xi)$ are:
   \beq \label{eqA}
   \lambda _{\pm}=-\vep |\xi|^2\pm i \sqrt{1+|\xi|^2-\vep^2 |\xi|^4}\define -\vep |\xi|^2\pm ib(\xi)
   \eeq
   where we cut the lower half imaginary axis in the definition of square root of a complex function.
 What is more, one can easily check that the Green matrix is
   \ben\label{def of green function}
   e^{-t\hat{A}(\xi)}&=&\frac{1}{\lambda_{+}-\lambda_{-}}
   \left(
  \begin{array}{cc}
    \lp e^{\lm t}-\lm e^{\lp t}&(e^{\lm t}-e^{\lp t})\lxr\\
    (e^{\lp t}-e^{\lm t})\lxr&\lp e^{\lp t}-\lm e^{\lm t}\\
  \end{array}
\right)
   \define
    \left(
  \begin{array}{cc}
   \mathcal{G}_1&-\mathcal{G}_2\\
   \mathcal{G}_2&\mathcal{G}_3\\
  \end{array}
\right)
   \een
  Note $\mathcal{G}_1,\mathcal{G}_2,\mathcal{G}_3$ are well defined everywhere since there is no singularity when $\lambda_{+}=\lambda_{-}$.
 When we focus on the Low frequency, ie, when $\vep|\xi|^2\leq 3\kpz<<1$, we can diagonalize  $A$ as:
   \ben\label{Qdef}
A(D)&=&\left(
  \begin{array}{cc}
    1&1\\
    -\f{\lm(D)}{\lnr}&-\f{\lp(D)}{\lnr}\\
  \end{array}
\right)
\left(
  \begin{array}{cc}
    -\lm & 0\\
   0 & -\lp \\
  \end{array}
\right)
\left(
  \begin{array}{cc}
    \lp&\lnr\\
    -\lm &-\lnr\\
  \end{array}
\right)\f{1}{2ib}\nonumber\\
&\define&Q\left(
  \begin{array}{cc}
    -\lm & 0\\
   0 & -\lp \\
  \end{array}
\right)Q^{-1},\quad Q^{-1}=\left(
  \begin{array}{cc}
    \lp&\lnr\\
    -\lm &-\lnr\\
  \end{array}
\right)\f{1}{2ib}.
\een

 We denote then $W=Q^{-1}\chi^{L}U\define(w,\bar{w})$ for which the first component satisfies the equation:
 \ben\label{eq of w}
 \pt{w}-\lambda_{-}(D)w&=&\f{\lambda_{+}}{2ib}\chi^L(D)F_1(\mathrm{a},\mathrm{c}) %\lnr\R\chi^{L}(\R U_1 \R U_2)
 +\f{\lnr}{2ib}\chi^L(D)F_2(\mathrm{a},\mathrm{c})%\f{|\na|}{2}\chi^{L}|\R U_2|^{2}
 \nonumber\\
 &=&\f{\lambda_{+}}{2ib}\chi^L(D)F_1(\mathrm{a}^L,\mathrm{c}^L)_+\f{\lnr}{2ib}\chi^L(D)F_2(\mathrm{a}^L,\mathrm{c}^L)\nonumber \nonumber\\
 &&\qquad\qquad+[Q^{-1}\chi^L(D)(F(\mathrm{a}^h,\mathrm{c}^L)+F(\mathrm{a}^L,\mathrm{c}^h)+F(\mathrm{a}^h,\mathrm{c}^h))]_{1} \nonumber\\
 %&=&\R\big(B(w,w)+\lnr\chi^{L}(\R U^{h}\R U^{L}+\R U^{h}\R U^{h})\big)\nonumber\\
 &\define&\R\big(B(w,w)+\lnr\chi^{L}H\big).
 \een
where, $H=\R \mathrm{a}^L\R \mathrm{c}^h+\R \mathrm{c}^L\R \mathrm{a}^h+\R \mathrm{a}^h+\R \mathrm{c}^h\approx \R U^L\R U^h+\R U^h \R U^h$ and by relation $\mathrm{a}^L=w+\bar{w}, \mathrm{c}^L=-(\f{\lambda_{-}}{\lnr}w+\f{\lambda_{+}}{\lnr}\bar{w})$,
$B(w,w)$ is defined by
\beq\label{def of B}
\cF{B(w,w)}=\sum_{\mu,\nu\in\{+,-\}}\int m_{\mu\nu}(\xi,\eta)\widehat{\R w^{\mu}}(\xi-\eta)\widehat{\R w^{\nu}}(\eta)\d \eta
\eeq
with $m_{\mu,\nu}(\xi,\eta)=\lxr n_{\mu}(\xi-\eta)n_{\nu}(\eta)\chi^{L}(\xi)\chi^{L}(\xi-\eta)\chi^{L}(\eta),$
$n_{+}\in\{-\f{\lambda_{-}}{\lnr},1\},n_{-}\in\{-\f{\lambda_{+}}{\lnr},1\}$ and exponent
$\{\pm\}=\{Id, conjugate\}$.

%n(D)=\{\chi^{L}\f{\lambda_{+}}{2ib}\lnr,\chi^{L}\f{\lnr}{2ib}|\na|\},
Note that in the above (and hereafter), for notational convenience, we do not make difference between ‘real’ Riesz potential and general 
zero order Fourier multipliers %which have the similar properties %as the standard Riesz potential 
whose symbol satisfies zero homogeneous condition and is smooth away from the origin, since they have similar properties. For example, they are both bounded operators in $L^p (1<p<\infty).$
Moreover, we do not distinguish the scalar Riesz potential $\f{\na_j}{|\na|}(j=1,2)$ and vector one $\f{\na}{|\na|}.$ 
One easily checks that in the above,
$\R$ can represent anyone of the set $\{\f{\div}{|\na|}\f{\lambda_{+}}{2ib(D)}\chi^{L}(D),\f{|\na|}{4ib}\chi^L(D),\f{|\na|}{\lnr},\f{\na}{|\na|}\}$.
%Before defining the a priori norm, we denot

After recalling the definition:
$U^{L}=\chi^{L}(D)U=\chi^{l}(D)U+\chi^{m}(D)U=U^{l}+U^{m}, U^{m}=\chi^{m}(D)U, U^{h}=\chi^{h}(D)U$,
we define the following norm:
 \beq\label{def of real norm}
 \begin{aligned}
 &\|U\|_{X^{\si}_T}
 \define\sup_{t\in[0,T)}\ltr\||\na|^{\f{1}{2}}\lnr Q^{-1}U^L(t)\|_{W^{\si,\infty}}%+\ltr^{1-2\delta}\|U^{L}\|_{W^{2,\f{1}{\delta}}}
 +\|xe^{itb(D)}w(t)\|_{W^{\si+4,\f{2}{1-\delta}}}+\|U^{L}(t)\|_{H^{\si+N'}}\\
% +\ltr^{\f{3}{2}}\|\ep\Delta U^{L}\|_{W^{\sigma_1,4}}\\
 &%+\ltr\|\ep\Delta U^{L}\|_{H^{\sigma}}
\quad +\ltr^{1-3\delta}\|U^{m}(t)\|_{H^{2\si+N-1}}+\ltr^{\f{3}{2}}\|U^m(t)\|_{W^{1,4}}+\ltr^{\alpha}\|U^{h}(t)\|_{H^{2\si+N-2}}+\ltr^{-\delta}\|U(t)\|_{H^{2\si+N}},
\end{aligned}
 \eeq
where $N'=7, N=N'+4$, $\alpha=2-5\delta$ and $\delta$ is chosen to be very small (say $\delta=\f{1}{1000}$). 
We comment that the norm defined above is slightly different from \eqref{normintr} mainly due to the presence of Riesz potential in the nonlinear term (see \eqref{eq of w}). 
We shall prove the global existence of system \eqref{eqU} in the Banach space $X^{\si}_T$ defined by the norm $\|\cdot\|_{X_T^{\si}}$. In the sequel, for notational
clarity, we shall assume that $\si=0$ (and denote $X_T^{0}=X_T$), since the case $\si>0$ can be easily generalized.
We first remark that by dispersive estimate \eqref{dispersive} and H\"{o}lder's inequality, for any $0\leq t<T$
\beno
\|U^L(t)\|_{W^{2,\f{1}{\delta}}}\lesim 
\|w(t)\|_{W^{2,\f{1}{\delta}}}\lesim \ltr^{-(1-2\delta)} \|e^{itb(D)}w(t)\|_{W^{4(1-\delta),\f{1}{1-\delta}}}\lesim \ltr^{-(1-2\delta)}\|U\|_{X_T}.
\eeno
Moreover, we have:
\beno
\|\na u(t)\|_{L^{\infty}}=\|\na \R U(t)\|_{L^{\infty}}&\leq&\|\na \R U^{L}(t)\|_{L^{\infty}}+\|\na \R U^{h}(t)\|_{L^{\infty}}\\
%&\lesssim&\sum_{k}2^{k}\|\hdb_{k}w\|_{L^{\infty}}+\ltr^{-\alpha}\|U\|_{X}\\
&\lesssim&\sum_{k}2^{\f{1}{2}k}\langle2^{k}\rangle^{-1}\|\hdb_{k}|\na|^{\f{1}{2}}\lnr w(t)\|_{L^{\infty}}+\ltr^{-\alpha}\|U\|_{X_T}\\
&\lesssim&\||\na|^{\f{1}{2}}\lnr w(t)\|_{L^{\infty}}+\ltr^{-\alpha}\|U\|_{X}
\lesssim\ltr^{-1}\|U\|_{X_T}.
\eeno

%\beno\|U^L\|_{W^{\si+2,\f{1}{\delta}}}\lesim \|w\|_{W^{\si+2,\f{1}{\delta}}}\lesim \ltr^{-(1-2\delta)} \|e^{itb(D)}w\|_{W^{\si+4(1-\delta),\f{1}{1-\delta}}}\lesim \ltr^{-(1-2\delta)}\|U\|_{X^{\si}}.\eeno
%Moreover, we have:
%\beno\|\na u\|_{W^{\si,\infty}}=\|\na \R U\|_{W^{\si,\infty}}&\leq&\|\na \R U^{L}\|_{W^{\si,\infty}}+\|\na \R U^{h}\|_{W^{\si,\infty}}\\
%&\lesssim&\sum_{k}2^{k}\|\hdb_{k}w\|_{L^{\infty}}+\ltr^{-\alpha}\|U\|_{X}\\&\lesssim&\sum_{k}2^{\f{1}{2}k}\langle2^{k}\rangle^{-1}\|\hdb_{k}|\na|^{\f{1}{2}}\lnr w\|_{W^{\si,\infty}}+\ltr^{-\alpha}\|U\|_{X}\\&\lesssim&\||\na|^{\f{1}{2}}\lnr w\|_{W^{\si+\infty}}+\ltr^{-\alpha}\|U\|_{X}\lesssim\ltr^{-1}\|U\|_{X}.\eeno

In the following of this section, we will give some preliminary lemmas which will be used later.

\subsection{Linear estimates}
We present in this subsection the linear estimates for Low (lowest and intermediate) and highest frequency.
%dispersive estimate for $e^{itb(D)}\chi^{L}$ and smoothing effect for $e^{-tA}\chi^{h} $ which correspond to Low  estimates respectively for the linearized system whose proof can be found in Corollary 3.7 and Lemma 3.5 respectively of \cite{rousset2019stability}.

\subsubsection{Linear estimates for Low frequency}
\begin{lem}{Dispersive estimate for $e^{itb(D)}\chi^L$.}\label{lem dispersive}\\
For every $\kpz$ is small enough (say $ \kpz\leq \f{1}{200}$) and for any $2\leq p\leq\infty$, we have the following dispersive estimate: %which is the same as Klein-Gordon operator $e^{it\lnr}$,
\beq\label{dispersive}
\|e^{itb(D)}\chi^{L}(D)f\|_{L^p}\lesssim_{\kpz}(1+|t|)^{-(1-\f{2}{p})}\|f\|_{W^{2(1-\f{2}{p}),p'}},\quad \forall t\in \mathbb{R}.
\eeq
\end{lem}
\begin{proof}
Indeed, \eqref{dispersive} holds when 
$e^{itb(D)}\chi^L(D)$ is replaced by $e^{it\lnr},$ which follows from the classical stationary phase arguments. Nevertheless, when $\kpz$ is chosen small enough, $e^{itb(D)}\chi^L(D)$
enjoys the similar algebraic properties as $e^{it\lnr}.$
 One can refer to Corollary 3.7 of \cite{rousset2019stability} for the detailed proof.
\end{proof}
\begin{lem}{$L^p\rightarrow L^p $ boundedness for $e^{itb(D)}\chi^L$.}
\label{Lp bounds for eitb(D)}

Suppose $\kpz\leq\f{1}{200}$.
For any $1<p<\infty$,  we have the following estimate:
\beq\label{Lp bounds}
\|\Delta_{k}e^{itb(D)}\chi^{L}(D)u\|_{L^p}\lesssim_{\kpz}\ltr^{|1-\f{2}{p}|}\langle 2^{k}\rangle^{|1-\f{2}{p}|}\|\Delta_{k}u\|_{L^p},  (k\geq -1)
\eeq
\beq\label{Lp bounds1}
\|e^{itb(D)}\chi^{L}(D)u\|_{L^p}\lesssim_{\kpz}\ltr^{|1-\f{2}{p}|}\|u\|_{W^{s,p}}.
\eeq
where $s>|1-\f{2}{p}|$.
\end{lem}
\begin{proof}
%We will need also the boundedness of $e^{itb(D)}\chi^{L}(D)$ from $L^p$ to $L^p$. Noticing that $e^{itb(D)}\chi^L(D)$ behaves very like $e^{it\lnr}$,
This lemma has essentially been proved in Lemma 2.2 of \cite{MR3274788} where $e^{it\lnr}$ rather than $e^{itb(D)}\chi^L(D)$ is considered. We will 
sketch the proof of \eqref{Lp bounds} for $p=1,2,\infty$, the other case for \eqref{Lp bounds} and \eqref{Lp bounds1} follows from interpolation and summation respectively.
%One only need to prove the above two inequalities for . 
By Young's inequality, it suffices for us to show:
\beq\label{eq3}
\|\cF^{-1}\big(\Phi_{k}(\xi)e^{itb(D)}\chi^{L}(\xi)\big)\|_{L^1}\lesim_{\kpz}\ltr\langle 2^{k}\rangle,  k\geq 0; \qquad \|\cF^{-1}\big(\Psi(\xi)e^{itb(D)}\chi^{L}(\xi)\big)\|_{L^1}\lesim_{\kpz}\ltr
\eeq
where $\Phi_k,\Psi_k$ is defined in \eqref{jthdyadic}.
To prove \eqref{eq3}, one uses the inequality: $\|f\|_{L^1}\lesim \|f\|_{L^2}^{\f{1}{2}}\|x^2f\|_{L^2}^{\f{1}{2}}$ and
elementary estimate:
$$\|\Phi_{k}(\xi)e^{itb(D)}\chi^{L}(\xi)\|_{L^2}\lesssim 2^{k},\qquad \|\p_{\xi}^2\big(\Phi_{k}(\xi)e^{itb(D)}\chi^{L}(\xi)\big)\|_{L^2}\lesssim_{\kpz} 2^{-k}\langle 2^{k}t\rangle^2,$$ %2^{k} t^2+t+2^{-k}
$$\|\Psi(\xi)e^{itb(D)}\chi^{L}(\xi)\|_{L^2}\lesssim 1,\qquad \|\p_{\xi}^2\big(\Phi_{k}(\xi)e^{itb(D)}\chi^{L}(\xi)\big)\|_{L^2}\lesssim_{\kpz} \langle t\rangle^2.$$ 
\end{proof}
\subsubsection{Linear estimate for high frequency}
\begin{lem}{%Exponential damping 
Linear estimate for $e^{-tA}\chi^{h}$.}\label{high frequency estiamte}

There exists a constant $c=c(\kpz)$, such that,
for any real number $s$, we have:
\beqs
\|e^{-tA}\chi^{h}U\|_{H^s}\lesssim_{\kpz} e^{-ct}\|U\|_{H^s}.
\eeqs
\end{lem}
\begin{proof}
One needs to study carefully the Green matrix \eqref{def of green function} localized on high frequency, since the algebraic computations does not depend on the dimension, one can refer to Lemma 3.5 of \cite{rousset2019stability} where the similar property is shown in dimension 3.
\end{proof}
\subsubsection{Additional estimate for intermediate frequency}

For the intermediate frequency, we could use the spectral localization to get the boundedness of $e^{-tA}\chi^m$ from $W^{|1-\f{2}{p}|_{+},p}$ to $L^p (1<p<\infty).$ %with norm $e^{-\f{1}{5}\kpz t}$.
\begin{lem}\label{cor immediate norm }
Recall
$\chi^{m}(\xi)=\chi_{2}(\sqrt{\f{\ep}{\kpz}}\xi)$ where $\chi_2$ is smooth function supported on $\{\xi\big|\f{1}{2}\leq|\xi|\leq 3\}$.
%which corresponds to the spectral localization on intermediate frequency is defined in Section 2,
We have for any $1<p<\infty$,
\beqs
\|e^{-tA}\chi^{m}(D)u\|_{L^p}\lesssim_{\kpz}e^{-\f{\kpz}{5} t}\|u\|_{W^{|1-\f{2}{p}|_{+},p}}.
\eeqs
\end{lem}
\begin{proof}
We first prove
\beqs
\|e^{\ep t\Delta}\chi^{m}(D)f\|_{L^p}\lesssim_{\kpz}e^{-\f{1}{4}\kpz t}\|u\|_{L^p}.
\eeqs
which follows from the Young's inequality and the fact:
$\|f\|_{L^1}\lesssim \|f\|_{L^2}^{\f{1}{2}}\|x^2f\|_{L^2}^{\f{1}{2}}.$
Indeed, one has:
\beno
\|\cF^{-1}\big(e^{-\ep|\xi|^{2}t}\chi^{m}(\xi)\big)\|_{L^1}&=&\|\cF^{-1}(e^{-\kpz\xi|^{2}t}\chi_2(\xi))\|_{L^1}\\
&\lesssim&\|\cF^{-1}\big(e^{-\kpz|\xi|^{2}t}\chi_2(\xi)\big)\|_{L^2}^{\f{1}{2}}\|x^2\cF^{-1}\big(e^{-\kpz|\xi|^{2}t}\chi_2(\xi)\big)\|_{L^2}^{\f{1}{2}}\\
&\lesssim& e^{-\f{1}{2}\kpz t}\langle\kpz t\rangle\lesssim e^{-\f{1}{4}\kpz t}.
\eeno
By the definition of the Green matrix,
\beno
   e^{-t\hat{A}(\xi)}&=&\frac{1}{\lambda_{+}-\lambda_{-}}
   \left(
  \begin{array}{cc}
    \lp e^{\lm t}-\lm e^{\lp t}&(e^{\lm t}-e^{\lp t})\lxr\\
    (e^{\lp t}-e^{\lm t})\lxr&\lp e^{\lp t}-\lm e^{\lm t}\\
  \end{array}
\right).
   \eeno
and eigenvalue $\lambda_{\pm}=\epd\pm ib(D)$ , we see that $e^{-tA}$ is indeed the combination of terms like $e^{\lambda_{\pm}(D)}q(D)$
where $q(D)\in\{\f{\ep\Delta}{b(D)},\f{\lnr}{b(D)},Id\}$. Therefore, by Lemma \ref{Lp bounds for    eitb(D)} and the definition of $\chi^m$,
\beno
\|e^{itb(D)}e^{\ep t\Delta}\chi^{m}q(D)(u)\|_{L^p}&\lesssim&
\|e^{itb(D)}\tilde{\chi}^{m}e^{\ep t\Delta}\chi^{m}q(D)\|_{L^p}\\
&\lesssim_{\kpz}& \ltr^{|1-\f{2}{p}|}\|e^{\ep t\Delta}\chi^{m}q(D)u\|_{W^{|1-\f{2}{p}|_{+},p}}\\
&\lesssim_{\kpz}& \ltr^{|1-\f{2}{p}|}e^{-\f{1}{4}\kpz t}\|\chi^m q(D)u\|_{W^{|1-\f{2}{p}|_{+},p}}\\
&\lesssim_{\kpz}& e^{-\f{1}{5}\kpz t}\|u\|_{W^{|1-\f{2}{p}|_{+},p}}.
\eeno
 \end{proof}
\subsection{Bilinear estimates}

As we shall use the normal form transformation, it is necessary
for us get some continuous properties for  bilinear operators defined by \eqref{defbilinear}.
To start, we present some elementary properties of bilinear multipliers which is useful to derive the bilinear estimates.
\begin{prop}\label{derivative estiamte for phase}
 Define the phase function $$\phi_{\mu,\nu}(\xi+\eta,\eta)=i(b(\xi+\eta)-\mu b(\xi)-\nu b(\eta))+Z(\xi,\eta),$$
and multiplier function $$m_{\mu,\nu}(\xi+\eta,\eta)=\lxper{\chi^{L}(\xi)\chi^{L}(\eta)\chi^{L}(\xi+\eta)}%n(\xi+\eta)
n_{\mu}(\xi)n_{\nu}(\eta)$$
where $b(\xi)=\sqrt{1+|\xi|^{2}-\ep^{2}|\xi|^{4}}, Z(\xi,\eta)=\ep(|\xi+\eta|^2-|\xi|^2-|\eta|^2)$
 and $\mu,\nu\in\{+,-\}.$
$n_{+}\in\{-\f{\lambda_{-}}{\lnr},1\}$,
$n_{-}\in\{-\f{\lambda_{+}}{\lnr},1\}$.
Suppose $\kpz\leq \f{1}{200}$,
then for any multi-index $\alpha,\beta\in \mathbb{N}^2$,  the following estimate hlods uniformly in $\ep\in(0,1]$: 
\beqs
|\pab \f{m_{\mu,\nu}}{\phi_{\mu,\nu}}(\xi+\eta,\eta)|\lesssim_{\alpha,\beta,\kpz}\lxper \min\{b(\xi),b(\eta),b(\xi+\eta)\}
\eeqs
\end{prop}
\begin{proof}
We only present the proof for $\mu=\nu='+',$ since the others are easier or can be obtained by symmetry.
 At first, we have
\beno
\f{1}{\phi_{++}}&=&i\f{b(\xi)+b(\eta)+b(\xi+\eta)+iZ(\xi,\eta)}{(b(\xi)+b(\eta)+iZ(\xi,\eta))^2-b^2(\xi+\eta)}\\
&\define& i \f{b(\xi)+b(\eta)+b(\xi+\eta)+iZ(\xi,\eta)}{B}.
\eeno
where 
\beno
B&=&\big(b(\xi)+b(\eta)\big)^2-b^2(\xi+\eta)-Z^2(\xi,\eta)+2iZ(\xi,\eta)\big(b(\xi)+b(\eta)\big)\\
&\define& A-Z^2(\xi,\eta)+2iZ(\xi,\eta)\big(b(\xi)+b(\eta)\big) 
\eeno
Note that $A$ has the lower bound:
\ben\label{Alowbdd}
A&=&1+2b(\xi)b(\eta)-2\xi\cdot\eta+\ep^{2}(|\xi|^{4}+|\eta|^{4}-|\xi+\eta|^{4})\nonumber\\
&\geq& 1-27\kpz^2+2b(\xi)b(\eta)-2\xi\cdot\eta\nonumber\\
&=&\f{(1-27\kpz^2+2b(\xi)b(\eta))^{2}-4|\xi\cdot\eta|^2}{1-27\kpz^2+2b(\xi)b(\eta)+2\xi\cdot\eta}%\f{\ler}{\lxr}
\gtrsim\f{(b(\xi)+b(\eta))^{2}}{b(\xi)b(\eta)}\gtrsim 1.%\nonumber\\
\een

We will prove that on the support of $\chi^{L}(\xi)\chi^{L}(\eta)\chi^{L}(\xi+\eta)$, for any multi-index $\alpha,\beta\in \mathbb{N}^2$, the following property holds:
\beq\label{ineq for B}
|\pab\f{1}{B}|\lesssim_{\alpha,\beta,\kpz}\f{1}{|B|}
\eeq
which is an easy consequence of Leibniz's rule and
\beq\label{derivative of B}
|\pab{B}|\lesssim_{\alpha,\beta,\kpz}|B|,\quad \forall \alpha,\beta\in \mathbb{N}^2.
\eeq
However, \eqref{derivative of B} can be derived once we have the estimate for $A$:
\beq\label{derivative of A}
|\pab{A}|\lesssim_{\alpha,\beta,\kpz}A,\quad \forall \alpha,\beta\in \mathbb{N}^2.
\eeq
Indeed, since %once we have \eqref{derivative of A}, estimate \eqref{Alowbdd} combined with the fact:
\beqs
|\pab Z^{2}(\xi,\eta)|+|\pab Z(\xi,\eta)(b(\xi)+b\big(\eta)\big)|\leq P(\alpha,\beta,\kpz), \quad for  (\xi,\eta)\in \Supp m(\xi,\eta)
\eeqs
where $P(\alpha,\beta,\kpz)$ is a polynomial with respect to $\kpz$ which can be bounded by a constant $C(\alpha,\beta)$ %uniform in $\ep\in(0,1]$
if we choose $\kpz$ small (say $\kpz\leq\f{1}{200}$), we can use \eqref{derivative of A},\eqref{Alowbdd}
to get that:
\beqs
|\pab{B}|\leq |\pab{A}|+C(\alpha,\beta)\lesssim_{\alpha,\beta,\kpz} A+C(\alpha,\beta)\lesssim_{\alpha,\beta,\kpz}A\lesssim_{\alpha,\beta,\kpz}|B|.
\eeqs
%In the following, we will devote to prove 
Nevertheless, we note that the estimate of \eqref{derivative of A} has been proved in the appendix of \cite{rousset2019stability}.
Therefore, inequality \eqref{ineq for B} holds, which leads to the following computation:
\beno
\big|\pab \f {m_{++}(\xi,\eta)}{\phi_{++}}\big|&=&\big|\sum c_{\alpha_1\alpha_2\beta_1\beta_2}\partial_{\xi}^{\alpha_1}\partial_{\eta}^{\beta_1}(m_{++}(\xi,\eta))
\partial_{\xi}^{\alpha_2}\partial_{\eta}^{\beta_2}\f {b(\xi)+b(\eta)+b(\xi+\eta)+iZ(\xi,\eta)}{B}\big|%\partial_{\xi}^{\alpha_3}\partial_{\eta}^{\beta_3}{\f{1}{A}}
\\
&\lesssim_{\kpz}&\lxper(b(\xi)+b(\eta)+b(\xi+\eta))\f{1}{A}\\
&\lesssim_{\kpz}&\lxper \min \{b(\xi),b(\eta),b(\xi+\eta)\}.
\eeno
\end{proof}
This proposition in hand, we then show the following bilinear estimate:
\begin{lem}\label{lem bilinear estimate}
Let $m_{\mu\nu},\phi_{\mu\nu}$ being defined as the last proposition, one has bilinear estimate:
\beq \label{bilinear estimate}
 \|T_{\f{m_{\mu\nu}}{\phi_{\mu\nu}}}(f,g)\|_{L^{p}}\lesim_{\kpz} \|f\|_{W^{2_{+},q_1}}\|g\|_{W^{2,r_1}}+\|f\|_{W^{2,r_2}}\|g\|_{W^{2_{+},q_2}}
\eeq
     where    $\f{1}{p}=\f{1}{q_1}+\f{1}{r_1}=\f{1}{q_2}+\f{1}{r_2},$ $1<r_1,r_2\leq +\infty, 1\leq q_1,q_2<+\infty,$
     $T_{\f{m_{\mu\nu}}{\phi_{\mu\nu}}}$ is the bilinear operator defined in \eqref{eqbilinear} and $k_{+}$ is a real number slightly larger than $k$.
\end{lem}
\begin{proof}
As before, we only treat the case $T_{\f{m_{++}}{\phi_{++}}}.$ Let $\psi_1,\psi_2\in C_{b}^{\infty}(\mathbb{R}^{4})$ which satisfy the following conditions:
 \begin{equation*}
\left\{
\begin{array}{l}
\displaystyle  \psi_1+\psi_2=1 \qquad\forall (\xi,\eta), \\
\displaystyle \Supp \psi_1 \subset \{ (\xi,\eta)\big| \lxmer\geq \f{\ler}{2}\},\\
\displaystyle \Supp \psi_2 \subset \{ (\xi,\eta)\big| \ler> \lxmer\}.\\
\end{array}
\right.
\end{equation*}
We write
\beno
\frac{m_{++}}{\phi_{++}}(\xi,\eta)&=&\f{m_{++}\psi_1(\xi,\eta)}{\phi_{++}\lxmer^{2_{+}}\ler^2} \lxmer^{2_{+}}\ler^2+\f{m_{++}\psi_2(\xi,\eta)}{\phi_{++}\ler^{2_{+}} \lxmer^2}\ler^{2_{+}} \lxmer^2\\
&\define& M_1(\xi,\eta)\lxmer^{2_{+}}\ler^2+M_2(\xi,\eta)\ler^{2_{+}} \lxmer^2.
\eeno
By Proposition \ref{derivative estiamte for phase}, we have for any $\alpha,\beta$ with $|\alpha|+|\beta|\leq 3$,
\beqs
|\pab{M_1}|\leq I_{\lxmer\geq \f{\ler}{2}}\lxmer^{-1_{+}}\ler^{-1}.
\eeqs
Therefore,
 $ M_{1}, \partial_{\xi}^{3}M_{1}, \partial_{\eta}^{3}M_{1}\in L^2(\mathbb{R}^{4})$,
which leads to the fact: $\mathcal{F}^{-1}(M_1)(x,y)\in L^{1}_{x,y}$. Indeed,
 \beqs
 \|\mathcal{F}^{-1}(M_1)(x,y)\|_{L_{x,y}^{1}}\lesssim \|(1+|x|^{3}+|y|^{3})^{-1}\|_{L_{x,y}^2}(\|M_1\|_{L^2}+\|\partial_{\xi}^{3}M_{1}\|_{L^2}+\|\partial_{\eta}^{3}M_{1}\|_{L^2}).
 \eeqs
By the definition of bilinear operator $T_{m}$ (\ref{eqbilinear}) and Fourier transform:
\beqs
T_{M_{1}\lxr^{2_{+}}\ler^2}(f,g)=\int (\cF^{-1}M_{1})(x',y'-x')\langle D_{x}\rangle^{2_{+}}f(x-x')\langle D^{2}_{x}\rangle g(x-y')\d x'\d y'.
\eeqs
Therefore, By Minkowski's inequality,
\beno
\|T_{M_{1}\lxr^{2_{+}}\ler^2}(f,g)\|_{L^{p}}
&\leq& \int \|\langle D^{2}_{x}\rangle g\|_{L^{r_1}}\|\int (\cF^{-1}M_{1})(x',y'-x')\langle D_{x}\rangle^{2_{+}}f(x-x')\d x'\|_{L^{q_1}}\d y'\nonumber\\
&\leq& \|\cF^{-1}M_{1}\|_{L_{x,y}^{1}} \|f\|_{W^{2_{+},p_1}}\|g\|_{W^{2,r_1}}.
\eeno
The similar result for  $M_2$ can be derived in  the same fashion.
\end{proof}

\begin{rmk}\label{rmk for bilinear}
It is easy to adapt the proof of the above lemma to get that:
 \beq \label{bilinear estimate 2}
     \|T_{\f{m_{\mu\nu}}{\phi_{\mu\nu}}}(f,g)\|_{L^{p}}\lesim_{\kpz} \|f\|_{W^{3_{+},q_1}}\|g\|_{W^{1,r_1}}+\|f\|_{W^{1,r_2}}\|g\|_{W^{3_{+},q_2}}.
  \eeq

\end{rmk}

\begin{rmk}
     Of course the norm of $T_{\f{m_{\mu\nu}}{\phi_{\mu\nu}}}$ in \eqref{bilinear estimate} and \eqref{bilinear estimate 2} is dependent on $2_{+}-2$ or $3_{+}-3$, but when we use this lemma, we fixed $2_{+}$ and $3_{+}$.
 \end{rmk}
\begin{rmk}
From now on, we will fix $\kpz=\f{1}{200}$.
\end{rmk}
\begin{cor}\label{cortrilinear}
Recall $B(w,w)\approx \R\lnr(\R w)^2$ is defined in \eqref{def of B}, the following trilinear estimate holds
\beno
\|T_{\f{m}{\phi}}(\R B(w,w),\R w)\|_{W^{\si,p}}\lesim \| w\|_{W^{2,p_1}}\| w\|_{W^{2,p_2}}\|w\|_{W^{\si+3_{+},p_3}}
\eeno
where $1<p_1,p_2,p_3<\infty$ and $\f{1}{p}=\f{1}{p_1}+\f{1}{p_2}+\f{1}{p_3}.$
\end{cor}

\subsection{Useful Lemmas for local existence.}
 In this subsection, we give some preliminary lemmas which will be used in the proof of local existence.
  \begin{lem}\label{commutator estimate}
  Let $\zeta$ be a compactly support smooth function and $\theta\in \dot{W}^{1,\infty}$, and denote $\zeta_{R}=\zeta(\f{x}{R})$.  %and recall
  %$\chi^{L}(\xi)=(\chi_1+\chi_2)(\sqrt{\f{\ep}{\kpz}}\xi)$, 
  then
  %for any $t\in(0,1]$, $2\leq p<\infty$
  %\beq
  %\|[\psi_{R},e^{it\lambda_{-}(D)}\chi^{L}(D)f\|_{L^p}\lesssim_{\kpz,\ep,p}\f{1}{R}\|f\|_{L^2}
  %\eeq
  \beqs
  \|[\zeta_{R}(\cdot),\Theta(D)]f\|_{L^2}\lesssim \f{1}{R}\|f\|_{L^2}
  \eeqs
  \end{lem}
  \begin{proof}
  %\beno
 % \cF([\psi_{R},e^{it\lambda_{-}(D)}\chi^{L}(D)f)&=&\int\widehat{\psi_{R}}(\eta)\hat{f}(\xi-\eta)[e^{it\lambda_{-}(\xi)}\chi^{L}(\xi)-e^{it\lambda_{-}(\eta)}\chi^{L}(\xi-\eta)]\d\eta\\
 % &=&\int\widehat{\psi_{R}}(\eta)\hat{f}(\xi-\eta)\int_{0}^{1}\eta\cdot\na(e^{it\lambda_{-}}\chi^{L}(\xi-\eta+\tau\eta))\d \tau \d\eta\\
 % &=&-i\int \hat{f}(\xi-\eta)\widehat{\na\psi_{R}}(\eta)\cdot\int_{0}^{1}\na_{\xi}(e^{it\lambda_{-}(\xi)}\chi^{L})(\xi-\eta+\tau\eta))\d \tau \d\eta
  %\eeno
  
  \ben\label{identity}
  \cF([\zeta_{R},\Theta(D)]f)&=&\int\widehat{\zeta_{R}}(\eta)\hat{f}(\xi-\eta)
  [\Theta(\xi-\eta)-\Theta(\xi)]\d\eta \nonumber\\
  &=&-\int\widehat{\zeta_{R}}(\eta)\hat{f}(\xi-\eta)\int_{0}^{1}\eta\cdot\na\Theta(\xi-\tau\eta)\d \tau \d\eta\nonumber\\
  &=&i\int \hat{f}(\xi-\eta)\widehat{\na\zeta_{R}}(\eta)\cdot\int_{0}^{1}\na \Theta(\xi-\tau\eta))\d \tau \d\eta.\qquad\quad
  \een
%  Therefore, 
  Therefore, by Parseval's inequality and Young's inequality,
  \beno
  \|[\zeta_{R},\Theta(D)]f\|_{L^2}&\lesssim&\|\cF([\zeta_{R},\Theta(D)]f)\|_{L^{2}}\\
  &\lesim&\|\widehat{\na\zeta_{R}}\|_{L^1}\|\hat{f}\|_{L^2}\|\na_{\xi}\Theta\|_{L^{\infty}}\lesssim\f{1}{R}\|f\|_{L^2}.
  \eeno
  Note that in the above, we have used the fact:
  $\|\widehat{\na\zeta_{R}}\|_{L^1}\lesssim\f{1}{R}.$
  \end{proof}
  \begin{cor}\label{useful for the local existence in weighted norm}
  Denote $\Phi_{j}$ the j-th dyadic function coming from Littlewood-Paley theory, and $\chi^{L}(\xi)=(\chi_1+\chi_2)(\sqrt{\f{\ep}{\kpz}}\xi)$ (See the definitions in Section 2).
  Then, for any $t\in(0,1]$, $2\leq p<\infty$,
  \beq
  \|[\Phi_{j},e^{it\lambda_{-}(D)}\lnr^{2}
  \chi^{L}(D)]f\|_{L^p}\lesssim_{\kpz,\ep,p}2^{-j}\|f\|_{L^2}.
  \eeq
  \end{cor}
  \begin{proof}
By virtue of Hausdorff's inequality, identity \eqref{identity} and Young's inequality, 
%and the fact $\lambda'(D)$ is a $L^p$ multiplier, 
we have:
  \beno
  \|[\Phi_{j},e^{it\lambda_{-}(D)}]\lnr^{2}\chi^{L}(D)f\|_{L^p}&\lesssim&\|\cF([\Phi_{j},e^{it\lambda_{-}(D)}\lnr ^{2} 
  \chi^{L}(D)]f)\|_{L^{p'}}\\
  &\lesssim&\|\widehat{\na\Phi_{j}}\|_{L^1}\|\hat{f}\|_{L^2}\|\na_{\xi}(e^{it\lambda_{-}(\xi)}\lxr^{2}
  \chi^{L})\|_{L^{\f{2p}{p-2}}}\\
  &\lesssim&2^{-j}(\f{\ep}{\kpz})^{-(\f{1}{2}-\f{1}{p})}(1+\f{\kpz}{\ep})
  \|f\|_{L^2}.
  \eeno
  \end{proof}
  We will need to estimate 'weighted product term'  like $xfg$, the following lemma allows us not to lose derivative on weighted term. 
  \begin{lem}\label{weighted product}
  For $s\geq 0$, $\iota>0$, the following weighted product estimate holds :
  \beno 
  \|xfg\|_{H^{s}}\lesim \min\left\{\|xf\|_{L^2}\|g\|_{H^{s+1+\iota}},\|xf\|_{L^{\infty}}\|g\|_{H^{s+\iota}}\right\}+\|xg\|_{L^2}\|f\|_{H^{s+1+\iota}}+\|f\|_{H^{s+\iota}}\|g\|_{H^{s+\iota}}.
  %\|xfg\|_{H^{s}}\lesim \|xf\|_{L^{\infty}}\|g\|_{H^{s+\delta}}+\|xg\|_{L^2}\|f\|_{H^{s+1+\delta}}+\|f\|_{H^{s+\delta}}\|g\|_{H^{s+\delta}}.
  \eeno
  \end{lem}
  \begin{proof}
  Write $fg=\sum_{j\geq 1}S_{j-1}f \Delta_{j}g+\Delta_{j}f S_{j}g $. Thanks to the Bernstein inequality and Young's inequality the first term can be estimated as:
  \beno
  \|x\sum_{j\geq 1}S_{j-1}f \Delta_{j}g\|_{H^s}&=&\sum_{j\geq 1}2^{js}(\|[x,S_{j-1}]f \Delta_{j}g\|_{L^2}+\|S_{j-1}(xf) \Delta_{j}g\|_{L^2})\\
  &\leq& \sum_{j\geq 1}2^{js} (\|f\|_{L^{2}}\|\Delta_{j} g\|_{L^2}+\min\{2^j\|xf\|_{L^2},\|xf\|_{L^{\infty}}\} \|g\|_{L^2})\\
  &\leq&\min\{\|xf\|_{L^2}\|g\|_{H^{s+1+\iota}},\|xf\|_{L^{\infty}}\|g\|_{H^{s+\iota}}\}+\|xf\|_{L^2}\|g\|_{H^{s+\iota}}
  \eeno
  The second term can be controlled similarly, we omit the detail.
  %It has essentially been proved in Lemma 4.2 \cite{MR3274788}. In fact, one write $xfg$ in frequency space and using identity $\p_{\xi}\hat{f}(\xi-\eta)=-\p_{\eta}\hat{f}(\xi-\eta)$ to integrate by parts in $\eta$ if $|\xi-\eta|>|\eta|.$ 
  \end{proof}
\section{Local existence and time continuity of weighted norm}
By classical iteration technique, one could construct solution for system \eqref{ANSP1} in $C([0,T_{\ep}];H^N),(N\geq 3)$ for some $T_{\ep}>0$
(please refer to \cite{MR564670} for example), which leads to the local existence for system \eqref{eqU}.
We thus focus  on the local boundedness of weighted norm  $\|xe^{itb(D)}w(t)\|_{W^{4,\f{2}{1-\delta}}}$ and its continuity in time. We start with the weighted $L_x^2$ estimate for high frequency which shall be useful later.
%The local existence in $C([0,T];H^N),N\geq 3$ is easy, 
\begin{lem}
There exists a constant $M_0>0$, such that for small but fixed time $T_0<1$,   the following a-priori estimate holds
 \beq\label{weighted high local}
 \sup_{t\in [0,T_0]}\|x(u^h,\vr^h)\|_{L^2}\lesim M_0e^{M_0T_0}T_0(1+\|x(u^L,\vr^L)\|_{L^{\infty}([0,T_0],L_{x}^{\infty})}).
 \eeq
\end{lem}
\begin{proof}
We consider the system satisfied by the high frequency:
\beq \label{NSPhighfrequency}
 \left\{
\begin{array}{l}
\displaystyle \pt \vr^{h} +\div u^{h}+\div( \vr u)^{h}=0,\\
\displaystyle \pt u^{h}+(u \cdot {\na u})^{h}-2\varepsilon \Delta u^{h}+
\nabla \vr^{h}-\nabla \varphi^{h}=0 ,  \\
\displaystyle \Delta \varphi^{h} =\vr^{h},\\
\displaystyle u|_{t=0} =\mathcal{P}^{\perp}u_0^{\varepsilon} ,\vr|_{t=0}=\vr_0^{\varepsilon}.
\end{array}
\right.
\eeq
 Set
$\psi_{R}(x)=x \Psi(\f{x}{R})$, where %we recall 
the compactly supported function $\Psi\equiv 1$ on $B_{\f{3}{2}}$ and vanish on $B_{\f{5}{3}}^{c}$.
Multiplying the system \eqref{NSPhighfrequency} by $\psi_{R}(x)$, and test $\psi_{R}(\vr^h,u^h)$, one gets the energy equality:
\beno
&&\f{1}{2}\f{\d}{\d t}\|\psi_{R}(\vr^h,u^h)\|_{L^2}^{2}\\
&=&-\int\psi_{R}^{2}(\vr^{h}\div u^{h}-\na \vr^{h} u^{h})\d x+\int \psi_{R}(\na\varphi)^{h} \psi_{R} u^{h}\d x +2\ep \int\psi_{R}\Delta u^{h}\psi_{R}u^{h} \d x\\
&&-\int\psi_{R}\div(\vr u)^{h}\psi_{R}\vr^{h}\d x-\int \psi_{R}(u\cdot\na u)^{h}\psi_{R} u^{h}\d x\\
&\define&\mathcal{T}_1+\mathcal{T}_2+\cdots \mathcal{T}_5.
\eeno
We now estimate $\mathcal{T}_1,\cdots \mathcal{T}_5.$
For $\mathcal{T}_1$,  integration by parts and  H\"{o}der inequality yield:
\beqs
\mathcal{T}_1=2\int \psi_{R}\vr^{h}\na \psi_{R}u^{h}\d x\leq \|\psi_{R}\vr^{h}\|_{L^2}\|\na \psi_{R}u^{h}\|_{L^2}\lesim \|\psi_{R}\vr^{h}\|_{L^2}\|u^{h}\|_{L^2}.
\eeqs
Note that $\na\psi_{R}$ is pointwise bounded uniformly in $R$.
For $\mathcal{T}_2$, by H\"{o}der inequality,
\beqs
\mathcal{T}_2\lesssim \|\psi_{R}u^h\|_{L^2} \|\psi_{R}\na \varphi^h\|_{L^2}\lesim \|\psi_{R}u^h\|_{L^2}(\|\psi_{R}\vr^{h}\|_{L^2}+\|\vr^{h}\|_{L^2})
\eeqs
where we have used $$\psi_{R}\na\varphi^{h}=\psi_{R}\na(\Delta)^{-1}\tilde{\chi}^{h}\vr^{h}=[\psi_{R},\na (\Delta)^{-1}\tilde{\chi}^{h}]\vr^{h}+\na (\Delta)^{-1}\tilde{\chi}^{h}(D) (\psi_{R}\vr^{h}).$$
Notice that by the Lemma \ref{commutator estimate} and the spectral localization of $\tilde{\chi}^{h}$, one has that:
\beno
\|[\psi_{R},\na (\Delta)^{-1}\tilde{\chi}^{h}]\vr^{h}\|_{L^2}&\lesssim& \|\cF^{-1}(\na\psi_R)\|_{L^1}\|\partial_{\xi}(\f{\xi}{|\xi|^{2}}\tilde{\chi}^{h})\|_{L_{\xi}^{\infty}}\|\vr^{h}\|_{L^2}\lesim \f{\ep}{\kpz}\|\vr^{h}\|_{L^2},\\
\|\na (\Delta)^{-1}\tilde{\chi}^{h}(D) (\psi_{R}\vr^{h})\|_{L^2}&\lesssim& \sqrt{\f{\ep}{\kpz}}\|\psi_{R}\vr^{h}\|_{L^2}.
\eeno
For $K_3$, it is direct to see
\beno
\mathcal{T}_3+2\ep \int |\psi_{R}\na u^{h}|^{2}\d x&=&-4\ep\int \psi_{R}u^{h}\na u^{h}\na\psi_{R}\d x\\
&\lesim& \|\psi_{R}u^{h}\|_{L^2}\|\na u^h\|_{L^2}.
\eeno
For $\mathcal{T}_4$, using again $[\psi_{R},\chi^h(D)]$ belongs to  $\mathcal{L}(L^{2}(\mathbb{R}^2))$
whose norm is independent of $R$, we get:
\beno
\mathcal{T}_4&=&-\int [\psi_{R},\chi^h(D)](u\cdot \na u)\psi_{R}u^h +\chi^{h}(D)(\psi_{R}u\cdot\na u)\psi_{R}u^h\d x\\
&\lesim&\|\psi_{R}u^h\|_{L^2}\big(\|u\cdot \na u\|_{L^2}+\|\psi_{R}u^h\|_{L^2}\|\na u\|_{L^{\infty}}+\|\psi_{R}u^L\|_{L^{\infty}}\|\na u\|_{L^2}\big).
\eeno
Similarly, for $\mathcal{T}_5$, we have:
\beno
\mathcal{T}_5&=&-\int[\psi_{R},\chi^h(D)]\div(\vr u)\psi_{R}\vr^{h}+\chi^{h}(\na\vr\cdot\psi_{R}u+\psi_{R}\vr \div u)\psi_{R}\vr^{h}\d x\\
&\lesim& \|\psi_{R}\vr^{h}\|_{L^2}\big(\|\div(\vr u)\|_{L^2}+\|\psi_{R}(u^h,\vr^h)\|_{L^2}\|(u,\vr)\|_{W^{1,\infty}}+\|\psi_{R}(u^L,\vr^L)\|_{L^{\infty}}\|(\vr,u)\|_{H^1}\big).
\eeno
 Summing up the above estimates, we finally get:
 \beno
 \pt \|\psi_{R}(u^h,\vr^h)\|_{L^2}&\lesim& \|\psi_{R}(u^h,\vr^h)\|_{L^2}(1+\|(u,\vr)\|_{W^{1,\infty}})\nonumber\\
 &&+(1+\|(u,\vr)\|_{L^{\infty}})\|(u,\vr)\|_{H^1}+\|\psi_{R}(u^L,\vr^L)\|_{L^{\infty}}\|(\vr,u)\|_{H^1}. \qquad\qquad
\eeno
 Gr\"{o}nwall's inequality then gives us for any $t<T_0$, there exists constant $M_0$ which depends on $\|(u,\vr)\|_{L^{\infty}([0,T_0], H^{1}\cap W^{1,\infty})}$,
 \beqs
 \|\psi_{R}(u^h,\vr^h)(t)\|_{L^2}\lesim M_0e^{M_0T_0}T_0(1+\|\psi_{R}(u^L,\vr^L)\|_{L^{\infty}([0,T_0],L_{x}^{\infty})}).
 \eeqs
Letting $R$ tends to $+\infty$, we obtain \eqref{weighted high local}
\end{proof}.

 One can easily adapt the above proof to derive the following Corollary:
\begin{cor}\label{weighted high by weighted low}
There exists another constant $M_0'$ which is dependent on 
$\|(u,\vr)\|_{L^{\infty}([0,T_0], H^{1}\cap W^{1,\infty}\cap W^{1,4})}$, such that for any $j\geq 0$,
 \beqs
 \sup_{t\in [0,T_0]}\|\Phi_{j}(\cdot)(u^h,\vr^h)(\cdot)\|_{L^2}\lesim M_0'e^{M_0'T_0}T_0(2^{-j}+\|\Phi_{j}(u^L,\vr^L)\|_{L^{\infty}([0,T_0],L_{x}^{4})}).
 \eeqs
where $\Phi_{j}$ is the dyadic function defined in \eqref{jthdyadic}.
\end{cor}

\subsection{Weighted norm for low frequency: $x(u^L,\vr^L)$
}

%Recall the definition:
%$$U^{L}=QW=\left(  \begin{array}{cc}1&1\\ -\f{\lm(D)}{\lnr}&-\f{\lp(D)}{\lnr}\\\end{array}\right)\left( \begin{array}{c}   w\\ \bar{w} \end{array}\right),$$  which yields:

 %-\R\big(\f{\epd}{\lnr}(w+\bar{w})+\f{ib(D)}{\lnr}(\bar{w}-w)\big)$
 By the Duhamel formula,
 \beq\label{formulae of w}
 w=e^{t\lambda_{-}(D)}w_0+\int_{0}^{t}e^{(t-s)\lambda_{-}(D)}\chi^{L}(D)\big[\f{\lambda_{+}(D)}{2ib}\lnr \f{\div}{|\na|}(\vr u)+\f{\lnr}{2ib}|\na||u|^2\big]\d s.
 \eeq
 Since
  $\vr^L=\f{|\na|}{\lnr}(w+\bar{w})$,
 $u^L=-\R\big(\f{\lambda_{-}(D)}{\lnr}w+\overline{\f{\lambda_{-}(D)}{\lnr}w}\big),$ we define the linear and  nonlinear flow for $\vr^L,u^L$:
%as in \cite{MR3274788}:
\beq\label{linear flow}
\mathcal{L}_{u^L}=\R \f{\lambda_{-}(D)}{\lnr}e^{it\lambda_{-}(D)}w_0\qquad \mathcal{L}_{\vr^L}=\f{|\na|}{\lnr} e^{it\lambda_{-}(D)}w_0
\eeq
\beno
\mathcal{N}_{u^L}&=&\int_{0}^{t}e^{i(t-s)\lambda_{-}(D)}\chi^{L}(D)\big[\R \f{\lambda_{-}\lambda_{+}(D)}{2ib} \f{\div}{|\na|}(\vr u)+\f{\lambda_{-}(D)}{2ib}\na|u|^2\big]\d s\\
&=&\int_{0}^{t}e^{i(t-s)\lambda_{-}(D)}\chi^{L}(D)\big[\R \f{\lnr^2}{2ib} \f{\div}{|\na|}(\vr u)+\f{\lambda_{-}(D)}{2ib}\na|u|^2\big]\d s\\
\mathcal{N}_{\vr^L}&=&\int_{0}^{t}e^{i(t-s)\lambda_{-}(D)}\chi^{L}(D)\big[\f{\lambda_{+}(D)}{2ib} \div(\vr u)-\f{\div}{2ib}\na|u|^2\big]\d s
\eeno
Inspired by \cite{MR3274788}, we need to prove the following claim:\\
\textbf{Claim}: For $j\geq 0$, one has:
\beq\label{claim}
\sup_{t\in[0,T_0]}\|\Phi_{j}(\cdot)(\lnr u^L, \lnr\vr^L)(\cdot)\|_{L_{x}^4}\lesim 2^{-j}
\eeq
We will postpone the proof of \eqref{claim} and first show the local boundedness of $\|xe^{itb(D)}w\|_{W^{4,\f{2}{1-\delta}}}$ and its continuity in time.
By \eqref{claim}, we have for any $j\in\mathbb{N}$,
 \beqs
 \sup_{t\in[0,T_0]}\|\Phi_j (u^L,\vr^L)\|_{W^{1,4}}\lesim 2^{-j}.
 \eeqs
which leads to, by Sobolev embedding, %and local boundedness of $(u^L,\vr^L)$ in $H^2$,
\beqs
\sup_{t\in[0,T_0]}\|x(u^L,\vr^L)\|_{L_{x}^{\infty}}\lesim 1.
\eeqs
Note that by \eqref{weighted high local}, we have also $\sup_{t\in[0,T_0]}\|x(u^h,\vr^h)(t)\|_{L_{x}^2}<+\infty.$

$\bullet$ Local boundedness of $\|xe^{itb(D)}w\|_{W^{4,\f{2}{1-\delta}}}$.
The boundedness of the linear term $xw_0$ stems from the assumption imposed upon the initial data. We thus focus on the boundedness of the nonlinear term. In light of \eqref{formulae of w}, it suffices for us to consider the typical term :
$$x\int_{0}^{t}e^{isb(D)}e^{\ep(t-s)\Delta}\lnr \R \chi^{L}(\vr u)\d s.$$
Nevertheless, since for arbitrary function $g$,
\ben\label{reduction}
&&\|x \R g\|_{W^{4,\f{2}{1-\delta}}}%\nonumber\\
\lesssim \||\na|^{-1}g\|_{W^{4,\f{2}{1-\delta}}}+\|xg\|_{W^{4,\f{2}{1-\delta}}}\lesssim \|g\|_{H^{4+\delta}}+\|xg\|_{H^{4+\delta}},
\een
it remains for us to control
$\|x\int_{0}^{t}e^{isb(D)}e^{\ep(t-s)\Delta}\lnr \chi^{L}(\vr u)\d s\|_{H^{4+\delta}}.$

To begin with, we write:
\beno
&&x\int_{0}^{t}e^{isb(D)}e^{\ep(t-s)\Delta}\lnr \chi^{L}(\vr u)\d s\\&=&
\int_{0}^{t}e^{isb(D)}e^{\ep(t-s)\Delta}(sb'(D)+2\ep(t-s)\na)\chi^{L}\lnr(\vr u)\d s\\
&&+\int_{0}^{t}e^{isb(D)}e^{\ep(t-s)\Delta}(\chi^{L}\langle\cdot\rangle)'(D)(\vr u)\d s+\int_{0}^{t}e^{isb(D)}e^{\ep(t-s)\Delta}\chi^{L}\lnr(x\vr u)\d s\\
&\define& \uppercase\expandafter{\romannumeral1}+\uppercase\expandafter{\romannumeral2}+\uppercase\expandafter{\romannumeral3}.
\eeno
On one hand, the first two terms can be controlled directly by:
\ben\label{local bound 1}
\|\uppercase\expandafter{\romannumeral1}+\uppercase\expandafter{\romannumeral2}\|_{H^{4+\delta}}\lesim \int_{0}^{t} \|\vr u\|_{H^{5+\delta}}\d s\lesim T_0 \sup_{t\in[0,T_0]}\|\vr\|_{H^{5+\delta}}\|u\|_{H^{5+\delta}}.
\een
On the other hand, by Lemma \ref{weighted product}:
\ben\label{local bound 2}
\|\uppercase\expandafter{\romannumeral3}\|_{H^{4+\delta}}&\lesim&\int_{0}^{t} \|x\vr u\|_{H^{5+\delta}}\d s\nonumber\\
&\lesim& T_0 \sup_{t\in[0,T_0]}(\|x(u^L,\vr^L)\|_{L^{\infty}}\|(u,\vr)\|_{H^{5+2\delta}}\\
&& \qquad +\|x(u^h,\vr^h)\|_{L^2}\|(u,\vr)\|_{H^{6+2\delta}}+\|(u,\vr)\|_{H^{5+2\delta}}^{2}).\nonumber
\een

$\bullet$ Continuity in time of the weighted norm.
By \eqref{reduction} again, it suffices for us to prove that:
$\|xe^{\ep t\Delta}\int_{0}^{t}e^{isb(D)}e^{-\ep s\Delta}\chi^{L}\lnr(\vr u)\d s\|_{H^{4+\delta}}$ is continuous in time. 

Denote $z(t)=\int_{0}^{t}e^{isb(D)}e^{-\ep s\Delta}\chi^{L}\lnr(\vr u)(s)\d s$.
Notice that $xe^{\ep t\Delta}z(t)=\ep t\na e^{\ep t\Delta}z(t)+e^{\ep t\Delta}(xz(t)) .$ Since $e^{\ep t\Delta}$ is a continuous operator in $H^{4+\delta}$, we reduce the problem to the continuity of
$\|x z(t)\|_{H^{4+\delta}},$ which is the consequence of the following: %We concentrate on the continuity of $\|xz(t)\|_{H^{4+\delta}}$, as the continuity of $\|z(t)\|_{H^{2+\delta}}$ is easy.It is direct to see that we only need to prove that
$$\sup_{t\in[0,T_0]}\|xe^{itb(D)}e^{-\ep s\Delta}\chi^{L}\lnr(\vr u)(t)\|_{H^{2+\delta}}< +\infty.$$
However, it has essentially been included in the proof of \eqref{local bound 1},\eqref{local bound 2}. One notes here that $e^{-\ep s\Delta}\chi^{L}$ is a $L^2$ multiplier whose norm is less than $e^{\kpz t}\leq e^{\kpz T_0}$.

\subsection{Proof of the claim %\eqref{claim}
}
We are now in position to prove \eqref{claim}.

$\bullet$ Linear flow estimate.
In light of \eqref{linear flow} and the
 crude approximation:
$$w_0=\f{\lambda_{+}}{2ib(D)}\f{\lnr}{|\na|}\vr_{0}^{L}+\f{\lnr}{2ib}\R\cdot u_{0}^L=\R(\f{\lnr}{|\na|}\vr_{0}^{L}+u_{0}^L)=\R (\lnr \R\cdot\na\varphi_0^{L}+u_{0}^L),$$ it suffices for us to show that for $\forall j\geq 0$,
\beq\label{linear flow local }
\|\Phi_{j}\R n_{1}(D)e^{it\lambda_{-}(D)}g\|_{L^4}\lesim 2^{-j}\|\langle x\rangle  g\|_{H^{\f{3}{2}}}
%(\|(u_{0}^{L},\f{\lnr}{|\na|}\vr_{0}^{L})\|_{H^{\f{3}{2}+\delta}}+\| x(u_{0}^{L},\f{\lnr}{|\na|}\vr_{0}^{L})\|_{W^{\f{3}{2}+\delta,\f{2}{1-\delta}}}).
\eeq 
where we denote $n_1(D)=\lambda_{-}(D)$ or $|\na|$, $g=u_{0}^L$ or $\lnr \na \varphi^L$.
%We will show for arbitrary function $g$,
%$$\|x\R n_1(D)e^{it\lambda_{-}(D)}\tilde{\chi}^{L}(D)g\|_{L^4}\lesim \|g\|_{H^{\f{3}{2}+\delta}}+\| x g\|_{W^{\f{3}{2}+\delta,\f{2}{1-\delta}}}.$$
Nevertheless, we have by Sobolev embedding, Hausdorff-Young inequality %and Lemma \ref{Lp bounds for eitb(D)} 
that,
\beno
&&\|x\R n_1(D)e^{it\lambda_{-}(D)}\tilde{\chi}^{L}(D)g\|_{L^4}\\
%\lesim \|\langle x\rangle n_1(D)e^{it\lambda_{-}(D)}\tilde{\chi}^{L}(D)g\|_{ H^{\f{1}{2}}%W^{\f{1}{2},\f{2}{1-\delta}}}\\
&\lesim&\|n_1(D)e^{it\lambda_{-}(D)}\tilde{\chi}^{L}(D)g\|_{L^2}+\|xn_1(D)e^{it\lambda_{-}(D)}\tilde{\chi}^{L}(D)g\|_{H^{\f{1}{2}}}%\f{2}{1-\delta}}}
\\
&\lesim&\|g\|_{H^{1}}+\|n_1(D)e^{it\lambda_{-}(D)}\tilde{\chi}^{L}(D)xg\|_{H^{\f{1}{2}}}%,\f{2}{1-\delta}}}
+\|[x,n_1(D)e^{it\lambda_{-}(D)}\tilde{\chi}^{L}(D)]g\|_{H^{\f{1}{2}}}%,\f{2}{1-\delta}}}
\\
&\lesim& \|g\|_{H^{1}}+\|xg\|_{H^{\f{3}{2}}}+(1+t)\|g\|_{H^{\f{1}{2}}}\lesim \ltr\|\langle x\rangle g\|_{H^{\f{3}{2}}}.
\eeno

%\begin{rmk}
%One can easily see that we actually could replace the norm in the right hand side of \ref{linear flow local } by $\|\langle x \rangle(u_{0}^{L},\f{\lnr}{|\na|}\vr_{0}^{L})\|_{H^{\f{3}{2}}}$. However, we prefer to involve the weighted $L^{\f{2}{1-\delta}}$ (with $\f{2}{1-\delta}$ just slightly larger than 2 ) norm because that: we want to impose the weighted $L^2$ boundedness assumption upon the initial data $(u_{0},\vr_{0},\na\phi_0)$ which corresponds to its $L^1$ norm.
%In fact, one have:
%$\|x (u_{0}^{L},\f{\lnr}{|\na|}\vr_{0}^{L})\|_{W^{\f{3}{2}+\delta,\f{2}{1-\delta}}}\lesim\|x(u_{0}^{L},\vr_{0}^{L},\na\phi_{0}^{L})\|_{H^{\f{3}{2}+2\delta}}$.
%\end{rmk}

$\bullet$ Nonlinear flow estimate.

For notational brevity, we replace again $\chi^{L}\f{\lambda_{\pm}}{2ib},\chi^{L}\f{\lnr}{2ib},\chi^{L}\f{\lnr}{2ib},\f{|\na|}{\lnr}$ by $\R$ since they have the similar properties.
Therefore, it remains for us to show: for any $j\in \mathbb{Z}$
\beq\label{nlflow}
\|\Phi_j(x) \int_{0}^{t}e^{i(t-s)\lambda_{-}(D)}\chi^{L}(D)\R \lnr^{2}(\vr u+|u|^2)\d s\|_{L^4}\lesim _{\ep,\kpz,T_0}2^{-j}
\eeq
By Corollary \ref{useful for the local existence in weighted norm}:
\ben\label{nonlinear ineq1}
&&\|\Phi_j(x) \int_{0}^{t}e^{i(t-s)\lambda_{-}(D)}\chi^{L}(D)\R \lnr^{2}(\vr u+|u|^2)\d s\|_{L^4}\nonumber\\
&\lesim _{\ep,\kpz,T_0}&2^{-j}+\| \int_{0}^{t}e^{i(t-s)\lambda_{-}(D)}\lnr^{2}\chi^{L}(D) \big(\Phi_j(x)\R %\tilde{\chi}^{L}
(\vr u+|u|^2)\big)\d s\|_{L^4}\nonumber\\
&\lesim _{\ep,\kpz,T_0}&2^{-j}+\| \int_{0}^{t}e^{i(t-s)\lambda_{-}(D)}\lnr^{\f{5}{2}}\chi^{L}(D) \big(\Phi_j(x)\R %\tilde{\chi}^{L}
(\vr u+|u|^2)\big)\d s\|_{L^2}\nonumber\\
&\lesim _{\ep,\kpz,T_0}&2^{-j}+T_0(1+\f{\kpz}{\ep})^{\f{5}{4}}\sup_{t\in[0,T_0]}\|\Phi_j(x)\R %\tilde{\chi}^{L}
(\vr u+|u|^2)(t)\|_{L^2}
\een
To prove \eqref{nlflow}, we are left to establish that:
\beq\label{nonlinear ineq2}
\sup_{t\in[0,T_0]}\|\Phi_j(x)\R %\tilde{\chi}^{L} 
(\Phi_{\geq {j+2}}+\Phi_{\leq {j-2}})(\vr u+|u|^2)\|_{L^2}\lesim 2^{-j}
\eeq
We show for example the case of $\Phi_{\geq {j+2}}$ as the other is similar. Denote
%For $|x|\geq 2^{10}$
$\hdb_{k}\R\tilde{\chi}^{L}g=G_{k}\star g$,
where
$$G_k(x)=\int e^{ix\cdot\xi}\Phi_{k}(\xi)%\tilde{\chi}^{L}(\xi)
\f{\xi}{|\xi|}\d \xi=2^{2k}\int e^{i2^{k}x\cdot\xi}\Phi(\xi)%\tilde{\chi}^{L}(2^{k}\xi)
\f{\xi}{|\xi|}\d \xi$$
%As for $|x|>0$, $|\na_{\xi}e^{i2^{k}x\cdot\xi}|=2^{k}|x|>0$, thus by classical stationary phase method, that is using the fact
By virtue of the identity $e^{i2^{k}x\cdot\xi}=\f{1-i2^{k}x\cdot\na_{\xi} }{\langle 2^{k}x\rangle^2}e^{i2^{k}x\cdot\xi}$, one can integrate by parts to get that,
 for any non-negative integer $l$,
\beqs
|G_k(x)|\lesim 2^{2k}\langle 2^k |x|\rangle^{-l}.
\eeqs
Note that there is no singularity when the derivative hits on $\f{\xi}{|\xi|}$ since $\Phi(\xi)$ is supported on annulus.
%On the other hand, it is obvious that the above inequality is also true for $x=0$. Thus it holds for any $x\in \mathbb{R}^{2}.$
Therefore, we have that
\beno
&&\|\Phi_j(x)\R %\tilde{\chi}^{L}
\Phi_{\geq {j+2}}(\vr u+|u|^2)\|_{L^2}\lesim \sum_{k\in \mathbb{Z}}\|\Phi_j(x)\hdb_{k}\R\Phi_{\geq {j+2}}(\vr u+|u|^2)\|_{L^2}\\
&\lesim& \sum_{k\in \mathbb{Z}}\|G_{k}I_{|\cdot|\gtrsim 2^{j}}\star \Phi_{\geq {j+2}}(\vr u+|u|^2)\|_{L^2}
\lesim \sum_{k\in \mathbb{Z}}\|G_{k}I_{|\cdot|\gtrsim 2^{j}}\|_{L^2}\|\vr u+|u|^2\|_{L^1}\\
&\lesim& \sum_{k\in \mathbb{Z}}\langle 2^k 2^{j}\rangle^{-3}2^{k}\|(\vr,u)\|_{L^2}^{2}\lesim 2^{-j}\|(\vr,u)\|_{L^2}^{2}.
\eeno
where the following two inequalities has been used:
\beqs
\|G_{k}I_{|\cdot|\gtrsim 2^{j}}\|_{L^2}\lesim 2^{2k}(\int_{|x|\gtrsim 2^j} \langle 2^{k}|x|\rangle^{-8}\d x)^{\f{1}{2}}\lesim \langle 2^k 2^{j}\rangle^{-3}2^{k},
\eeqs
\beqs
\sum_{k\in \mathbb{Z}}\langle 2^k 2^{j}\rangle^{-3}2^{k}\lesim(\sum_{k\leq -j}+\sum_{k\geq -j})\langle 2^k 2^{j}\rangle^{-3}2^{k}\lesim 2^{-j}.
\eeqs

Finally, denote $\tilde{\Phi}_j=\Phi_{j-1}+\Phi_{j}+\Phi_{j+1}$, Corollary \ref{weighted high by weighted low} implies :
\ben\label{nonlinear ineq3}
\|\Phi_j \R \tilde{\Phi}_j(\vr u+|u|^2)\|_{L^2}&\lesim& \| \tilde{\Phi}_j(\vr u+|u|^2)\|_{L^2}\nonumber\\
&\lesim&\|\tilde{\Phi}_j(u^L,\vr^L)\|_{L^{4}}\|(u,\vr)\|_{L^4}+\|\tilde{\Phi}_j(u^h,\vr^h)\|_{L^2}\|(u,\vr)\|_{L^{\infty}}\nonumber\\
&\lesim_{\ep,\kpz}&(1+T_0)(2^{-j}+\|\tilde{\Phi}_j(u^L,\vr^L)\|_{L^{4}}).
\een
Define %sequence 
$y_j=\|\Phi_j(x)(\lnr u^L,\lnr \vr^{L})\|_{L^4}$,
by \eqref{linear flow local }, \eqref{nonlinear ineq1}-\eqref{nonlinear ineq3}, one gets that, as long as $T_0$ is chosen small enough,
$$y_j\lesim 2^{-j}+\f{1}{16}(y_{j-1}+y_j+y_{j+1}),$$
which, combined with the fact $y_j\lesim 1$ ($j\geq 0$) and iteration arguments, yields that
$y_j\lesim 2^{-j}$. We thus finish the proof of \eqref{claim}.
%\section{A priori estimate}

\section{Weighted $L^2$ norm for high frequency: a priori estimate}

 The goal of this section is to get the weighted $L^2$ estimate for $(\vr,u,\na \varphi)^{h}$ which shall be used in Section 7.
\begin{lem}
There exists a constant $\vartheta_2>0,$ for any $\varepsilon\in (0,1]$, if $\|U\|_{X_T}\leq \vartheta_2$, we have the a priori estimate:
\beq\label{weighted high frequency}
\|\langle x\rangle\big( \vr,u,\na \varphi)^{h}(t)\|_{L^2}\lesim \ltr^{2\delta}(\|\langle x\rangle(\vr_0,u_0,\na\varphi_0)^{h}\|+\|U\|_{X_T}\big), \quad \forall      t\in [0,T)
\eeq
\end{lem}
\begin{proof}

We first claim that the following estimate holds,
\beq
\|x(\vr,u,\na\varphi)^{L}(t)\|_{W^{1,4}}\lesim \ltr^{\f{1}{2}+\delta}\|U\|_{X_T}.
\eeq
Indeed, %after defining the profile of $U^L$: $f=e^{itb(D)}U^{L}$, we have 
by virtue of the Hardy-Littlewood-Sobolev inequality, %the definition of $X_T$ norm and 
the $L^p$ boundedness of $e^{itb(D)}\tilde{\chi}^{L}$
(Lemma \ref{Lp bounds for eitb(D)}),
\beno
\|x(\vr,u,\na\varphi)^{L}\|_{W^{1,4}}&=&\|x \R \tilde{\chi}^{L} w\|_{W^{1,4}}
\approx\||\na|^{-1}w\|_{W^{1,4}}+\|xw\|_{W^{1,4}}
\lesim\|\langle x\rangle w \|_{W^{1+\delta,4-\delta}}\\
&\lesssim&\|w\|_{H^2}+\|xe^{itb(D)}f\|_{W^{1+\delta,4-\delta}}
\lesim\|U\|_{X}+t\|w\|_{W^{1+\delta,4-\delta}}+\|e^{itb(D)}\tilde{\chi}^{L}xf\|_{W^{1+\delta,4-\delta}}\\
&\lesim&\|U\|_{X_T}+t^{1-(1-\f{2}{4-\delta})}\|U\|_{X_T}+\|e^{itb(D)}\tilde{\chi}^{L}xf\|_{W^{\f{3}{2}+\delta,2+\delta}}\lesim\ltr^{\f{1}{2}+\delta}\|U\|_{X_T}.
\eeno
Note that $f=e^{-itb(D)}w.$
Hereafter, unless specifically emphasized, the spatial Sobolev norm is estimated in each temporary time $t\in [0,T).$\\
The system projected onto the high frequency reads:
 \beq \label{NSPhighfrequency1}
 \left\{
\begin{array}{l}
\displaystyle \pt \vr^{h} +\div u^{h}+\div( \vr u)^{h}=0,\\
\displaystyle \pt u^{h}+(u \cdot {\na u})^{h}-2\varepsilon \Delta u^{h}+
\nabla \vr^{h}-\nabla \varphi^{h}=0 ,  \\
\displaystyle \Delta \varphi^{h} =\vr^{h},\\
\displaystyle u|_{t=0} =\mathcal{P}^{\perp}u_0^{\varepsilon} ,\vr|_{t=0}=\vr_0^{\varepsilon}.
\end{array}
\right.
\eeq
Multiplying the system \eqref{NSPhighfrequency1} by $x$, and testing it by $x(\vr^h,u^h)$, we obtain the energy equality:
\beno
&&\f{1}{2}\f{\d}{\d t}\|x(\vr^h,u^h,\na \varphi^{h})\|_{L^2}^{2}+2\ep \int|x\na u^{h}|^{2}\d x\\
&=&-\int x^{2}(\div u^{h}\vr^{h}-\na \vr^{h} u^{h})\d x+\int x^{2}((\na\varphi)^{h} u^{h}-\na\varphi^{h}\na (\Delta)^{-1}\div u^{h})\d x -4\ep \int \na u^{h}xu^{h} \d x\\
&&-\int x\div(\vr u)^{h}x\vr^{h}\d x-\int x(u\cdot\na u)^{h}x u^{h}\d x-\int x\na\varphi^{h}x\na(\Delta)^{-1}\div(\vr u)^{h}\d x\\
&\define&G_1+G_2+\cdots G_6.
\eeno
The following task is to estimate $G_1,\cdots G_6.$
At first, by integration by parts
and H\"{o}lder inequality, %and the fact that $\na \psi_{R}$ is bounded, (uniformly in $R$) we have:
$G_1$ can be controlled as:
\beqs
G_1=-2\int x\vr^{h}u^{h}\d x\lesim \|x\vr^{h}\|_{L^2}\|u^{h}\|_{L^2}.
\eeqs
Similarly, by using the fact $\|\varphi^h\|_{L^2}\lesssim \sqrt{\f{\ep}{\kpz}}\|\na\varphi^h\|_{L^2},$  %$G_2$ can be estimated as:
\beno
G_2&=&2\int \varphi^{h}x(\na(\Delta)^{-1}\div -1)u^{h}\d x \\
&\lesssim& \|\varphi^h\|_{L^2} (\|xu^h\|_{L^2}+\|[x,\na(\Delta)^{-1}\div\tilde{\chi}^{h}]u^h\|_{L^2})\\
&\lesim& %\|\na \psi_{R}\|_{L^{\infty}}
\|\varphi^h\|_{L^2}(\|xu^{h}\|_{L^2}+%\|\cF^{-1}(\na\psi_R)\|_{L^1}
\|\partial_{\xi}
(\f{\xi_{i}\xi_{j}}{|\xi|^{2}}\tilde{\chi}^{h})\|_{L_{\xi}^{\infty}}\|u^{h}\|_{L^2})\\%\lesim %\f{\ep}{\kpz}\|\vr^{h}\|_{L^2})\\
&\lesim&(\sqrt{\f{\ep}{\kpz}}+\f{\ep}{\kpz})\|\na \varphi^{h}\|_{L^2}(\|u^{h}\|_{L^{2}}+\|xu^{h}\|_{L^2}).
\eeno
%where we have used $$\psi_{R}\na\phi^{h}=\psi_{R}\na(\Delta)^{-1}\tilde{\chi}^{h}\vr^{h}=[\psi_{R},\na (\Delta)^{-1}\tilde{\chi}^{h}]\vr^{h}+\na (\Delta)^{-1}\tilde{\chi}^{h} \psi_{R}\vr^{h}$$.
%Noticing that:
%\beno
%\|[\psi_{R},\na (\Delta)^{-1}\tilde{\chi}^{h}]\vr^{h}\|_{L^2}&\lesssim& \|\cF^{-1}(\na\psi_R)\|_{L^1}\|\partial_{\xi}(\f{\xi}{|\xi|^{2}}\tilde{\chi}^{h})\|_{L_{\xi}^{\infty}}\|vr^{h}\|_{L^2}\lesim \f{\ep}{\kpz}\|\vr^{h}\|_{L^2}\\
%\|\na (\Delta)^{-1}\tilde{\chi}^{h} \psi_{R}\vr^{h}\|_{L^2}&\lesssim& \sqrt{\f{\ep}{\kpz}}\|\psi_{R}\vr^{h}\|_{L^2}
%\eeno
For $G_3$, we just use H\"{o}lder inequality: $G_3\leq 2\ep \|\na u^{h}\|_{L^2}\|xu^h\|_{L^2}$.\\
 By setting $P(f,g)=f^{h}g+f^{L}g^{h}+f^{L}g^{m}+f^{m}g^{l}$, we estimate $G_4$ as follows:
\beno
G_4&=&-\int x\div(\vr u)^{h} x\vr^{h}\d x
=-\int ([x,\chi^{h}]\div(\vr u)+\chi^{h}x\big(P(\na\vr, u)+P(\vr ,\div u))\big)x\vr^{h}\d x\\
&\lesim& \|x\vr^{h}\|_{L^{2}}(\sqrt{\f{\ep}{\kpz}}\|\div(\vr u)^{h}\|_{L^2}+\|x\big(P(\na\vr, u)+P(\vr ,\div u)\big)\|_{L^2})\\
&\lesim& \|x\vr^{h}\|_{L^{2}}\sqrt{\f{\ep}{\kpz}}\|(\vr,u)\|_{L^2}\|\na(\vr,u)\|_{L^{\infty}}
+\|x\vr^{h}\|_{L^{2}}\big(\ltr^{-(1-\delta)} \|U\|_{X_T}^{2}+\|x(\vr^{h},u^{h})\|_{L^2}\ltr^{-1}\|U\|_{X_T}\big)\\
&\lesim&\ltr^{-1}\|U\|_{X_T}(\|x\vr^{h}\|_{L^{2}}^{2}+\|x u^{h}\|_{L^{2}}^{2})+\ltr^{-(1-\delta)}\|U\|_{X_T}^{2}\|x\vr^{h}\|_{L^{2}}.
\eeno
where we have used that:
\beno
\|xP(\na\vr, u)\|_{L^{2}}+\|xP(\vr ,\div u)\|_{L^2}\lesim (1+t)^{-1}\|x(\vr,u)^{h}\|_{L^2}+(1+t)^{-(1-\delta)}\|U\|_{X_T}^{2}.
\eeno
For example, since $\|(u^{m},\vr^m)\|_{W^{1,4}}\lesim \ltr^{-\f{3}{2}}\|U\|_{X_T},\|u^{h}\|_{H^M}\lesim \ltr^{-\alpha}\|U\|_{X_T},(\alpha=2-5\delta)$, one has that
\beno
\|xP(\na\vr, u)\|_{L^{2}}&\lesim& \|xu^{h}\|_{L^2}\|\na\vr^{h}\|_{L^{\infty}}+\|xu^{L}\|_{L^4}\|\na\vr^{h}\|_{L^4}
+\|x\na\vr^L\|_{L^4}\|u^{h}+u^m\|_{L^4}+\|xu^{L}\|_{L^4}\|\na\vr^{m}\|_{L^4}\\
&\lesim& (1+t)^{-1}\|x(\vr,u)^{h}\|_{L^2}+(1+t)^{-(1-\delta)}\|U\|_{X_T}^{2}.
\eeno
The estimate of $G_5,G_6$ is similar, we thus omit the details.
To summarize, we have obtained:
\ben \label{xu^{h}}
&&\f{\d}{\d t}(\|x(\vr,u,\na\varphi)^{h}\|_{L^2}^{2})\nonumber\\
&\lesim& (1+t)^{-1}\|U\|_{X_T}\|x(\vr,u,\na\varphi)^{h}\|_{L^2}^{2}
+(1+t)^{-(1-\delta)}\|U\|_{X_T}^{2}\|x(\vr,u,\na\varphi)^{h}\|_{L^2}\nonumber\\
&&+(1+t)^{-\alpha}\|U\|_{X_T}\|\langle x\rangle(\vr,u,\na\varphi)^{h}\|_{L^2}.
\een
In the same fashion, one can show that:
\ben \label{u^{h}}
\f{\d}{\d t}(\|(\vr,u,\na\varphi)^{h}\|_{L^2}^{2})&\lesim&(1+t)^{-1}\|U\|_{X_T}\|(\vr,u,\na\varphi)^{h}\|_{L^2}^{2}
\een
Finally, we set $g(t)=\|\langle x\rangle(\vr,u,\na\varphi)^{h}\|_{L^2}$. Summing up the above two estimates \eqref{xu^{h}}-\eqref{u^{h}}, we see that there exists three constants $c_1,c_2,c_3$ which are independent of $\ep$, such that
\beno
\f{\d}{\d t}g(t)\leq c_1 (1+t)^{-1}\|U\|_{X_T}g(t)+c_2(1+t)^{-(1-\delta)}\|U\|_{X_T}^{2}+c_3(1+t)^{-\alpha}\|U\|_{X_T}.
\eeno
Suppose that $\|U\|_{X_T}\leq \vartheta_2\leq  \frac{\delta}{c_1}$, then for any $0\leq t<T,$ the Gr\"{o}nwall inequality leads to:
\beqs%\label{weighted high frequency}
g(t)\lesim \ltr^{\delta}g(0)+\ltr^{2\delta}\|U\|_{X_T}^{2}+\|U\|_{X_T}\lesim \ltr^{2\delta}(\|\langle x\rangle(\vr_0,u_0,\na\varphi_0)^{h}\|_{L^2}+\|U\|_{X_T}).
\eeqs
\end{proof}

\section{Estimate of Sobolev norm}
In this section, we aim to get the highest Sobolev estimate for $U$: $\|U\|_{H^N}$ and Sobolev estimate for high and intermediate frequencies: $\|U^h\|_{H^{N-2}}$, $\|U^m\|_{H^{N-1}}$ and $\|U^m\|_{W^{1,4}}$.
\subsection {Control of highest Sobolev norms}
Define the energy norm
$$E_N(t)=\sum_{|\alpha|\leq N}\f{1}{2}\int \rho |\pa u(t)|^{2}+|\pa \vr(t)|^{2}+|\pa \na\phi(t)|^{2}\d x.$$
 In our former paper \cite{rousset2019stability} where the 3d NSP is considered, it has been showed that if $\|\vr\|_{L^{\infty}}\lesssim \f{1}{6}$, then the following energy inequality holds:
$$\f{\d }{\d t}E_N(t) \lesssim (\| u(t)\|_{W^{1,\infty}}+\|\varrho(t)\|_{L^{\infty}})E_N(t).$$ However, such an inequality is not enough to close the energy estimate in 2d case. Indeed, due to the presence of Riesz potential in the quadratic nonlinearity (see \eqref{eqU}), one could only expect that $\|\na u(t)\|_{L^{\infty}}$ rather than $\|u(t)\|_{W^{1,\infty}}$ has the critical decay
$(1+t)^{-1}$ .
Nevertheless, it is not hard to modify the proof in \cite{rousset2019stability} to get that:
\ben\label{energyeq}
\f{\d }{\d t}E_N(t) \lesssim (\|\na u(t)\|_{L^{\infty}}+\|\varrho(t)\|_{L^{\infty}})E_N(t).\een
Indeed, denote $E_{\alpha}=\f{1}{2}\int \rho |\pa u|^{2}+|\pa \rho|^{2}+|\pa \na\phi|^{2}\d x$, we then have by using the equations \eqref{ANSP1} that:
\beno
\frac{\d}{\d t}E_{\alpha}
&=&\int\rho \pa u\cdot \big[\pa,u\big]\na u+\int\pa \rho \div\big([\pa,u ]\rho\big)\\
&&+\int[\rho \pa u\pa\na\varphi+\pa\varphi\pa\div(\rho u)]+2\ep \int
 \rho \pa u\cdot\pa \Delta u \\
&\define&L_1+L_2+L_3+L_4
\eeno
One can estimate all the terms in the same way as that in \cite{rousset2019stability} except the term  $L_3$. However, for any  $|\alpha|\geq 1$, it can be rewritten  as 
\beno
L_3&=&\int\vr \pa u\pa \na \varphi+\pa \varphi \pa\div(\vr u)\d x\\
&=&\int\vr \pa u\pa \na \varphi+\pa(\vr \div u)\pa\varphi+[\pa,u]\na\vr\pa\varphi+\pa\na\vr\cdot u\pa\na\varphi\d x\\
&=&\int\pa\varphi\big([\pa,\vr]\div u+\na\vr\cdot\pa u\big)+[\pa,u]\na\vr\pa\varphi-\pa\vr\div u\pa\varphi-\f{\pa |\na\varphi|^2}{2}\div u \\
&&\qquad\qquad\qquad\qquad\qquad \qquad\qquad\qquad+\pa\na\varphi\cdot\na u\cdot \pa\na\varphi \d x
\eeno

Notice that in the above expressions, there is at least one spatial derivative
in front of $u$, we thus conclude by standard commutator estimate that:
$$|L_3|\leq (\|\na u\|_{L^{\infty}}+\|\vr\|_{W^{1,\infty}})(\|u\|_{\dot{H}^{|\alpha|}}^2+\|\vr\|_{\dot{H}^{|\alpha|}}^2+\|\na\varphi\|_{\dot{H}^{|\alpha|-1}\cap \dot{H}^{|\alpha|}}^2),$$
which ends the proof \eqref{energyeq}.

It is easy to see that $\|U\|_{H^{N}}^2\approx E_{N}$ by noting the relation $u=\R \mathrm{c},\vr=\f{|\na|}{\lnr}\mathrm{a}.$
This, combined with \eqref{energyeq} and the definition of $X_T$ norm \eqref{def of real norm}, yields:
\beno
\|U\|_{H^N}^{2}\lesssim E_{N}(t)
&\lesssim&E_N(0)+\int_{0}^{t}(\|\na u\|_{L^{\infty}}+\|\vr\|_{L^{\infty}})\|U(s)\|_{H^N}^{2}\d s\\
&\lesssim& E_N(0)+\int_{0}^{t} (1+s)^{-1}(1+s)^{2\delta}\|U\|_{X_T}^{2}\d s
\lesssim E_N(0)+(1+t)^{2\delta}\|U\|_{X_T}^{2}.
\eeno

%We thus have: $\|U^L\|_{H^N}\lesim E_N(0)\|+(1+t)^{2\delta}\|U\|_{X}^{2}$.

\subsection{High and intermediate frequency estimate}
%The high and intermediate frequency estimates will be much easier once 
We have firstly the following estimate for nonlinear term $F(a,c)=F(U,U)$.
\begin{lem}\label{nonlinear estimate}
For every $t\in[0,T),$ the following estimate holds:
\beq\label{nonlinear estimate1}
 \|\chi^{h}F(U,U)(t)\|_{H^{N-2}}\lesssim (1+t)^{-(2-5\delta)}\|U\|_{X}^{2},
 \eeq
 \beq \label{nonlinear estimate2}
 \|F(U,U)(t)\|_{H^{N-1}}\lesssim (1+t)^{-(1-3\delta)}\|U\|_{X_T}^{2}.
 \eeq
\end{lem}
\begin{proof}{}
We begin with the proof of \eqref{nonlinear estimate1}. By the definition of truncation functions, one has $\chi^{l}(\xi-\eta)\chi^{l}(\eta)\chi^{h}(\xi)=0$, which leads to the decomposition:
\beqs
\chi^{h}(D)F(U,U)=\chi^{h}(D)\big(F(U^L,U^h)+F(U^h,U)+F(U^l,U^m)+F(U^m,U^l)+F(U^{m},U^{m})\big).
\eeqs
We only detail the estimation of $F(U^l,U^m)$, the other terms are much easier.
By the definition, one has that
 $F(U,U)\approx\lnr \R (\R U\cdot\R U)$. Therefore, owing to the tame estimate, Sobolev embedding and the definition of $X_T$ norm,
\begin{equation}\label{low-int}
\begin{aligned}
\|F(U^{l},U^m)\|_{H^{N-2}}&\lesssim \|\R U^{l}\cdot\R U^{m}\|_{H^{N-1}}\lesssim \|\R U^l\|_{L^{\infty}}\|\R U^m\|_{H^{N-1}}\\
&\lesssim\|\R U^l\|_{W^{1,\f{1}{\delta}}}\|U^m\|_{H^{N-1}}\lesssim (1+t)^{-(2-5\delta)}\|U\|_{X_T}^{2}.  
\end{aligned}
\end{equation}
We next show \eqref{nonlinear estimate2}, by splitting $F(U,U)$ into:
\beqs
F(U,U)=F(U^{L},U^L)+F(U^L,U^h)+F(U^h,U).
\eeqs
Similar to \eqref{low-int},  $F(U^L,U^L)$ can be controlled as: 
\beno
\|F(U^L,U^L)\|_{H^{N-1}}\lesssim \|(\R U^L)^2\|_{H^{N}}
\lesssim\|\R U^L\|_{W^{1,\f{1}{\delta}}}\|\R U^L\|_{H^{N}}\lesssim(1+t)^{-(1-3\delta)}\|U\|_{X_T}^{2}.
\eeno
%where we have used Lemma \ref{high regularity decay} and condition $M+2<\f{1}{20}(1+\delta)+N'\f{19}{20}$.
The other two terms are easier, we omit the detail.
\end{proof}
\subsubsection{High frequency estimate: control of $\|U^h\|_{H^{N-2}}$}
 %We now start to estimate the high frequency. 
 
 By Duhamel's formula:
 $$U^{h}(t)=e^{-tA}U_{0}^{h}+\int_{0}^{t}e^{-(t-s)A}\chi^{h}F(U,U)(s)\d s,$$
 Lemma \ref{high frequency estiamte} and Lemma \ref{nonlinear estimate}  then imply:
\beno
\|U^{h}\|_{H^{N-2}}&\lesssim& e^{-ct}\|U_0\|_{H^{N-2}}+\int_{0}^{t}e^{-c(t-s)}\|{\chi}^{h}(F(U,U))\|_{H^{N-2}}\d s\\
&\lesssim& e^{-ct}\|U_0\|_{H^{N-2}}+\int_{0}^{t}e^{-c(t-s)}(1+s)^{-\alpha}\|U\|_{X_T}^{2}\d s\\
&\lesssim&e^{-ct}\|U_0\|_{H^{N-2}}+(1+t)^{-\alpha}\|U\|_{X_T}^{2}.
\eeno

\subsubsection{Intermediate frequency estimate: control of $\|U^{m}\|_{H^{N-1}}$ and $\|U^{m}\|_{W^{1,4}}$ }

 By Duhamel's formula, 
 Lemma \ref{nonlinear estimate}, and Lemma \ref{high frequency estiamte}, one can control the Sobolev norm of intermediate frequency as follows
\beno
\|U^{m}\|_{H^{N-1}}&\lesssim&
e^{-ct}\|U_{0}^{m}\|_{H^{N-1}}+\izt e^{-c(t-s)}\|F(U,U)(s)\|_{H^{N-1}}\d s
\\
&\lesssim&e^{-ct}\|U_0\|_{H^{N-1}}+\int_{0}^{t}e^{-c(t-s)}(1+s)^{-(1-3\delta)}\|U\|_{X_T}^{2}\d s\\
&\lesssim&e^{-ct}\|U_0\|_{H^{N-1}}+(1+t)^{-(1-3\delta)}\|U\|_{X_T}^{2}.
\eeno

We can  estimate $\|U^m\|_{W^{1,4}}$ in the same fashion. In fact, by Corollary \ref{cor immediate norm } (we will use $(\f{1}{2})_{+}=\f{3}{4}$), Lemma \ref{high regularity decay}(we use relation $\f{11}{4}\leq \f{2}{5}\cdot 7$), the definition of $X_T$ norm, we get:
\beno
\|U^m\|_{W^{1,4}}&\lesssim& e^{-ct}\|U_{0}^m\|_{W^{\f{7}{4},4}}+\int_{0}^{t}e^{-c(t-s)}\|F(U,U)\|_{W^{\f{7}{4},4}}\d s\\
%&\lesssim& e^{-ct}\|U_{0}^m\|_{W^{2,4}}+\int_{0}^{t}e^{-c(t-s)}\|(\R U)^{2}\|_{W^{3,4}}\d s\\
&\lesssim& e^{-ct}\|U_{0}^m\|_{W^{\f{7}{4},4}}+\int_{0}^{t}e^{-c(t-s)}\|U\|_{W^{\f{11}{4},5}}\|U\|_{L^{20}}\d s\\
&\lesssim&e^{-ct}\|U_{0}^m\|_{H^{\f{9}{4}}}+\int_{0}^{t}e^{-c(t-s)}(1+s)^{-\f{3}{2}}\|U\|_{X_T}^{2}\d s\\
&\lesssim& e^{-ct}\|U_0\|_{H^{\f{9}{4}}}+(1+t)^{-\f{3}{2}}\|U\|_{X_T}^{2}.
\eeno

%we thus get:
%$$\sup_{t}\ltr^{2\delta}\lesim \ltr^{-2\delta}g(0)+\|U\|_{X}^{\f{3}{2}}$$
\section{Low frequency estimate}
%Define the profile of $w$: $f=e^{itb(D)}w$, then by equation \eqref{eq of w}, we know that $f$ satisfy the equation:
%\beq
%e^{-itb(D)}\pt f=\ep \Delta w+\R(B(w,w)+\lnr\chi^{L}H)
%\eeq
In this section, we focus on the a priori estimate of Low frequency: $\||\na|^{\f{1}{2}}\lnr Q^{-1}U^L\|_{L^{\infty}}$, $\|xe^{itb(D)}w\|_{W^{4,\f{2}{1-\delta}}}$, $\|U^{L}\|_{H^{N'}}.$ In practice, we shall perform the decay estimate and weighted estimate in the same time.

By equation \eqref{eq of w} and Duhamel principle:
\ben\label{formula of w}
w&=&e^{t\lambda_{-}(D)}w_0+\R\int_{0}^{t}e^{(t-s)\lambda_{-}(D)}(B(w,w)+n(D)\chi^{L}H)(s)\d s\nonumber\\
&\define& K_1+\R \big(K_2+K_3\big)
\een
 To close the decay estimate for $\R K_2$,  the 'space-time resonance' philosophy that change the quadratic nonlinearity to the cubic one needs to be enforced. More specifically, we rewrite (2) in the following fashion. Recall the definition of the phase function $\phi_{\mu,\nu}(\xi,\eta)=i(b(\xi)-\mu b(\xi-\eta)-\nu b(\eta))+\ep(|\xi|^2-|\xi-\eta|^2-|\eta|^2)$ 
 $(\mu,\nu\in\{+,-\})$ and the bilinear operator $T_{m}$ in \eqref{defbilinear}. %By using the fact that $|\phi_{\mu,\nu}|>0  $ on the support of $m_{\mu,\nu}$, we have identity . 
Denote  $\tilde{f}=e^{-t\lambda_{-}(D)}w,$ then $f$ is governed by 
\beqs
e^{t\lambda_{-}(D)}\pt \tilde{f}=\R^2(B(w,w)+\lnr\chi^{L}H).
\eeqs
One thus has by identity $e^{s\phi_{\mu,\nu}}=\f{1}{\phi_{\mu,\nu}}\partial_{s}e^{is\phi_{\mu,\nu}}$ and integration by parts in time that:
\ben \label{quadratic low term}
&&\int_{0}^{t}e^{(t-s)\lambda_{-}(D)}B(w,w)\d s\nonumber\\
&=&e^{t\lambda_{-}(D)}\sum_{\mu,\nu\in\{+,-\}}\cF^{-1}\big(\int_{0}^{t}\int
e^{-s\lambda_{-}(D)} m_{\mu\nu}(\xi,\eta)\widehat{\R w^{\mu}}(s,\xi-\eta)\widehat{\R w^{\nu}}(s,\eta)\d \eta \d s  \big)\nonumber\\
&=&e^{t\lambda_{-}(D)}\sum_{\mu,\nu\in\{+,-\}}\cF^{-1}\big(\int_{0}^{t}\int
e^{s\phi_{\mu,\nu}} m_{\mu\nu}(\xi,\eta)\widehat{\R \tilde{f}^{\mu}}(s,\xi-\eta)\widehat{\R \tilde{f}^{\nu}}(s,\eta)\d \eta \d s  \big)\nonumber
\een
\ben
&=&\sum_{\mu,\nu\in\{+,-\}} \big[T_{\f{m_{\mu\nu}}{\phi_{\mu\nu}}}(\R w^{\mu}(t),\R w^{\nu}(t))-e^{t\lambda_{-}(D)}e^{\ep t\Delta}[T_{\f{m_{\mu\nu}}{\phi_{\mu\nu}}}(\R w^{\mu}_0,\R w^{\nu}_0)\nonumber\\
&&-\int e^{-i(t-s)b(D)}e^{(t-s)\ep\Delta}T_{\f{m_{\mu\nu}}{\phi_{\mu\nu}}}(\R (B(w,w)+\lnr\chi^{L}H)^{\mu},\R w^{\nu})\d s+ symmetric \quad terms\big]\nonumber\\
&\define&I_1+\cdots+I_4+symmetric \quad terms.
\een

Note that we denote $\R^{2}=\R$ as they have the same property (they are both $L^p(1<p<\infty)$ multiplier). It is also worthy to remark that the operator $e^{isb(D)}\hat{w}$ is well defined as $\hat{w}$ is supported on the low frequency region. %Lemma \ref{lem dispersive} shows that $e^{isb(D)}$ has the same properties as $e^{is\lnr}$ when we focus on the low frequency.
For notational brevity, we shall not distinguish $m_{\mu,\nu}$ (just write them as $m$) and ignore the summation on $\mu,\nu$.
Therefore, in the following, we will write $I_1-I_4$ as follows:
\beno
I_1&=&T_{\f{m}{\phi}}(\R w(t),\R w(t)),\quad I_2=-T_{\f{m}{\phi}}(\R w_0,\R w_0),\nonumber\\
I_3&=&-\int e^{-i(t-s)b(D)}e^{(t-s)\ep\Delta} T_{\f{m}{\phi}} (\R B(w,w),\R w)\d s,\nonumber\\
I_4&=&-\int e^{-i(t-s)b(D)}e^{(t-s)\ep\Delta} T_{\f{m}{\phi}} (\R\lnr\chi^{L}H,\R w)\d s.\nonumber\\
\eeno
\subsection{Decay estimate and weighted  estimate}
\subsubsection{Estimate of $K_1$ and boundary terms $I_1,I_2$}   
 
 We begin with the decay estimate of $K_1$. Since $w_0=(Q^{-1}\chi^L U_0)_1\approx \R \f{\lnr}{|\na|}\vr_0+i\R^{*}u_0,$ we have by the dispersive estimate \eqref{dispersive}
\beno
\|%|\na|^{\delta}
|\na|^{\f{1}{2}} e^{-itb(D)}e^{\ep t\Delta}w_0\|_{W^{1,\infty}}
\lesssim \||\na|^{\f{1}{2}}w_0\|_{W^{3,1}}\lesssim\|(\vr_0,u_0,\na\varphi_0)\|_{W^{\f{7}{2}+\delta,1}}
\eeno
Note that the last inequality in the above arises from the fact that $\R \hdb_{j}$ is $L^1$ multiplier for any $j\in \mathbb{Z}$.
% we have used the fact that:
% $\R e^{isb(D)}\partial_{s}f(s)=\R \R (\ep\Delta w+B(w,w)+n(D)\chi^{L}H)$ 
% For $I_3,I_4$, we integrate by parts in time again to get:
% \beno
 %I_3&=&\epd T_{\f{m}{\phi}}(\R w(t),\R w(t))-\epd e^{-itb(D)}e^{\ep t\Delta}T_{\f{m}{\phi}}(\R w_0,\R w_0)\nonumber\\
 %&&+\int_{0}^{t}e^{-i(t-s)b(D)}e^{(t-s)\ep\Delta}(\ep\Delta)^{2} T_{\f{m}{\phi}}(\R w,\R w)(s)\d s\nonumber\\
%&&+\int e^{-i(t-s)b(D)}e^{(t-s)\ep\Delta}\epd T_{\f{m}{\phi}} (\R (\ep\Delta w+B(w,w)+\lnr\chi^{L}H),\R w)\d s\nonumber\\
%&&\qquad\qquad + symmetric \quad term]\\
%&\define&I_{31}+I_{32}\cdots+I_{36}
 %\eeno
  %\beno
 %I_4&=& T_{\f{m}{\phi}}(\R \epd w(t),\R w(t))-e^{-itb(D)}e^{\ep t\Delta}T_{\f{m}{\phi}}(\R \epd  w_0,\R w_0)\nonumber\\
 %&&+\int_{0}^{t}e^{-i(t-s)b(D)}e^{(t-s)\ep\Delta}(\ep\Delta) T_{\f{m}{\phi}}(\R \epd  w,\R w)(s)\d s\nonumber\\
%&&+\int e^{-i(t-s)b(D)}e^{(t-s)\ep\Delta} T_{\f{m}{\phi}} (\R \epd (\ep\Delta w+B(w,w)+\lnr\chi^{L}H),\R w)\d s\nonumber\\
%&&\qquad\qquad + symmetric \quad term]\\
%&\define&I_{41}+I_{42}\cdots+I_{46}
% \eeno
%\subsubsection{Estimate of boundary term}
Next, for the decay estimate for $\R I_2,$ we take benefits of the Sobolev embedding, dispersive estimate \eqref{dispersive}, bilinear estimate \eqref{bilinear estimate} (use $2_{+}=\f{9}{4}-\delta$) to get:
\beno
&&\||\na|^{\frac{1}{2}}\R e^{itb(D)}\tilde{\chi}^{L}e^{\ep t \Delta }T_{\f{m}{\phi}}(\R w(0),\R w(0))\|_{W^{1,\infty}}\lesim\ltr^{-1}
\sum_{j\in\mathbb{Z}}2^{\f{1}{2}j}\langle2^j\rangle^3\|T_{\f{m}{\phi}}(\R w(0),\R w(0))\|_{L^1}\\
 && \qquad\qquad\qquad\qquad\qquad\lesssim \ltr^{-1}\|T_{\f{m}{\phi}}(\R w(0),\R w(0))\|_{W^{\f{7}{2}+\delta,1}}
 \lesssim\ltr^{-1}\|w(0)\|_{H^{\f{23}{4}}}\|w(0)\|_{H^2}.
 \eeno
As for the decay estimate of $\R I_1$, it is helpful to establish the following lemma:
 %a lemma whose proof is essentially included in Lemma 5.2 of \cite{MR3274788}.
\begin{lem}\label{high regularity decay}
For any $2\leq p< \infty$, and $k<\f{2}{p}N'+(1-\f{2}{p})\f{3}{2}$, we have for every $t\in[0,T),$
\beq\label{inter1}
\|U^L(t)\|_{W^{k,p}}\leq (1+t)^{-(1-\f{2}{p})}\|U\|_{X_T}.
\eeq
similarly, if $k<\f{2}{p}N+(1-\f{2}{p})\f{3}{2}$, then
\beq\label{inter2}
\|U^L(t)\|_{W^{k,p}}\leq (1+t)^{-(1-\f{2}{p})+\f{2}{p}\delta}\|U\|_{X_T}.
\eeq
\end{lem}
\begin{proof}
We only detail the proof of \eqref{inter1} since \eqref{inter2} can be treated in the same manner. We shall use decomposition $U^L=\Delta_{-1}U^{L}+\sum_{j\geq 0}\Delta_j U^{L}$.
On one hand,  the low frequency can be dealt with as follows:
\beqs
\|\Delta_{-1}U^{L}\|_{W^{k,p}}\lesssim\|U^L\|_{L^p}\lesssim(1+t)^{-(1-\f{2}{p})}\|U\|_{X_T}.
\eeqs
 On the other hand, the high frequency term can be controlled by 
 interpolation and the definition of $X_T$ norm:
 \beno
 \sum_{j\geq 0}\|\Delta_j U^{L} \|_{W^{k,p}}&\lesssim& \sum_{j\geq 0} 2^{kj}\|\Delta_j U^{L}\|_{L^2}^{\f{2}{p}}\|\Delta_j U^L\|_{L^{\infty}}^{1-\f{2}{p}}\\
 &\lesssim&\sum_{j\geq 0}2^{j(k-N'\f{2}{p}-\f{3}{2}(1-\f{2}{p}))}(2^{N'j}\|\Delta_j U^{L}\|_{L^2})^{\f{2}{p}}
\||\na|^{\delta}\lnr\Delta_j U^L\|_{L^{\infty}}^{1-\f{2}{p}}\\
 &\lesssim&(1+t)^{-(1-\f{2}{p})}\|U\|_{X_T}.
 \eeno
\end{proof}
 In light of \eqref{bilinear estimate}, 
 \eqref{inter1} 
 and condition  $\f{19}{4}<\f{2}{3}\cdot 7+\f{1}{2}$, one has that
 \ben\label{decay for I1}
 &&\|\R |\na|^{\f{1}{2}}T_{\f{m}{\phi}}(\R w(t),\R w(t))\|_{W^{1,\infty}}\lesssim\|T_{\f{m}{\phi}}(\R w(t),\R w(t))\|_{H^{\f{5}{2}+\delta}}\nonumber\\
 && \qquad\lesssim\|w(t)\|_{W^{\f{19}{4},3}}\|\R w\|_{W^{2,6}}
 \lesssim(1+t)^{-1}\|U\|_{X_T}^{2}.
 \een

%We omit  $L^{\f{1}{\delta}}$ estimate as it is similar (note that we do not need to involve $|\na|^{\delta} $ in the estimate of this norm as $\R$ is $L^{\f{1}{\delta}}$ multiplier).

We are now committed to the weighted estimate. Let us first detail the estimate of boundary terms: $x\R e^{itb(D)}(I_1,I_2)$. Using
 \eqref{reduction} again, it suffices to estimate $\|x e^{itb(D)}(I_1,I_2)\|_{H^{4+\delta}}$. %Before carrying out the estimation, we recall the profile of $w$:
 Denote $f=e^{itb(D)}w$ the profile of $w,$ one then writes
 \beno
&& xe^{itb(D)}T_{\f{m}{\phi}}(\R w(t),\R w(t))=xe^{it\Im\phi}T_{\f{m}{\phi}}(\R f(t),\R f(t))\\
& =&%\cF^{-1}\big(
-te^{itb(D)}T_{\f{m\p_{\xi}(\Im\phi)}{\phi}}(\R w(t),\R w(t))+ie^{itb(D)}T_{\partial_{\xi}(\f{m}{\phi})}(\R w(t),\R w(t))\\
 &&+e^{itb(D)}T_{\f{m}{\phi}}(e^{isb(D)}x\R f(t),\R w(t)).
 \eeno
 where $\Im{\phi}=b(\xi)\pm b(\xi-\eta)\pm b(\eta)$.
Thanks to \eqref{bilinear estimate},
 \eqref{inter1} and relation $\frac{25}{4}\leq \frac{16}{17}\cdot 7,$ the first term can be controlled as:
 \beno
 &&\|te^{itb(D)}T_{\f{m}{\phi}}(\R w(t),\R w(t))\|_{H^{4+\delta}}
 \lesssim t\|\R w\|_{W^{2,34}}\|\R w\|_{W^{\f{25}{4},\f{17}{8}}}\lesim \|U\|_{X_T}^{2}.
 \eeno
 The second term is easier, since it does not contain prefactor $t.$
 Moreover, the quadratic form  $T_{\partial_{\xi}(\f{m}{\phi})}$ admits the similar bilinear estimates as $T_{\f{m}{\phi}}.$
  The third term is much involved since one could not put the loss of derivative on the weighted term.
  We thus write:
  \beno
 && %\lxr^{2+\delta}
 \cF(e^{itb(D)}T_{\f{m}{\phi}}(e^{isb(D)}x\R f(t),\R w(t)))\\
  &=&\int e^{itb(\xi)}\f{m}{\phi}(\xi,\eta)e^{-isb(\xi-\eta)}\widehat{x\R f}(\xi-\eta)\widehat{\R w}(\eta)\chi_{\{\lxmer\leq \ler\}}\d\eta\\
  &&+\int e^{it\Im\phi(\xi,\eta)}\f{m}{\phi}(\xi,\eta)\partial_{\xi}\widehat{\R f}(\xi-\eta)\widehat{\R f}(\eta)\chi_{\{\lxmer> \ler\}}\d\eta\define I_{131}+I_{132},\\
  \eeno
By virtue of \eqref{bilinear estimate}, 
 \eqref{inter1} and condition $7.5\leq \f{2}{3}\cdot 11+\f{1}{2}$, one gets that:
  \beno
  \|I_{131}\|_{H^{4+\delta}}&\lesssim&\|T_{\f{m}{\phi}\chi_{\{\lxmer\leq \ler\}}}(e^{itb(D)}x\R f(t),\R w(t))\|_{H^{4+\delta}}\\
 &\lesssim& \|e^{itb(D)}x\R f(t)\|_{W^{2,\f{2}{1-2\delta}}}\|\R w\|_{W^{\f{25}{4},\f{1}{\delta}}}\\
 &\lesssim&\ltr^{2\delta}\|\langle x\rangle f\|_{W^{3,\f{2}{1-\delta}}}\|w\|_{W^{7.5,3}}\lesssim\ltr^{-(\f{1}{3}-3\delta)}\|U\|_{X_T}^{2},
  \eeno
%{\color{green}write more}
For the term $I_{132}$, thanks to identity $\p_{\xi}\widehat{\R f}(\xi-\eta)=-\p_{\eta}\widehat{\R f}(\xi-\eta)$, one could integrate by parts in $\eta$ to rewrite it as:
\beno
I_{132}=T_{\f{m}{\phi}\chi_{\{\lxmer> \ler\}}}(\R w(t),e^{itb(D)}x\R f(t))+T_{\p_{\eta}(\f{m}{\phi}\chi_{\{\lxmer> \ler\}})}(\R w(t),\R w(t))\\
\qquad+ it T_{\f{m}{\phi}\chi_{\{\lxmer> \ler\}}\p_{\eta}(Im\phi)}(\R w(t),\R w(t)).
\eeno
Nevertheless, the first term in the above can be estimated exactly as $I_{311}$, the last two terms can be treated in the same manner as that of $\|I_1\|_{H^{2+\delta}}$, see \eqref{decay for I1}. 
 
 We are now in position to show the estimate of $xe^{itb(D)}I_2$. By definition,
 \beno
 xe^{itb(D)}I_2&=&xe^{\ep t\Delta}T_{\f{m}{\phi}}(\R w_0,\R w_0)\\
 &=&2e^{\ep t\Delta}\ep t\na T_{\f{m}{\phi}}(\R w_0,\R w_0)+e^{\ep t\Delta}T_{\f{m}{\phi}}(x \R w_0,\R w_0)+iT_{\partial_{\xi}(\f{m}{\phi})}(\R w_0,\R w_0).
 \eeno
Let us focus on the estimate of the first two terms, since the last one is easier.
Owing to the bilinear estimate \eqref{bilinear estimate}, one has 
 \beno
&& \|e^{\ep t\Delta}\ep t\na T_{\f{m}{\phi}}(\R w_0,\R w_0)\|_{H^{4+\delta}}+\|e^{\ep t\Delta}T_{\f{m}{\phi}(\chi_{\{\lxmer\leq \ler\}})}(x \R w_0,\R w_0)\|_{H^{4+\delta}}\\
 &\lesssim& \ep^{\f{1}{2}}t\|\cF^{-1}(e^{-\ep t|\xi|^{2}}\ep^{\f{1}{2}}\xi)\|_{L^2}\|T_{\f{m}{\phi}}(\R w_0,\R w_0)\|_{W^{4+\delta,1}}+\|T_{\f{m}{\phi}(\chi_{\{\lxmer\leq \ler\}})}(x \R w_0,\R w_0)\|_{H^{4+\delta}}\\
 &\lesssim&\|T_{\f{m}{\phi}}(\R w_0,\R w_0)\|_{W^{4+\delta,1}}+\|x\R (\vr_0,\varphi_0,u_0)\|_{W^{2,3}}\|\R (\vr_0,\varphi_0,u_0) \|_{W^{\f{25}{4},6}}\\
 &\lesssim&\|\R (\vr_0,\varphi_0,u_0)\|_{H^2}\|\R (\vr_0,\varphi_0,u_0) \|_{H^{\f{25}{4}}}+\|\langle x\rangle (\vr_0,\varphi_0,u_0) \|_{H^{\f{7}{3}}}\|(\vr_0,\varphi_0,u_0) \|_{H^7}\\
 &\lesssim&\|\langle x\rangle (\vr_0,\varphi_0,u_0) \|_{H^{\f{7}{3}}}\|(\vr_0,\varphi_0,u_0) \|_{H^7}.
 \eeno
Note that in the above, the following fact has been used:
 \ben
 &&\|\cF^{-1}(e^{-\ep t|\xi|^{2}}\ep^{\f{1}{2}}\xi)\|_{L^2}\lesssim \ep^{-\f{1}{2}}t^{-1},\label{elementary}\\
 &&\|x\R f\|_{L^3}\lesssim \|xf\|_{L^3}+\||\na|^{-1}f\|_{L^3}\lesssim \|xf\|_{H^{\f{1}{3}}}+\|f\|_{L^{\f{6}{5}}}\lesssim\|\langle x\rangle f\|_{H^{\f{1}{3}}}.\nonumber
 \een
 The estimate of $T_{\f{m}{\phi}(\chi_{\{\lxmer> \ler\}})}(x\R w_0,w_0)$ can be obtained by integrating by parts in $\eta$ as before.
 \subsubsection{Estimate of $K_3$ and $\R I_4$}
We begin with the decay estimate of $K_3$ (which is defined in \eqref{formula of w}). 
By dispersive estimate, $\|\R K_3\|_{L^{\infty}}$ can be estimated as follows:
\beno
 &&\|\R\int_{0}^{t}e^{i(t-s)\lambda_{-}(D)}|\na|^{\f{1}{2}}\lnr \lnr\chi^{L}H(s)\d s\|_{L^{\infty}} \\
 &\lesssim& \int_{0}^{t}(1+t-s)^{-1}\sum_{j\in\mathbb{Z}}2^{\f{1}{2}j}\langle 2^j\rangle^{4}\|\dot{\Delta}_{j} H(s)\|_{L^1}\d s\lesssim \int_{0}^{t}(1+t-s)^{-1}\|H(s)\|_{B_{1,2}^{5}}\d s\\
 &\lesssim& \int_{0}^{t}(1+t-s)^{-1}\|U^h\|_{H^{5}}\|U\|_{H^{5}}\d s \lesssim \int_{0}^{t}(1+t-s)^{-1}(1+s)^{-\alpha}\|U\|_{X_T}^{2}\d s\lesssim (1+t)^{-1}\|U\|_{X_T}^{2}
\eeno
We now prove the weighted estimate of $K_3.$ According to \eqref{reduction}, it suffices to estimate $\|x(3)\|_{H^{4+\delta}}$.
Rewrite $H=\R U^h \R U^{L}+\R U^{h}\R U^{h}\define H_1+H_2,$ one has that by the definition of $W:W=Q^{-1}U^{L}$,
\ben
&&x\int_{0}^{t}e^{isb(D)}e^{\ep (t-s)\Delta}\lnr\chi^{L}H_1 \d s\nonumber\\
&=&\cF^{-1}\big(\int_{0}^{t}\int s(b'(\xi)\pm b'(\xi-\eta))+2\ep(t-s)i\xi) e^{isb(\xi)}e^{-\ep(t-s)|\xi|^{2}}\lxr \widehat{\R U^{h}}(\eta) \widehat{\R U^L}(\xi-\eta)\d \eta \d s\nonumber\\
&&+\int_{0}^{t}e^{isb(\xi)}e^{-\ep(t-s)|\xi|^{2}}\partial_{\xi}\big(\lxr\chi^{L}(\xi)\big)\hat{H_1}(\xi)\d s \nonumber\\
&&+\int_{0}^{t}e^{is(b(\xi)\pm b(\xi-\eta))}e^{-\ep(t-s)|\xi|^{2}}\lxr\chi^{L}(\xi)\widehat{\R U^h}(\eta)\partial_{\xi}\widehat{ Q\tilde{\chi}^{L}\R f }(\xi-\eta)\d \eta\d s\big)\nonumber\\
&\define&(3)_{11}+(3)_{12}+(3)_{13}
\een
For $(3)_{11},$ by virtue of   \eqref{elementary}, the fact $b'\chi^{L}(D)$ is $L^p  (1<p<\infty)$ multiplier as well as Young's inequality:
\beno
\|(3)_{11}\|_{H^{4+\delta}}&\lesssim& \int_{0}^{t} s\|\R U^{h} \R U^{L}\|_{H^{5+\delta}}\d s+\izt \|\R U^h \R U^L\|_{W^{5+\delta,1}}\d s\\
&\lesssim& \int_{0}^{t} s \|\R U^h\|_{H^{5+\delta}} \|\R U^L\|_{W^{1,4}}\d s+\int \|\R U^h\|_{H^5}\|\R U^L\|_{L^2}\d s\\
&\lesssim& \int_{0}^{t} s\lsr^{-(2-5\delta+\f{1}{2})}\d s\|U\|_{X_T}^{2}+\int \lsr^{-(2-5\delta)}\|U\|_{X_T}^{2}\lesssim\|U\|_{X_T}^{2}.
\eeno
%Note we will choose $\delta\leq\f{1}{100}$.
The estimate of $(3)_{12}$ is similar, we thus skip it.
For $(3)_{13}$, by %bilinear estimate  
\eqref{bilinear estimate}, \eqref{Lp bounds1}, %$L^p-L^p$ estimate for $e^{itb(D)}\chi^L(D)$:  %(Lemma \ref{Lp bounds for eitb(D)}):
\beno
&&\|(3)_{13}\|_{H^{4+\delta}}\lesssim \int_{0}^{t} \|\R U^h\|_{W^{5+\delta,6}} \|e^{isb(D)}xQ\tilde{\chi}^{L}\R f\|_{L^3}\d s\\
&\lesim& \int_{0}^{t}\|\R U^h\|_{H^{6}}\lsr^{\f{1}{3}}\|xQ\tilde{\chi}^{L}\R f\|_{W^{\f{1}{2},3}}\d s
\lesim \int_{0}^{t} \lsr^{-(2-5\delta-\f{1}{3}})\|U\|_{X}^2\d s\lesim \|U\|_{X_T}^2.
\eeno
Note that in the above, we have also used the fact
$xQ\tilde{\chi}^{L}\R f\approx |\na|^{-1}\R f+\R (xf)$ which gives:
\beno
\|xQ\tilde{\chi}^{L}\R f\|_{W^{\f{1}{2},3}}\lesim \|\R f\|_{W^{\f{1}{2},\f{6}{5}}}+\|xf\|_{W^{
\f{1}{2},3}}\lesim \|f\|_{H^{\f{1}{2}}}+\|xf\|_{W^{\f{5}{6},\f{2}{1-\delta}}}.
\eeno
For the case of $H_2$, since in the original definition $H_2=\vr^{h}u^{h}$ or $(u^{h})^{2}$, one can split it into three terms:
\ben
&&x\int_{0}^{t}e^{isb(D)}e^{\ep (t-s)\Delta}\lnr\chi^{L}H_2 \d s\nonumber\\
&=&\int_{0}^{t}\int sb'(D)+2(t-s)\ep i\na) e^{isb(D)}e^{\ep(t-s)\Delta}\lxr H_2 \d s\nonumber\\
&&+\int_{0}^{t}e^{isb(D)}e^{\ep(t-s)\Delta}\cF^{-1}\big(\partial_{\xi}(\lxr\chi^{L})\hat{H_1}(\xi)\big)\d s 
+\int_{0}^{t}e^{isb(D)}e^{\ep(t-s)\Delta}\lnr \chi^{L}(x\vr^{h}u^{h})\d s\nonumber\\
&\define&(3)_{21}+(3)_{22}+(3)_{23}
\een
The estimates of $(3)_{21},(3)_{22}$ are similar to that of $(3)_{11},(3)_{12}$ and thus can be omitted.
For $(3)_{23}$, one uses Lemma \ref{weighted product}, the estimate \eqref{weighted high frequency} to get:
\beno
\|\lnr\chi^{L}(x\vr^{h}u^{h})\|_{H^{4+\delta}}&\lesim&\|xu^{h}\|_{L^2}\|\vr^h\|_{H^{6+2\delta,\infty}}+
\|x\vr^{h}\|_{L^2}\|u^h\|_{H^{6+2\delta}}+
\|\vr^{h}\|_{H^{5+2\delta}}\|u^h\|_{H^{5+2\delta}}\\
&\lesim& \lsr^{-(2-7\delta)}(\|U\|_{X_T}^{2}+\|x(\vr_0,u_0,\na\varphi_0)\|_{L^2}^2).
\eeno

 \subsubsection{Estimate of $I_4$}
 
% It remains for us to get the estimate for $I_4$. 
One first observes that
 \beno
 &&\||\na|^{\f{1}{2}}\lnr \R I_4\|_{L^{\infty}}\lesssim \sum_{j\in\mathbb{Z}}2^{\f{1}{2}j}\langle 2^{j}\rangle\|\dot{\Delta}_j I_4\|_{L^{\infty}}\\
 &\lesssim&\ltr^{-1} \sum_{j\in\mathbb{Z}}2^{\f{1}{2}j}\langle 2^{j}\rangle^{3}\|\dot{\Delta}_j e^{itb(D)}I_4\|_{L^{1}}
 \lesssim\ltr^{-1}\|e^{itb(D)}I_4\|_{W^{\f{7}{2}+\delta,1}}
 %&\lesssim&\ltr^{-1}(\|e^{itb(D)}I_4\|_{H^{3+2\delta}}+\|xe^{itb(D)}I_4\|_{W^{3+2\delta,2-\delta}})
 \eeno
Applying Lemma \ref{Lp bounds for eitb(D)} for $p=1$, we get:
\beno
&&\|e^{itb(D)}I_4\|_{W^{\f{7}{2}+\delta,1}}\lesim \izt s \|T_{\f{m}{\phi}}(\lnr \R \chi^L H, \R w)\|_{W^{\f{9}{2}+\delta,1}}\d s\\&\lesim& \izt s \|\lnr \R\chi^L H\|_{H^{\f{27}{4}}}\|w\|_{H^{\f{27}{4}}}\d s
\lesim\izt s\lsr^{3-7\delta}\d s\|U\|_{X}^3\lesim\|U\|_{X_T}^3.
\eeno
where the following crude estimate has been used
\beq\label{decay for $H$}
\|\lnr \chi^{L}\R H\|_{H^{N-3}}%+\|\lnr \chi^{L}\R H\|_{W^{1,\f{1}{\f{1}{2-\delta}-\f{1}{3}}}}
\lesssim \|H\|_{H^{N-2}}\lesssim\|U\|_{L^{\infty}}\|U^{h}\|_{H^{N-2}}\lesssim \lsr^{-(3-7\delta)}\|U\|_{X_T}^{2}.
\eeq
%Recall that $N\geq 11$.
We are now devoted to proving the estimate of $\|xe^{itb(D)}I_4\|_{W^{4,\f{2}{1-\delta}}}$. As before, let us write
 \beno
 xe^{itb(D)}I_4&=&\cF^{-1}\big(\int_{0}^{t}s(b'(\xi)\pm b'(\xi-\eta))+2\ep(t-s)i\xi e^{is(b(\xi)\pm b(\xi-\eta))}e^{-\ep (t-s)|\xi|^{2}}\\
&&\qquad \qquad\qquad\qquad\qquad\qquad\f{m}{\phi}(\xi,\eta)\chi^{L}(\eta) \ler \widehat{\R H}(\eta) \widehat{\R f}(\xi-\eta)\d \eta\d s\\
 &&+\int_{0}^{t}e^{is(b(\xi)\pm b(\xi-\eta))}e^{-\ep (t-s)|\xi|^{2}}\partial_{\xi}(\f{m}{\phi}) \chi^{L}(\eta) \ler \widehat{\R H}(\eta) \widehat{\R f}(\xi-\eta)\d \eta\d s \\
  &&+\int_{0}^{t}e^{is(b(\xi)\pm b(\xi-\eta))}e^{-\ep (t-s)|\xi|^{2}}\f{m}{\phi}(\xi,\eta)\chi^{L}(\eta) \ler \widehat{\R H}(\eta) \partial_{\xi}\widehat{\R f}(\xi-\eta)\d \eta\d s \big)\\
  &\define& J_{41}+J_{42}+J_{43}.
 \eeno
The first term $J_{41}$ can be dealt with similarly as the term $ (3)_{11}$. %we note that $b'(\xi)\chi^{L}(\xi)$ and $\ep^{\f{1}{2}}\na e^{\ep(t-s)\Delta}\chi^{L}$ are both $L^p$ multiplier ($1<p<\infty$) uniformly in $\ep$   with the norm $c$ and $c(1+t-s)^{-\f{1}{2}} respectively.$
%{\color{red}
Indeed, by using Lemma \ref{Lp bounds for eitb(D)}, Lemma \ref{lem bilinear estimate} and \eqref{decay for $H$}, one obtains
\beno
&&\|J_{41}\|_{W^{4,\f{2}{1-\delta}}}\lesssim\int_{0}^{t}\lsr^{1+\delta}\|T_{\f{m}{\phi}}(\lnr \chi^{L}\R H,\R w)(s)\|_{W^{4+2\delta,\f{2}{1-\delta}}}\d s
\\&&\qquad\qquad\qquad\qquad+\izt \lsr^{\delta} \|T_{\f{m}{\phi}}(\lnr \chi^{L}\R H,\R w)(s)\|_{W^{4+2\delta,\f{2}{2-\delta}}}\d s\\
&\lesssim& \int_{0}^{t} 
%(\|\R w|_{W^{1,10}}\|\lnr \chi^{L}\R H\|_{W^{6+3\delta,\f{1}{\f{1}{2-\delta}-\f{1}{10}}}}+\|\R w\|_{W^{6+3\delta,3}}\|\lnr \chi^{L}\R H\|_{W^{1,\f{1}{\f{1}{2-\delta}-\f{1}{3}}}})\\
\lsr^{1+\delta} \|\R w\|_{W^{\f{25}{4},\infty}}\|\lnr \chi^{L}\R H\|_{W^{\f{25}{4},\f{2}{1-\delta}}}\d s+\izt \lsr^{\delta}\|\R w\|_{W^{\f{25}{4},\f{2}{1-\delta}}}\|\lnr \chi^{L}\R H\|_{H^{\f{25}{4}}}\d s\\
&\lesssim& \int_{0}^{t} \lsr^{1+\delta} %(\lsr^{-\f{4}{5}}
\lsr^{-(3-7\delta)}%+\lsr^{-(1-2\delta+\alpha)}
\d s\|U\|_{X}^3\lesssim\|U\|_{X_T}^{3}.
\eeno
 $J_{42}$ can be estimated in the same manner, we thus do not detail it.
For $J_{43}$, one splits it into two terms:
\beno
J_{43}=\cF^{-1}\big(\int(\chi_{\lxmer\leq\ler}+\chi_{\lxmer\geq\ler})\cdots\d s\big)\define J_{431}+J_{432}.
\eeno
The estimate of $J_{431}$ is easy since we can put all the derivatives onto $H$. Indeed, by 
\eqref{decay for $H$}, Lemma \ref{Lp bounds for eitb(D)}, %Hardy-Littlewood-Sobolev inequality 
and the Sobolev embedding, one obtains that
\beno
\|J_{431}\|_{W^{4,\f{2}{1-\delta}}}&\lesssim& \int_{0}^{t}\|\lnr \chi^{L}H\|_{W^{\f{25}{4},\f{6}{1-3\delta}}}\|e^{isb(D)}x\R f\|_{W^{2,3%2+2\delta
}}\d s\\
&\lesssim&\int_{0}^{t}\lsr^{-(3-7\delta)}\lsr^{\f{1}{3}   %2\delta
}\|\langle x\rangle f\|_{W^{3,\f{2}{1-\delta}}}\d s\|U\|_{X_T}^{2}.
\lesssim\|U\|_{X_T}^{3}
\eeno
For $J_{432}$, we use the identity $\partial_{\xi}\widehat{\R f}(\xi-\eta)=-\partial_{\eta}\widehat{\R f}(\xi-\eta)$ 
to integrate by parts in $\eta$. Eventually, we get the terms like $J_{41},J_{42},J_{431}$ as well as the term:
\begin{equation}\label{last}
\int_{0}^{t} e^{isb(D)}e^{\ep (t-s)\Delta}T_{\f{m%\ler
\chi_{\{\lxmer\geq\ler\}}}{\phi}}(x\lnr\chi^{L}(D)\R H, \R w)\d s.
\end{equation}
Besides,
\beno
x\lnr \chi^{L}(D)\R H&=&\f{\na}{\lnr}\chi^L \R H+\lnr\big((i\chi^{L}%\langle\cdot\rangle 
)'(D)
\R H+\chi^{L}(D)x\R H\big)\\&\approx& \f{\na}{\lnr}\chi^L \R H+\lnr\big((i\chi^{L}%\langle\cdot\rangle
)'(D)\R H +\chi^{L}(D)|\na|^{-1}H+\chi^{L}(D)\R xH\big).
\eeno
%{\color{red}
To continue, the following estimate for $H$ shall be useful:
\begin{prop}\label{lem when estimate xR H}
\beno
\|(\chi^{L})'(D)\R H +\chi^{L}(D)|\na|^{-1}H\|_{W^{4,\f{2}{1-\delta}%(\f{1}{2+\delta}-\f{1}{3})^{-1}
}}\lesim \ltr^{-(2-5\delta)}\|U\|_{X_T}^{2},\\
\|\chi^{L}\R xH\|_{L^{2}}\lesssim \ltr^{-(1-4\delta)}(\|U\|_{X_T}^{2}+\|\langle x\rangle(\vr_0,u_0)^{h}\|_{L^{2}}^{2}).
\eeno
\end{prop}
We postpone the proof of this proposition and finish firstly the estimate of term \eqref{last}. Indeed,  Lemma \ref{high regularity decay}, Proposition \ref{lem when estimate xR H}, combined with the Sobolev embedding
$W^{\f{41}{4},\f{21}{10}}
\hookrightarrow W^{\f{37}{4},\infty},
H^{8}\hookrightarrow W^{\f{25}{4},\infty}$ yield that:
% and condition $\f{41}{4}<\f{20}{21}\cdot 11$ yield that:
%, we can control the $W^{4,\f{2}{1-\delta}}$ norm of above term by:
\beno
&&\int_{0}^{t}\|\R w\|_{W^{\f{25}{4},\infty}}\|\chi^{L}\R H+\lnr\big((\chi^{L}%\langle\cdot\rangle
)'(D)\R H +\chi^{L}(D)|\na|^{-1}H\big)\|_{W^{2,\f{2}{1-\delta}}}\\%(\f{1}{2+\delta}-\f{1}{3})^{-1}}}\\
&\quad&\qquad\qquad+\|\R w\|_{W^{\f{37}{4},\infty}}\| \chi^{L}(D)\R xH\|_{L^2}\d s\\
&\lesssim&\int_{0}^{t}\lsr^{-(2-5\delta)}\|U\|_{X}^{3}+\lsr^{-(\f{1}{21}-\delta)}\lsr^{-(1-4\delta)}(\|U\|_{X}^{2}+\|\langle x\rangle(\vr_0,u_0)^{h}\|_{L^2}^{2})\d s\\
&\lesssim&\|\langle x\rangle(\vr_0,u_0,\na\varphi_0)^{h}\|_{L^2}^{2}+\|U\|_{X_T}^{2}+\|U\|_{X_T}^{3}.
\eeno

%We note also that when we treat the term $T_{\f{m%\ler
%\chi_{\{\lxmer\geq\ler\}}}{\phi}}(\chi^{L}(D)\R x H, \R w)$, we used Sovolev embedding $H^{2+\delta}\hookrightarrow W^{2,2+\delta}$ and the fact $\ler<\lxmer$ to put all the derivatives onto $\R w$.

We are now left to prove Propostion \ref{lem when estimate xR H}.
\begin{proof}[Proof of Proposition \ref{lem when estimate xR H}]

Firstly,
%We write then  $x\chi^{L}(D)\R H=(\chi^{L})'(D)\R H+\chi^{L}(D)x\R H$,
since $(\chi^{L})'(D)\R$ is $L^2$ multiplier with norm $\lesssim \sqrt{\f{\ep}{\kpz}}$, one has by Sobolev embedding, the definition $H\approx \R U^{h}\R U$ that:
\beqs
\|(\chi^{L})'(D)\R H\|_{W^{4,\f{2}{1-\delta}}}\lesssim \sqrt{\f{\ep}{\kpz}}\|H\|_{W^{4,\f{2}{1-\delta}}}\lesssim\|U^h\|_{H^4}\|U\|_{H^1}\lesssim \ltr^{-(2-5\delta)}\|U\|_{X_T}^{2}.
\eeqs
Similarly, by Hardy-Littlewood-Sobolev inequality,
\beqs
\||\na|^{-1}H\|_{W^{4,\f{2}{1-\delta}%(\f{1}{2+\delta}-\f{1}{3})^{-1}
}}\lesim \|H\|_{W^{4,\f{2}{2-\delta}}}\lesim \|\R U^h\|_{H^4}\|\R U\|_{H^1}\lesim \ltr^{-(2-5\delta)}\|U\|_{X_T}^{2}.
\eeqs
We are now ready to estimate $\R x H.$ Notice that in the original definition $H=\rho^{L}u^{h}+\rho^{h}u$ or $u^{L}\cdot u^{h}+u^{h}\cdot u.$ %=\R (x\R U^{h},x\R U)\approx \R(|\na|^{-1}U^{h},\R U)+.
Therefore, due to the Sobolev embedding $W^{1,\f{1}{\delta}}\hookrightarrow L^{\infty}$ and weighted Sobolev estimate for high frequency \eqref{weighted high frequency}:
\beno
&&\|\chi^{L}\R xH\|_{L^{2}}\lesim \|x(\vr,u)^{h}\|_{L^2}\|(\vr,u)\|_{L^{\infty}}\\
&\lesim&\ltr^{2\delta}(\|U\|_{X_T}+\|\langle x\rangle(\vr_0,u_0,\na\varphi_0)^{h}\|_{L^{2}})\ltr^{-(1-2\delta)}\|U\|_{X_T}\\
&\lesssim& \ltr^{-(1-4\delta)}(\|U\|_{X_T}^{2}+\|\langle x\rangle(\vr_0,u_0,\na\varphi_0)^{h}\|_{L^{2}}^{2}).
\eeno
\end{proof}
\begin{subsubsection}{Estimate of $\R I_3$}

In view of the definition of $B(w,w)$ (see \ref{def of B}), we could indeed consider $B(w,w)$ as $\lnr \R (\R w)^{2}$ for simplicity. Therefore, by recalling the definition of profile $f(s)=e^{isb(D)}w$, we see that $I_3=e^{-itb(D)}I_3'$ %is the combinations of the terms 
with $I_3'$ under the form
\beq\label{I3'}
I_3'(t)=\izt\int\int e^{is\tilde{\phi}(\xi,\eta,\si)}e^{-\ep(t-s)|\xi|^2}\f{m}{\phi}(\xi,\eta)\ler\R (\eta) \widehat{\R f}(\xi-\eta)\widehat{\R f}(\eta-\si)\widehat{\R f}(\si)\d\si\d\eta\d s
\eeq
where $\tilde{\phi}=b(\xi)\pm b(\xi-\eta)\pm b(\eta-\si)\pm b(\si)$.
For the weighted estimate, thanks to Bernstein's inequality and Young's inequality, one has that
\begin{align*}
   \|x\R I_3'\|_{W^{4,\f{2}{1-\delta}}}&\leq \sum_{k\in\mathbb{Z}}2^{4k_{+}}\|\hdb_{k}x\R I_3'\|_{L^{\f{2}{1-\delta}}}\lesim
\sum_{k\in\mathbb{Z}}2^{4k_{+}}\left(\||\na|^{-1}\hdb_{k}I_3'\|_{L^{\f{2}{1-\delta}}}+\|\hdb_{k}xI_3'\|_{L^{\f{2}{1-\delta}}}\right)\\ &\lesssim\sum_{k\in\mathbb{Z}}2^{4k_{+}}\left(2^{\delta k}\|\hdb_{k}I_3'\|_{L^1}+2^{\f{6}{5}\delta k}\|\hdb_{k}xI_3'\|_{L^{2_{\delta}}}\right)\\
&\lesssim \|I_3'\|_{H^5}+ \sup_{k\in\mathbb{Z}}2^{-\f{4}{5}\delta k^{-}}2^{(4+2\delta)k^{+}}\left(\|[x,\hdb_{k}]I_3'\|_{L^{2_{\delta}}}+\|\hdb_{k}xI_3'\|_{L^{2_{\delta}}}\right)
\end{align*}
where we denote $2_{\delta}=\f{2}{1+\delta/5}$ and $k^{-}=\max(-k,0),k^{+}=\max(k,0)$.
Similarly, by dispersive estimate \eqref{dispersive}, H\"{o}lder's inequality,
\beqs
   \||\na|^{\f{1}{2}}I_3'\|_{W^{1,\infty}}\lesssim \ltr^{-1}\sum_{k\in\mathbb{Z}}
   2^{\f{k}{2}}2^{3 k_{+}}\|\hdb_{k} I_3'\|_{L^1} \lesssim \ltr^{-1}( \|I_3'\|_{H^4}+\sup_{k\in\mathbb{Z}
   }2^{\delta k_{-}}2^{4k_{+}}\|x\hdb_{k} I_3'\|_{L^{2_{\delta}}})
\eeqs
%Besides,
By Lemma \ref{Lp bounds for eitb(D)}, one has that: 
$\|e^{isb(D)}\chi^L(D)g\|_{L^{2_{\delta}}}\lesim \lsr^{\f{\delta}{5}}\|g\|_{W^{\delta,2_{\delta}}}.$ Therefore,
the commutator term can be bounded by applying Corollary \ref{cortrilinear} and Lemma \ref{high regularity decay}:
\beno
&&\|[x,\hdb_{k}]I_3'\|_{W^{4+2\delta,2_{\delta}}}=\|\izt e^{isb(D)}e^{\ep(t-s)\Delta}(\Phi_{k})'(D)T_{\f{m}{\phi}}\big(\R\lnr(\R w)^2,\R w\big)\|_{W^{4+2\delta,2_{\delta}}}\d s\\
&\lesim& \izt \lsr^{\f{\delta}{5}}\|T_{\f{m}{\phi}}
(\R\lnr(\R w)^2,\R w)\|_{W^{4+3\delta,\f{1}{1-3\delta/10}}}\|\cF^{-1}\big(2^{-k}\Phi'(2^{-k}\cdot)\big)\|_{L^{\f{10}{5+4\delta}}}\d s\\
&\lesim& 2^{-\f{4}{5}\delta k}\izt \lsr^{\f{\delta}{5}}\|w\|_{H^{N'}}\|w\|_{W^{2,4}}\|w\|_{W^{2,\f{20}{5-6\delta}}}\d s\lesim 2^{-\f{4}{5}\delta k}\izt\lsr^{\f{\delta}{5}}\lsr^{-1-\f{3}{5}\delta}\d s\|U\|_{X}^3\lesim 2^{-\f{4}{5}\delta k}\|U\|_{X_T}^3
\eeno
It now remains for us to estimate
$\sup_{k\in\mathbb{Z}
   }2^{\f{4}{5}\delta k_{-}}2^{(4+\delta)k_{+}}\|\hdb_{k}x I_3'\|_{L^{2_{\delta}}}$. By the expression of $I_3'$, we have that
   $ x I_3'=\sum_{j=1}^{4}Z_j$
   where 
   \begin{align*}
& Z_1=\int i s e^{isb(D)}e^{\ep(t-s)\Delta}T_{\f{m}{\phi}\p_{\xi}\tilde{\phi}}(\R \lnr (\R w)^2,\R w)\d s,\\
& Z_2=
\int_{0}^{t}e^{isb(D)}e^{\ep(t-s)\Delta}
T_{\partial_{\xi}(\f{m}{\phi})}(\R \lnr (\R w)^{2},\R w)\d s,\\
&Z_3=\int_{0}^{t}e^{isb(D)}e^{\ep(t-s)\Delta}T_{\f{m}{\phi}}(\R \lnr (\R w)^{2}, e^{-isb(D)}x\R f)\d s,\\
& Z_4=\izt \ep(t-s)\na e^{\ep(t-s)\Delta}e^{isb(D)}T_{\f{m}{\phi}}(\R \lnr (\R w)^{2},\R w)\d s.
   \end{align*}
We first remark that  $\|\hdb_{k}xZ_4\|_{W^{4+2\delta,2_{\delta}}}$
can be estimated in the same manner as that of $\|[x,\hdb_{k}]I_3'\|$, the only difference is that at this stage we use the fact the $L^{\f{10}{5+4\delta}}$ norm of the kernel of $\hdb_{k}\ep(t-s)\na e^{\ep(t-s)\Delta}$ is less than $2^{-\f{4}{5}\delta k}$ uniformly in $\ep\in (0,1].$ Indeed, one can think of $\hdb_{k}\ep(t-s)\na e^{\ep(t-s)\Delta}$ as $\hdb_{k}|\na|^{-1}$, since $\|\mathcal{F}^{-1}\big(\varepsilon(t-s)\xi^2 e^{-\varepsilon(t-s)|\xi|^2}\big)\|_{L^1}$ is uniformly bounded.
We point out here that we choose $2_{\delta}=\f{2}{1+\delta/5}$ mainly to manage to control the commutator term $[x,\hdb_{k}]I_3'$ and $Z_4$, since the situation is better if $2_{\delta}$ is closer to 2. We also emphasize that the presence of the half derivative when we control the $L_{x}^{\infty}$ norm of $I_3$ is necessary in here. Indeed, as we explained in the introduction, due to the weak dispersive estimate, we need to control it by the weighted norm: $\|xI_{3}'\|_{L^{2^{-}}}.$ Nevertheless, when we deal with the $\|Z_4\|_{L^{2^{-}}}$ which corresponds to the frequency derivative hits on $e^{\ep(t-s)\Delta}$, to compensate the growth of $(t-s)$, the best one
can use is: $\|\na e^{\ep(t-s)\Delta}\|_{L^{2^{-}}}\lesim (\ep(t-s))^{-1^{-}}$, which is obviously not enough. 
Based on this, the extra derivative can help in the sense that we could find some $1^{+}$, such that $\|\na^{1^{+}} e^{\ep(t-s)\Delta}\|_{L^{2^{-}}}\lesim (\ep(t-s))^{-1}$.

The estimations for $Z_1$-$Z_3$ are more involved since in these cases we can not use any information of heat kernel.
However, since it has been showed that $e^{itb(D)}\chi^{L}(D)$, $\Im\phi(\xi,\eta)=b(\xi)\pm b(\xi-\eta)\pm b(\eta)$ and $\tilde{\phi}(\xi,\eta,\si)=b(\xi)\pm b(\xi-\eta)\pm b(\eta-\si)\pm b(\si)$
has the same properties as $e^{it\lnr}$, $\lxr\pm\lxmer\pm\ler$ and $\lxr\pm \lxmer\pm \lemsr\pm \lsir$ respectively, they can be achieved  by the similar arguments as in \cite{MR3274788} where the global existence of 2-d Euler-Poisson is proved. %The weighted Sobolev norm estimate of $\R I_3$ is much involved, but we could keep track of the arguments in the proof of $2d$ Euler-Poisson system \cite{MR3274788}.
%Since they do not belong to our contribution in this paper, 
We thus just sketch them in the appendix for completeness and reader's convenience.

\end{subsubsection}

%\end{subsection}.

\begin{subsection}{Estimate of $H^{N'}$}
In this short subsection, we deal with the estimate of $\| U^L\|_{H^{N'}}$. By virtue of the definition $U^L=Q^{-1}\chi^{L}W$ and the fact that $Q^{-1}\chi^{L}$ is a $L^2$  multiplier, one easily sees that $\|U^L\|_{H^{N'}}\lesim \|W\|_{H^{N'}}$. 
 By \eqref{formulae of w}, \eqref{quadratic low term}, $$w=(1)+(3)+I_1+I_2+I_3+I_4$$ where $(1),(3)$ are defined in \eqref{formula of w} and $I_1-I_4$ is defined in \eqref{quadratic low term}.
 
 We shall only show the estimation of $I_1-I_4$.
For $I_1,I_{2}$, by bilinear estimate \eqref{bilinear estimate}(we use $2_{+}=\f{9}{4}-2\delta$), Sobolev embedding and the definition of $X_T$ norm,
\beno
&&\|I_1\|_{H^{N'}}\lesim\|T_{\f{m}{\phi}}(\R w(t)),\R w(t))\|_{H^{N'}}\lesim \|\R w\|_{H^{N'+\f{9}{4}}}\|\R w\|_{W^{2,\f{1}{\delta}}}\lesim \ltr^{-(1-3\delta)}\|U\|_{X_T}^{2}\\
&&\|I_2\|_{H^{N'}}\lesim\|T_{\f{m}{\phi}}(\R w_0,\R w_0)\|_{H^{N'}}\lesim \|w_0\|_{H^{N'+\f{9}{4}}}^2\|w_0\|_{W^{2,\infty}}\lesim \|w_0\|_{H^{N'+\f{9}{4}}}^2.
\eeno
For $I_3$, we use %bilinear estimate \eqref{bilinear estimate} with $2_{+}=\f{9}{4}$ and \eqref{bilinear estimate 2}
Corollary \ref{cortrilinear} with $3_{+}=\f{13}{4}-2\delta$ and assumption $N'=N-4$ to get:
\beno
&&\|I_3\|_{H^{N'}}\lesim \izt \|T_{\f{m}{\phi}}(\R B(\R w,\R w) ,\R w(s))\|_{H^{N'}}\d s\\
%&\lesim&\izt \|B(\R w,\R w)\|_{W^{N'+\f{9}{4},(\f{1}{2}-\delta)^{-1}}} \|\R w\|_{W^{2,\f{1}{\delta}}}+\|B(\R w,\R w)\|_{W^{1,\f{1}{2\delta}}}\|w\|_{W^{N'+\f{13}{4},(\f{1}{2}-2\delta)^{-1}}}\d s\\
&\lesim&\izt \|w\|_{H^{N'+\f{13}{4}}}\|w\|_{W^{2,\f{1}{\delta}}}^2\d s\lesim\izt \lsr^{-2(1-3\delta)}\|U\|_{X_T}^3\d s
\lesim \|U\|_{X_T}^3.
\eeno
$I_4$ can be estimated in the similar fashion,  %$H=\R U^h \R U^{L}+\R U^{h}\R U^{h}$, we have 
by bilinear estimate and the definition of $H,$
\beno
\|I_4\|_{H^{N'}}&\lesim& \izt\|T_{\f{m}{\phi}}(\R\lnr\chi^{L}H,\R w(s))\|_{H^{N'}}\d s\\&\lesim& \izt \|\lnr H\|_{W^{N'+\f{9}{4},(\f{1}{2}-\delta)^{-1}}}\|w\|_{W^{2,\f{1}{\delta}}}+\|\R \lnr H\|_{W^{2,4}}\|w\|_{W^{N'+\f{9}{4},4}}\d s\\
&\lesim&\izt \|U\|_{L^{\f{1}{\delta}}}\|U\|_{H^{N'+\f{13}{4}+2\delta}}\|w\|_{W^{2,\f{1}{\delta}}}+\|U^h\|_{W^{3,8}}\|U\|_{L^8}\|w\|_{H^{N'+\f{11}{4}}} \d s\\
&\lesim&\izt \lsr^{-2(1-3\delta)}\|U\|_{X_T}^{3}+\lsr^{-(2-7\delta)}\|U\|_{X_T}^{3}\d s\lesim \|U\|_{X_T}^3.
\eeno
\end{subsection}
\section{Conclusion of Theorem \ref{thmlow}}
By collecting the estimates in Section 6 and Section 7,
we find that if $\|U\|_{X_T}\leq \vartheta_2$, there exists three constants $d_1,d_2,d_3$  
 such that for any $\ep\in(0,1]$, any $T>0,$
\beq\label{energyineq}
\|U\|_{X_T}\leq d_{1}\|(u_0,\vr_0,\na \varphi_0)\|_{Y}+d_2 \|U\|_{X_T}^{2}+d_3\|U\|_{X_T}^{3}. 
\eeq
where 
\beno\label{def of initial norm}
\|(u_0,\vr_0,\na\varphi_0)\|_{Y}&\define& \|(u_{0},\vr_{0},\na\varphi_0)^{L}\|_{W^{4,1}}+\|x(u_{0},\vr_{0},\na\varphi_0)^{L}\|_{H^{4+\delta}}\nonumber\\
&&\qquad+\|x(u_{0},\vr_{0},\na\varphi_0)^{h}\|_{L^2}+\|(u_0,\vr_0,\na\varphi_0)\|_{H^N}.
\eeno
Combining with the local existence shown in Section 4, the global existence stems from the standard bootstrap arguments.
Indeed, assume  $\|(u_0,\vr_0,\na\varphi_0)\|_{Y}\leq \bar{\vartheta},$
and set
\beqs
T^{*}=\sup\{T| %(u,\vr,\nabla\varphi)
U \in C([0,T),H^N), \|U\|_{X_T}\leq 2d_1\bar{\vartheta}\}
\eeqs
Suppose that $T^{*}<+\infty,$ then
by \eqref{energyineq}, for any $t< T^{*},$ 
\beq
\|U\|_{X_t}\leq d_1 \bar{\vartheta}+ d_2 (2d_1\bar{\vartheta})^2+d_3(2d_1\bar{\vartheta})^3\leq \frac{3}{2}d_1\bar{\vartheta}
\eeq
if $\bar{\vartheta}$ is chosen small enough, say $\bar{\vartheta}\leq {\vartheta}_3.$ By the time continuity of $X_t$ norm, one gets that:
$\|U\|_{X_{T^{*}}}\leq \frac{3}{2}d_1\bar{\vartheta},$ which contradicts with the local existence and the definition of $T^{*}.$
We thus finish the proof of Theorem \ref{thmlow} by setting $\vartheta_1=\min\{{\vartheta_3},\frac{\vartheta_2}{2d_1}\}.$

%one can refer for instance to Section 4 of \cite{rousset2019stability}.

\section{Proof of Theorem \ref{thmper}}

This section is devoted to the proof of Theorem
\ref{thmper} concerning the life span of system
\eqref{NSPP0}:
 \beqs
 \left\{
\begin{array}{l}
\displaystyle \pt n +\div( \rho v+nu+nv)=0,\\
\displaystyle \pt v+u\cdot {\na v}+v\cdot (\na u+\na v)-\varepsilon \f{1}{\rho+n}\Delta v +\na n -\na \psi
=\varepsilon(\f{1}{\rho+n}-1) \Delta u,   \\
\displaystyle \Delta \psi=n,\\
\displaystyle v|_{t=0} =\mathcal{P}u_0^{\varepsilon},  n|_{t=0}=0,\na\psi|_{t=0}=0.
\end{array}
\right.
\eeqs

\begin{proof}[Proof of Theorem \ref{thmper}]
%Our strategy is to perturbate the system \eqref{NSPEO} by the solution of \eqref{ANSP}. More precisely, we write:
 %$(\rho^{\ep},u^{\ep},\na\phi^{\ep})=(\rho,u,\na\phi)+(n,v,\na\psi)$ where $(\vr,u,\na\phi)$ solves the system \eqref{ANSP} with initial data $(\vr_{0}^{\ep},\mathcal{P}^{\perp}u_{0}^{\ep},\na\phi_0^{\ep})$ where $\mathcal{P}=\na(\Delta)^{-1}\div$ is the Leray projector and $\mathcal{P}^{\perp}=Id-\mathcal{P}$. It remains for us to analyze $(n,v,\na\psi)$ which satisfies the following equations:
 The local existence of the above system in $C([0,T_{\varepsilon},H^3)$ results from the local existence of system \eqref{NSPO} and \eqref{NSPlow}, it thus suffices for us to extend the existence time to $\varepsilon^{-(1-\vartheta)}$ which follows from the energy estimates. 
We define the energy functional:
\beno
\mathscr{E}_{3}=\sum_{|\alpha|\leq 3}\mathscr{E}_{\alpha}=\sum_{|\alpha|\leq 3}\f{1}{2}\int(1+\vr+n)|\p^{\alpha}v|^2+|\p^{\alpha}n|^2+|\p^{\alpha}\na\psi|^2\d x.
\eeno
Taking the time derivative of the above energy functional and using the equations \eqref{NSPP0}, we get:
\ben\label{EIalpha}
\partial_{t}\mathscr{E}_{\alpha}+\ep\int
|\p^{\alpha}\na v|^2+|\pa\div v|^2 \d x=\sum_{j=1}^{9}F_{j},
\een
where:
\beno
F_1&=&-\int \rho^{\ep}\p^{\alpha}v\big[\pa,u+v\big]\na v\d x,\quad\quad F_2=-\int \pa n \big[\pa,\rho^{\ep}\big]\div v\d x,\\
F_3&=&-\int \pa\na\psi\big[\pa,\rho^{\ep}\big]v\d x,\qquad \qquad F_4=-\int\pa n\pa\div(nu)\d x,\\
F_5&=&=\int \pa\na\psi\pa(nu)\d x, \qquad
\qquad \quad F_6=-\int\rho^{\ep}\pa v\pa(v\cdot\na u)\d x,\\
F_7&=&\int\pa n\big(\na\rho^{\ep}\pa v-\pa(\na\rho^{\ep}v)\big)\d x,\quad F_8=\ep\int \rho^{\ep}\pa v\pa\big[(\f{1}{\rho^{\ep}}-1)\Delta u\big],\\
F_9&=&\ep \int \rho^{\ep}\pa v\big[\pa,\f{1}{\rho^{\ep}}\big](\Delta v+\na\div v)\d x.
\eeno
We recall that $\vr^{\ep}=\rho+n=1+\vr+n$.
It is easy to see that: $F_1=F_2=F_3=F_9=0$ if $|\alpha|=0.$
By standard commutator estimates: (we assume $|\alpha|\geq 1$ in the estimates of $F_1,F_2,F_3,F_9$)
\begin{equation*}
\begin{aligned}
|F_1|&\lesim\|v\|_{H^{|\alpha|}}^2(\|\na v\|_{L^{\infty}}+|\na u\|_{W^{|\alpha|-1,\infty}}),\quad
|F_2+F_3|\lesim \|(n,v,\na\psi)\|_{H^{|\alpha|}}^2(\|( v, n)\|_{W^{1,\infty}}+\|\vr\|_{W^{|\alpha|,\infty}}),\\
|F_4|&\lesim \|n\|_{H^{|\alpha|}}^2\|\na u\|_{W^{|\alpha|,\infty}},\qquad \qquad\qquad\qquad
|F_5|\lesim \|(n,\na\psi)\|_{H^{|\alpha|}}\|\na u\|_{W^{|\alpha|-1,\infty}},
\\
|F_6|&\lesim \|v\|_{H^{|\alpha|}}^2\|\na u\|_{W^{|\alpha|,\infty}},\qquad \qquad \qquad\qquad |F_7|\lesim \|(n,v)\|_{H^{|\alpha|}}^2(\|(\na n,\na v)\|_{L^{\infty}}+\|\na \vr\|_{W^{|\alpha|,\infty}}),\\
  |F_8|&\lesim \ep\|v\|_{{H}^{|\alpha|}}\|\Delta u\|_{W^{|\alpha|,\infty}}\|(n,\vr)\|_{H^{|\alpha|}}\\
  &\lesim \|(n,v)\|_{{H}^{|\alpha|}}^2\|\Delta u\|_{W^{|\alpha|,\infty}}+\ep^2\|\Delta u\|_{W^{|\alpha|,\infty}}\|\vr\|_{H^{|\alpha|}}^2\\
|F_9|&\lesim \ep \|v\|_{\dot{H}^{|\alpha|}}\big(\|\na^2 v\|_{\dot{H}^{|\alpha|-1}}\|(\na n,\na\vr)\|_{L^{\infty}}+\|\na^2 v\|_{L^{\infty}}\|(\na n,\na\vr)\|_{\dot{H}^{|\alpha|-1}}\big)\\
&\lesssim\ep  \|\na v\|_{H^3}^2 (\|n\|_{H^3}^2+\|\vr\|_{H^3}^2).
\end{aligned}
\end{equation*}
%\begin{equation*}
%\begin{aligned}
%\end{aligned}
%\end{equation*}
 We only detail the estimate of $F_5$, which seems not direct. Indeed, by the Poisson equation $\Delta\psi=n$, we have:
 \beno
 |F_5|&=&\int \pa\na\psi\cdot u \pa n\d x+\int \pa\na\psi[\pa,u]n\d x \\
 &=&-\int\pa\na\psi \cdot\na (\pa\na\psi\cdot u)\d x+\int \pa\na\psi[\pa,u]n\d x \\
 &=&\f{1}{2}\int|\pa\na\psi|^2\div u\d x-\int\pa\na\psi\cdot(\pa\na\psi\cdot u)\d x+\int \pa\na\psi[\pa,u]n\d x \\
 &\lesim&  \|(n,\na\psi)\|_{H^{|\alpha|}}^2\|\na u\|_{W^{|\alpha|-1,\infty}}.
 \eeno
Define $$T_{*}=\sup_{T}\left\{T\big|\sup_{0\leq t\leq T}\mathscr{E}_3(t)\leq 4\vartheta^{2}\ep^{2-\vartheta}\right\},$$
where $0< \vartheta\leq\vartheta_0\leq \vartheta_1,$ $\vartheta_1$ is defined in Theorem \eqref{thmlow} and $\vartheta_0$ is to be chosen.

Summing up the above estimates for any $|\alpha|\leq 3$, we obtain that, there exist three constants $C_2>0,C_3>0,C_4>0$, for any $0<\vartheta\leq \vartheta_1$, if
$\|(\rho_0^{\ep}-1,\mathcal{P}^{\perp}u_{0}^{\ep},\na\varphi_0^{\ep})\|_{Y^4}\leq \f{\vartheta}{C_3C_1}$  (which yields $(1+t)^{}\|(\na u,\vr)\|_{W^{4,\infty}}+(\na u,\vr)\|_{H^7}\leq \f{\vartheta}{C_3}$ by Theorem \ref{thmlow}),
%are three constants which are independent of $\ep$: $C_5,C_4,C_3$, 
such that the following energy inequality holds:
\ben
\pt \mathscr{E}_3+\varepsilon\|\nabla v\|_{H^3}^2\leq C_2 \mathscr{E}_{3}^{\f{3}{2}}+(1+t)^{-1}%\|(u,\vr,\na\phi)\|_{X}^3
\vartheta^3\ep^2+(1+t)^{-1}\vartheta
\mathscr{E}_3+C_4\varepsilon \|\nabla v\|_{H^3}(\frac{\vartheta}{C_3}+\mathcal{E}_3)
%\|(u,\vr,\na\phi)\|_{X}
%\delta_1 
.
\een

Now one can choose $\vartheta_0$   small enough, such that for any $0\leq t < T_{*},$
\begin{equation*}
   \|(n,\vr)\|_{L^{\infty}}\leq \f{1}{4}, \qquad
C_4(\frac{\vartheta}{C_3}+\mathcal{E}_3)\leq \frac{1}{2} 
\end{equation*}
which leads to:
 \ben\label{enineq}
\pt \mathscr{E}_3\leq C_2 \mathscr{E}_{3}^{\f{3}{2}}+(1+t)^{-1}%\|(u,\vr,\na\phi)\|_{X}^3
\vartheta^3\ep^2+(1+t)^{-1}\vartheta
\mathscr{E}_3\qquad \forall t\in[0,T_{*})
%\|(u,\vr,\na\phi)\|_{X}
%\delta_1 
.
\een

We are now ready to show $T_{*}\geq  \ep^{-(1-\vartheta)}.$
Indeed, for any $t\leq T_0\define \min\{\ep^{-(1-\vartheta)}, T_{*}\}$,
one has by \eqref{enineq}, Gronwall's inequality and assumptions: $\vartheta\leq\f{1}{2}$, $16C_2^{\f{3}{2}}\vartheta\leq 1$, $\mathscr{E}_3(0)\leq\vartheta^2\ep^2:$
\ben\label{bootstrap}
\mathscr{E}_3(t)&\leq& e^{\izt \vartheta(1+\tau)^{-1}\d \tau}\mathscr{E}_{3}(0)+\izt 
e^{\int_{s}^{t} \vartheta(1+\tau)^{-1}\d \tau}\big(C_2\mathscr{E}_{3}^{\f{3}{2}}+\vartheta^{3}(1+s)^{-1}\ep^2\big)\d s\nonumber\\
&\leq&(1+t)^{\vartheta}\mathscr{E}_{3}(0)+(1+t)^{\vartheta}\izt (1+s)^{-\vartheta}\big(C_2(4\vartheta^2\ep^{2-\vartheta})^{\f{3}{2}}+\vartheta^{3}(1+s)^{-1}\ep^2\big)\d s\nonumber\\
&\leq&2^{\vartheta}\ep^{-\vartheta(1-\vartheta)}\vartheta^2
\ep^2+\f{8}{1-\vartheta}\ep^{-(1-\vartheta)}C_2^{\f{3}{2}}\vartheta^3\ep^{3-\f{3}{2}\vartheta}+\vartheta^3\ep^2 \ep^{-\vartheta(1-\vartheta)}
\leq\f{7}{2}\vartheta^2\ep^{2-\vartheta},
\een
which ensures $T_0=\ep^{1-\vartheta}<T_{*}.$ Note that since $\f{1}{2}\leq\rho_{0}^{\ep}\leq\f{3}{2}$, the assumption
$\|(n,v,\na\psi)\|_{H^3}\leq \vartheta\ep$ leads to 
$\mathscr{E}_{3}(0)\leq \vartheta^2\ep^2.$ We thus finish the proof of Theorem \ref{thmper} by choosing $C=C_1C_3.$
\end{proof}
%therefore, there exists a solution in $C([0,\ep^{1-\delta_1}],H^3)$.

\section{Appendix}
We sketch in this appendix the proof of the decay and weighted norm of $\R I_3\define \R e^{-itb(D)}I_3'$: $W^{4+2\delta,2_{\delta}}$ (recall $2_{\delta}=\f{2}{1+\delta/5}$) norm of $Z_1,Z_2,Z_3$. Let us begin with the estimate $Z_2$ which is the easiest.
%we thus write $\lnr^{3+2\delta}xe^{itb(D)}I_3$ into three terms:
%\beno
%\lnr^{3+2\delta}xe^{itb(D)} I_3&=& \lnr^{3+2\delta}x\int_{0}^{t}e^{is(\Im\phi)(D)}e^{\ep(t-s)\Delta}T_{\f{m}{\phi}}(e^{-isb(D)}\R \lnr (\R w)^{2},\R f)\d s\\
%&=&\int_{0}^{t}\int s e^{isb(D)}e^{\ep(t-s)\Delta}T_{\f{m}{\phi}\lxr^{3+2\delta}(i\partial_{\xi}(\Im\phi)(\xi,\eta))} (\R \lnr (\R w)^{2},\R w) \d s\\
%&&+\int_{0}^{t}e^{isb(D)}e^{\ep(t-s)\Delta}
%T_{\partial_{\xi}(\f{m}{\phi})\lxr^{3+2\delta}}(\R \lnr (\R w)^{2},\R w)\d s\\
%&&+\int_{0}^{t}e^{isb(D)}e^{\ep(t-s)\Delta}T_{\f{m}{\phi}\lxr^{3+2\delta}}(\R \lnr (\R w)^{2}, e^{-isb(D)}x\R f)\d s\\
%&&{\color{red}+\izt \ep(t-s)\na e^{(t-s)\Delta}e^{isb(D)}T_{\f{m}{\phi}}(\R \lnr (\R w)^{2},\R w)\d s}\\
%&\define&Z_{1}+Z_{2}+Z_{3}+Z_4
%\eeno
By Lemma \ref{Lp bounds for eitb(D)}, Corollary \ref{cortrilinear} and Sobolev embedding,
\beno
\|Z_2\|_{W^{4+2\delta,2_{\delta}}}&\lesim&\int_{0}^{t}\lsr^{\f{\delta}{5}}\|T_{\partial_{\xi}(\f{m}{\phi})}(\R \lnr (\R w)^{2},\R w)\|_{W^{4+2\delta,2_{\delta}}}\d s\\
&\lesim&\int_{0}^{t}\lsr^{\f{\delta}{5}} \| w\|_{H^8}\|w\|_{W^{2,12}}^2\d s\lesim\int_{0}^{t}\lsr^{\f{\delta}{5}}\lsr^{\f{5}{3}}\d s\|U\|_{X}^3\lesim \|U\|_{X_T}^3.
\eeno
For $Z_3$, we split it into two terms:
\beno
Z_3&=&\int_{0}^{t}e^{isb(D)}e^{-\ep(t-s)\Delta}T_{\f{m}{\phi}(\chi_{\{\lxmer\leq \ler\}}+\chi_{\{\lxmer\geq \ler\}})}(\R\lnr (\R w)^{2}, e^{-isb(D)}x\R f)\d s\\
&\define&Z_{31}+Z_{32}
\eeno
{%\color{red}
The estimate of $Z_{31}$ is similar to that of $Z_2$, as we can put all the loss of derivative on the term $B(\R w,\R w)$.
Applying lemma \ref{Lp bounds for eitb(D)} with $p=2_{\delta}=\f{2}{1+\delta/5}$, bilinear estimate \eqref{bilinear estimate}, Lemma \ref{inter1} with $7.5\leq \f{2}{3}N+\f{1}{2}$}, we control $Z_{31}$ as:
\beno
\|Z_{31}\|_{W^{4+2\delta,2_{\delta}}}&\lesim& \int_{0}^{t}\lsr^{\f{\delta}{5}}\|T_{\f{m}{\phi}\chi_{\{\lxmer\leq \ler\}}}(\R \lnr (\R w)^{2}, e^{-isb(D)}x\R f)\|_{W^{4+3\delta,2_{\delta}}}\d s\\
&\lesim&\int_{0}^{t}\lsr^{\f{\delta}{5}}\|\R w\|_{L^{\infty}}\|\R w\|_{W^{7+4\delta,\f{10}{1+\delta}}}\|e^{-isb(D)}x\R f\|_{W^{2,\f{5}{2}}}\d s\\
&\lesim& \int_{0}^{t}\lsr^{\f{\delta}{5}}\|\R w\|_{W^{1,\f{1}{\delta}}}\|\R w\|_{W^{7.5,3}}\|\langle x\rangle f\|_{W^{3,\f{2}{\delta}}}\lsr^{\f{1}{5}}\d s\\
%&\lesim&\int_{0}^{t}\lsr^{\f{1}{5}+\delta}\|w\|_{W^{1,10}}\|w\|_{W^{8,4}}\|\langle x\rangle f\|_{W^{3,\f{2}{\delta}}}\d s\\
&\lesim&\int_{0}^{t}\lsr^{\f{1+\delta}{5}}\lsr^{-(1-2\delta)}\lsr^{-\f{1}{3}}\d s\|U\|_{X_T}^{3}\lesim\|U\|_{X_T}^{3}.
\eeno

For $Z_{32}$, one splits it again into two terms:
$$Z_{32}=%\cF^{-1}(\int_{0}^{t}\chi_{\{|\xi-\eta|\leq\lsr^{\delta_0}\}}+\chi_{\{|\xi-\eta|>\lsr^{\delta_0}\}})\cdots \d s\define Z_{321}+Z_{322}
\cF^{-1}\bigg(\int_{0}^{t}\big(\Psi(\f{\xi-\eta}{\lsr^{\delta_0}})+1-\Psi(\f{\xi-\eta}{\lsr^{\delta_0}})\big)\cdots\d s\bigg)\define Z_{321}+Z_{322}$$.

For $Z_{321}$, %as one projects $\max\{|\xi-\eta|, |\eta|\}$ onto the low frequency ($\lesim \lsr^{\delta_0}$), 
Corollary \ref{cortrilinear} and Sobolev embedding lead to :
\beno
\|Z_{321}\|_{W^{4+2\delta,2_{\delta}}}&\lesim&\int_{0}^{t}\lsr^{\fd}\|T_{\f{m}{\phi}\chi_{\{\lxmer\geq \ler\}}}(\R \lnr (\R w)^{2}), P_{\leq \lsr^{\delta_0}}e^{-isb(D)}x\R f)\|_{W^{4+3\delta,2_{\delta}}}\d s\\
&\lesim&\int_{0}^{t}\lsr^{\fd}
\| w\|_{W^{2,12}}^2\| P_{\leq \lsr^{\delta_0}}e^{-isb(D)}x\R f)\|_{W^{8,3}}\d s\\
&\lesim&\int_{0}^{t}\lsr^{\fd}\|w\|_{W^{2,12}}^2\lsr^{5\delta_0}\lsr^{\f{1}{3}}\|x\R f\|_{W^{3,3}}\d s\\
&\lesim&\int_{0}^{t}\lsr^{\fd}\lsr^{-\f{5}{3}}\|U\|_{X_T^{2}}\lsr^{\f{1}{3}+5\delta_0}\|U\|_{X_T}\d s\lesim\|U\|_{X_T}^3.
\eeno
Note we can  choose $\delta<\delta_0\leq \f{1}{50}$.

For $Z_{322}$, we define
$M_{1}(\xi,\eta)=\f{m}{\phi}%\lxr^{3+2\delta}
(1-\Psi(\f{\xi-\eta}{\lsr^{\delta_0}}))\chi_{\{\lxmer>\ler\}}\ler $
and recall $\tphi=b(\xi)\pm b(\xi-\eta)\pm b(\eta-\si)\pm b(\si) $.
Using identity $\p_{\xi}\widehat{\R f}(\xi-\eta)=-\p_{\eta}\widehat{\R f}(\eta)$ to integrate by parts in $\eta,$ one gets:
\beno
Z_{322}&=&\cF^{-1}\big( \int_{0}^{t}\int\int is\partial_{\eta}\tphi e^{is\tphi}e^{-\ep(t-s)|\xi|^{2}}M_1(\xi,\eta)\R(\eta)\widehat{\R f}(\eta-\si)\widehat{\R f}(\si)\widehat{\R f}(\xi-\eta)\d \eta\d\si\d s\\
&&+\int_{0}^{t}\int\int
e^{is\tphi}e^{-\ep(t-s)|\xi|^{2}}\partial_{\eta}M_1(\xi,\eta)\R(\eta)\widehat{\R f}(\eta-\si)\widehat{\R f}(\si)\widehat{\R f}(\xi-\eta)\d \eta\d\si\d s\\
&&+\int_{0}^{t}\int\int
e^{is\tphi}e^{-\ep(t-s)|\xi|^{2}}M_1(\xi,\eta)|\eta|^{-1}\widehat{\R f}(\eta-\si)\widehat{\R f}(\si)\widehat{\R f}(\xi-\eta)\d \eta\d\si\d s\\
&&+\int_{0}^{t}\int\int
e^{is\tphi}e^{-\ep(t-s)|\xi|^{2}}M_1(\xi,\eta)\R(\eta)\widehat{x\R f}(\eta-\si)\widehat{\R f}(\si)\widehat{\R f}(\xi-\eta)\d \eta\d\si\d s\big)\\
&\lesim&Z_{3221}+\cdots Z_{3224}
\eeno
The estimation of 
$Z_{3222}=\int_{0}^{t}e^{isb(D)}e^{\ep(t-s)\Delta}T_{\partial_{\eta}M_1}
(\R(\R w)^2,\R w)\d s$ is similar to that of $Z_2$, we thus do not detail it.
For $G_{3223}$, thanks to Corollary \ref{bilinear estimate} and Hardy-Littlewood-Sobolev inequality,
\beno
\|Z_{3223}\|_{W^{4+2\delta,2_{\delta}}}&\lesim&\int_{0}^{t}\lsr^{\fd}\|T_{M_1%\lxr^{2\delta} 
}(|\na|^{-1}(\R w)^{2},P_{\geq \lsr^{\delta_0}}\R w)\|_{W^{4+3\delta,2_{\delta}}}\d s\\
&\lesim&\lesim\int_{0}^{t}\lsr^{\fd}
\|P_{\geq \lsr^{\delta_0}}\R w)\|_{H^8}\|(\R w)^2\|_{W^{2,{2_{\delta}}}}\d s\\
&\lesim& \int_{0}^{t}\lsr^{\fd}\lsr^{-(N-8)\delta_0+\delta}\|w\|_{H^{N}}\lsr^{-(1-\f{\delta}{5})}\d s\|U\|_{X_T}^{2}\lesim\|U\|_{X_T}^{3}.
\eeno
Note $N= 11$, the last inequality holds if we choose $\delta\leq \delta_0$.

The term $Z_{3224}$ can be estimated in the similar manner as that of $Z_{31}$, we omit the detail. 
The estimate of $Z_{3221}$ is similar to that of $Z_1$, we thus only focus on the estimate of $Z_1$ in the following.

As before, we split it as:
$${Z_1}=\cF^{-1}\big(\int_{0}^{t}(\Psi(\f{\xi-\eta}{\lsr^{\delta_0}})+1-\Psi(\f{\xi-\eta}{\lsr^{\delta_0}}))\cdots\d s\big)\define Z_{11}+Z_{12}.$$
Split $Z_{12}$ further as:
$$Z_{12}=\int_{0}^{t}se^{isb(D)}e^{\ep(t-s)\Delta} T_{M_2}((P_{\geq \lsr^{\delta_0}}+P_{\leq \lsr^{\delta_0}})(\R(\R w)^{2}, P_{\geq \lsr^{\delta_0}}\R w)\d s$$
where $M_2=\f{m}{\phi}\ler(i\partial_{\xi}\tilde{\phi})$. Therefore, by bilinear estimate lemma \ref{lem bilinear estimate} and the Sobolev embedding,%{\color{red}explain more} %$H^{\f{25}{4}+6\delta}\hookrightarrow W^{\f{25}{4}+2\delta,(\f{1}{2-\delta}-2\delta)^{-1}}$ 
\beno
&&\|Z_{12}\|_{W^{4+2\delta,2_{\delta}}}\lesim \int_{0}^{t}s\lsr^{\fd}\big(\|\R(\R w)^{2}\|_{W^{2,\f{1}{2\delta}}}^2\|P_{\geq \lsr^{\delta_0}}\R w\|_{H^7}+\|P_{\geq \lsr^{\delta_0}}\R(\R w)^{2}\|_{H^{7}}\|\R w\|_{W^{2,\f{1}{\delta}}}\big)\d s\\
&\lesim&\int_{0}^{t}s\lsr^{\fd}\|\R w\|_{W^{2,\f{1}{\delta}}}^2\lsr^{-(N-7)\delta_0}\|w\|_{H^{N}}\d s\lesim\int_{0}^{t}s\lsr^{\fd}\lsr^{-2(1-2\delta)}\lsr^{-(N-7)\delta_0+\delta}\d s \|U\|_{X_T}^{3}\lesim \|U\|_{X_T}^{3}.
\eeno
if we assume $N=11$, $\delta_0\geq 3\delta.$

For $Z_{11}$, one uses decomposition
$g^2=gP_{\geq \lsr^{\delta_0}}g+P_{\leq \lsr^{\delta_0}}g P_{\geq \lsr^{\delta_0}}g+P_{\leq \lsr^{\delta_0}}g P_{\leq \lsr^{\delta_0}}g$
to split it into three terms and denote as $Z_{111}+Z_{112}+Z_{113}$.
For the term $Z_{111}$, one can use bilinear estimate \eqref{bilinear estimate} %with$2_{+}=\f{9}{4}$, 
Sobolev embedding and the spectral localization of each term to get:
\beno
\|Z_{111}\|_{w^{4,2_{\delta}}}%&=&\int_{0}^{t}s \lsr^{\fd}\| T_{M_2}(\R(P_{\geq \lsr^{\delta_0}}\R w\cdot\R w),P_{\leq \lsr^{\delta_0}}\R w)\|_{L^{2-\delta}}\d s\\
&\lesim&\izt \lsr^{1+\fd} (\|P_{\leq \lsr^{\delta_0}}\R w\|_{W^{7,\f{1}{\delta}}}\|P_{\geq \lsr^{\delta_0}}\R w\cdot\R w\|_{W^{2,\f{10}{5-9\delta}}}\\
&&\qquad +\|P_{\leq \lsr^{\delta_0}}\R w\|_{W^{2,\f{1}{\delta}}}\|P_{\geq \lsr^{\delta_0}}\R w\cdot\R w\|_{W^{\f{25}{4}+2\delta,\f{10}{5-9\delta}}}\d s\\
&\lesim& \izt \lsr^{1+\fd}\|P_{\leq \lsr^{\delta_0}}\R w\|_{W^{2,\f{1}{\delta}}}\|\R w\|_{L^{\f{1}{\delta}}}\big(\lsr^{5\delta_0}\|P_{\geq \lsr^{\delta_0}}\R w\|_{H^{3}}+\|P_{\geq \lsr^{\delta_0}}\R w\|_{H^{8}}\big)\d s\\
&\lesim&\izt \lsr^{1+\fd}\lsr^{-2(1-2\delta)}\lsr^{-(N-8)\delta_0+\delta}\d s\|U\|_{X_T}^{3}\lesim \|U\|_{X_T}^{3}
\eeno
as $N= 11$, and $\delta_0\geq 5\delta$.
The estimate of $Z_{112}$ is similar to that of $Z_{111}$, we thus skip it.
Now, we focus on the estimation of
\beno
Z_{113}&=&\izt se^{isb(D)}e^{\ep(t-s)\Delta} T_{M_2}(\R(P_{\leq \lsr^{\delta_0}}\R w P_{\leq \lsr^{\delta_0}}\R w, P_{\leq \lsr^{\delta_0}}\R w)\d s\\
&=& \cF^{-1}\big(\izt s \partial_{\xi}\tphi e^{is\tphi}e^{\ep(t-s)\Delta}\f{m}{\phi}(\xi,\eta)%\lxr^{4+2\delta}
\R(\eta)\ler\Phi_{\leq \lsr^{\delta_0}}(\eta-\si)\Phi_{\leq \lsr^{\delta_0}}(\si) \\
&& \qquad \Phi_{\leq \lsr^{\delta_0}}(\xi-\eta) \widehat{\R f)}(\eta-\si)\widehat{\R f)}(\si)\widehat{\R f)}(\xi-\eta)\d\si\d\eta\d s\big)
\eeno
Recall that $\tphi=b(\xi)\pm b(\xi-\eta)\pm b(\eta-\si)\pm b(\si)$.
For the case $'+++','++-','+-+','-++','---'$,
one could easily prove that:
$\f{1}{\tilde{\phi}}>_{\kpz}\min \{\lxr,\lxmer,\lemsr,\lsir\}.$  
Therefore, we shall use the identity $\f{1}{i\tilde{\phi}}\p_{s}{e^{is\tilde{\phi}}}=e^{is\tilde{\phi}}$
to integrate by parts in time, which gives us essentially the quartic terms that are easy to handle, we do not detail them. The remaining three terms $'--+','-+-','+--'$ are more involved, we will detail the $'+--'$ case for instance. The other two cases are
a little bit easier.
Firstly, for
$\tphi=b(\xi)+ b(\xi-\eta)- b(\eta-\si)- b(\si)$, we have $\partial_{\xi}\tphi=\f{(1-2\ep^{2}|\xi|^{2})\xi}{b(\xi)}+ \f{(1-2\ep^{2}|\xi-\eta|^{2})(\xi-\eta)}{b(\xi-\eta)}$.
In this case, by Lemma \ref{lemvector} below, one can find  two matrices $Q_1,Q_2$,
st, $\partial_{\xi}\tphi(\xi,\eta,\si)=-2Q_{1}(\xi,\eta)\partial_{\eta}\tphi-Q_2(\xi,\eta,\si)\partial_{\si}\tphi$, and $Q_j(j=1,2)$
satisfy the condition:
\beqs
|\pab\partial_{\si}^{\gamma}Q_j(\xi,\eta,\si)|\lesim_{\alpha,\beta,\gamma,\kpz} \langle|\xi|+|\eta|+|\si|\rangle^{3}.
\eeqs
We now split $Z_{113}$ again:
$$Z_{113}=\cF^{-1}(\izt \int (\Psi_{\geq \lsr^{-\delta_0}}(\eta)+\Psi_{\leq \lsr^{-\delta_0}}(\eta))\cdots\d\eta\d s )\define Y_1+Y_2.$$
Let us see $Y_1$. In this case,  there is no singularity for $\R (\eta)$ as $\eta$ does not vanish. Therefore, we could use the identity:
$$is\partial_{\xi}\tphi e^{is\tphi}=-2Q_1(\xi,\eta,\si)\partial_{\eta}(e^{is\tphi})-Q_{2}(\xi,\eta,\si)\partial_{\si}(e^{is\tphi})$$
and integrate by parts in $\eta$ and $\si$
respectively. We only detail the situation of integration by parts in $\eta$ as the other case is similar. To continue, we denote
$m_{j}(\xi,\eta,\si)=Q_j\f{m}{\phi}(\xi,\eta)%\lxr^{3+2\delta}
\R(\eta)\ler\Psi_{\leq \lsr^{\delta_0}}(\xi-\eta)\Psi_{\leq \lsr^{\delta_0}}(\eta-\si)\Psi_{\leq \lsr^{\delta_0}}(\si)\Psi_{\geq \lsr^{-\delta_0}}(\eta)\R(\eta)\ler$.

After integrating by parts in $\eta$, there are two terms to be estimated:
\beno
Y_{11}&=&\izt e^{isb(D)}e^{\ep(t-s)\Delta} T_{\partial_{\eta}m_{1}}(\R w,\R w,\R v)\d s\\
Y_{12}&=&\izt e^{isb(D)}e^{\ep(t-s)\Delta}T_{m_1}(e^{isb(D)}x\R f,\R w,\R w)\d s+ similar \quad term.
\eeno
However, these two terms can be easily treated once we have the following lemma.
\begin{lem}\label{trilinear estimate}
$m_j(\xi,\eta,\si), j=1,2$ is defined as follows, the following trilinear estimates hold:
 \beno
&&\|T_{m_j}(f,g,h)\|_{L^{p}}\lesim \lsr^{12\delta_0}\|f\|_{L^{p_1}}\|g\|_{L^{p_2}}\|h\|_{L^{p_3}},\\
&& \|T_{\partial_{\eta}m_j}(f,g,h)\|_{L^{p}}\lesim \lsr^{13\delta_0}\|f\|_{L^{p_1}}\|g\|_{L^{p_2}}\|h\|_{L^{p_{3}}},
\eeno
where $\f{1}{p}=\f{1}{p_1}+\f{1}{p_2}+\f{1}{p_3}$
\end{lem}

\begin{proof}
It is not hard to check that:
$\pab\partial_{\si}^{\gamma}m_{1}\lesim \lsr^{\delta_0|\beta|}\lsr^{5\delta_0},$
which yields
\beno
&&\|\cF^{-1}(m_{j}(\xi,\eta,\si))\|_{L^{1}}\lesim \|(1+\p_{\xi}^{4}+\p_{\eta}^{4}+\p_{\si}^{4})m_j\|_{L^2}\\
&\lesim&(\int_{|(\xi,\eta,\si)|\lesim \lsr^{\delta_0}}|(1+\p_{\xi}^{4}+\p_{\eta}^{4}+\p_{\si}^{4})m_j|^{2}\d\xi\d\eta\d\si)^{\f{1}{2}}\lesim\lsr^{9\delta_0}\lsr^{3\delta_0}\lesim \lsr^{12\delta_0}
\eeno
Similarly, $\|\cF^{-1}(\p_{\eta}m_{j}(\xi,\eta,\si))\|_{L^{1}}\lesim \lsr^{13\delta_0}$.
We thus finish the proof by noticing the explicit formulae of $T(f,g,h)$
$$T_{m_j}(f,g,h)=\int\int\int \cF^{-1}(m_{j})(y.z-y,w-z)f(x-y)g(x-z)h(x-w)\d y \d z\d w$$
\end{proof}
We now estimate $Y_{12}$ for example,
by the spectral localization and Lemma
\ref{Lp bounds for eitb(D)},
\beno
\|Y_{12}\|_{W^{4+2\delta,2_{\delta}}}&\lesim& \int_{0}^{t} \lsr^{\fd}\lsr^{(4+3\delta)\delta_0} \|T_{m_{1}\lxr^{\delta}}(e^{isb(D)}x\R f,\R w,\R w)\|_{L^{2_{\delta}}}\d s\\
&\lesim& \int_{0}^{t}\lsr^{17\delta_0}\|e^{isb(D)}x\R f\|_{L^{(\f{1}{2_{\delta}}-2\delta)^{-1}}}\|\R w\|_{L^{\f{1}{\delta}}}^{2}\d s\\
&\lesim&\izt \lsr^{17\delta_0}\lsr^{-2(1-2\delta)}\lsr^{4\delta}\d s\|U\|_{X_T}^{3}\lesim \|U\|_{X_T}^{3}. 
\eeno
where the folowing fact has been used:
$$\|e^{isb(D)}x\R f\|_{L^{(\f{1}{2_{\delta}}-2\delta)^{-1}}}\lesim \lsr^{4\delta}\|x\R f\|_{W^{\f{1}{2},(\f{1}{2_{\delta}}-2\delta)^{-1}}}\lesim \lsr^{4\delta} \|xf\|_{W^{1,\f{2}{1-\delta}}}.$$

 We now go back to estimate $Y_2,$
as before, we split it into two terms:
$$Y_2=\cF^{-1}(\izt \int \big(\Psi_{\leq 3\lsr^{-\delta_0/5}}(\xi-\eta)+\Psi_{\geq 3\lsr^{-\delta_0/5}}(\xi-\eta)\big)\cdots\d\eta\d s)\define Y_{21}+Y_{22}$$
For $Y_{21}$, one can use the specific form of 
$\p_{\xi}\tphi=\f{1-2\ep|\xi|^2}{b(\xi)}\xi+\f{1-2\ep |\xi-\eta|^{2}}{b(\xi-\eta)}(\xi-\eta)$.
The observation is that, having projected to the low frequency for $\xi-\eta$ and $\eta$, we could make use of $\xi-\eta$ and $\eta$ appearing in $\p_{\xi}\tphi$.
We also recall that $(1-2\ep |\xi|^{2})$ here is bounded on the support of $\chi^L(\xi)$.
Formally, we could write $Y_{21}$ as
\beq\label{Y_{211}}
\izt is e^{isb(D)}e^{\ep(t-s)\Delta}\na T_{\f{m}{\phi}%\lxr^{3+2\delta}
\f{1-2\ep|\xi|^2}{b(\xi)}\ler}(\R P_{\lesim \lsr^{-\delta_0}}
(P_{\lesim \lsr^{\delta_0}}\R w)^{2},
P_{\lesim 3\lsr^{-\delta_0/5}}\R w)\d s
\eeq
and similar term.
%Having projected $\eta,\xi$ onto the low frequency ($|\xi-\eta|\lesim \lsr^{-\delta_0/5},|\eta|\lesim \lsr^{-\delta_0}$), 
One can estimate \eqref{Y_{211}} as follows:
\beno
\|\ref{Y_{211}}\|_{W^{4+2\delta,2_{\delta}}}&\lesim&\izt \lsr^{1+\delta}\lsr^{-\f{1}{5}\delta_0}\|\R w\|_{W^{2,\f{1}{\delta}}}^{2}\|\R w\|_{W^{8,(\f{1}{2_{\delta}}-2\delta)^{-1}}}\d s\\
&\lesim& \izt \lsr^{1+\delta -\f{1}{5}\delta_0}\lsr^{-2(1-2\delta)}\d s\|U\|_{X_T}^{3}\lesim \|U\|_{X_T}^{3}
\eeno
if we choose $\delta_0\geq 50\delta.$

For $Y_{22}$, we need to split again into two terms.
$$Y_{22}=\cF^{-1}(\izt \int\int \big(\Psi_{\geq 2\lsr^{-\delta_0}}(\si)+\Psi_{\leq 2\lsr^{-\delta_0}}(\si)\big)\cdots\d\si\d\eta\d s)\define Y_{221}+Y_{222}$$

Let us see $Y_{221}$,
in this case we have $|\si-\eta|>|\si|-|\eta|>\f{4}{9}|\si|$.
Besides, one can find a matrix $Q_3$, such that: 
$$\p_{\si}\tphi=-[\f{1-2\ep|\si|^2}{b(\si)}\si +\f{1-2\ep|\si-\eta|^2}{b(\si-\eta)}(\si-\eta)]=Q_3(2\si-\eta)$$
so we have: $|\p_{\si}\tphi|\gtrsim\|Q_3^{-1}\|^{-1}|2\si-\eta|\gtrsim\f{|\si|}{\lsir \lemsr}\gtrsim \lsr^{-3\delta_0}.$
We thus could use identity $e^{is\tphi}=\f{\p_{\si}\tphi\cdot\p_{\si}}{is|\p_{\si}\tphi|^2}\cdot\p_{\si}e^{is\tphi}$ and integrate by parts in $\si$, this leads to two terms:
\beno
Y_{2211}&=&\izt e^{isb(D)}e^{\ep(t-s)\Delta}T_{m_{3}}
(\R T_{\p_{\si}\tilde{m}}(\R w)^{2},\R w )\d s\\
Y_{2212}&=&\izt e^{isb(D)}e^{\ep(t-s)\Delta}T_{m_{3}}
(\R T_{\tilde{m}}(e^{isb(D)}x\R f,\R w),\R w )\d s\\
\eeno
where we denote
\begin{equation*}
\begin{aligned}
    m_3(\xi,\eta)&=\f{m}{\phi}%\lxr^{3+2\delta}
\ler \Psi_{\leq \lsr^{-\delta_0}}(\eta)\Psi_{\leq \lsr^{\delta_0}}(\xi-\eta)\Psi_{\geq 3\lsr^{-\delta_0/5
}}(\xi-\eta),\\
\tilde{m}(\eta,\si)&=\f{\p_{\si}\tphi}{|\p_{\si}\tphi|^2}\Psi_{\geq 2\lsr^{-\delta_0}}(\si)\Psi_{\leq \lsr^{\delta_0}}(\si)\Psi_{\leq \lsr^{\delta_0}}(\eta-\si)\tilde{\chi}^{L}(\eta-\si)\tilde{\chi}^{L}(\si). 
\end{aligned}
\end{equation*}
Similar to Lemma \ref{trilinear estimate}, one has the following inequality:
\beno
&&\|T_{m_3}(f,g)\|_{L^p}\lesim \lsr^{6\delta_0}\|f\|_{L^{p_1}}\|g\|_{L^{p_2}}\\
\|T_{\tilde{m}}(u,v)\|_{L^p}&\lesim& \lsr^{14\delta_0}\|u\|_{L^{p_1}}\|v\|_{L^{p_2}}\quad \|T_{\p_{\si}\tilde{m}}(u,v)\|_{L^p}\lesim \lsr^{17\delta_0}\|u\|_{L^{p_1}}\|v\|_{L^{p_2}}
\eeno
for any $1\leq p\leq \infty$ and $\f{1}{p}=\f{1}{p_1}+\f{1}{p_2}$.
 These inequalities in hand, the estimates of $Y_{2211}, Y_{2212}$ are direct. For example:
\beno
\|Y_{2212}\|_{W^{4+2\delta,2_{\delta}}}&\lesim&\izt
\lsr^{\fd}\lsr^{(4+3\delta)\delta_0}\lsr^{20\delta_0}\|e^{isb(D)}x\R f\|_{L^{(\f{1}{2_{\delta}}-2\delta)^{-1}}}\|\R w\|_{L^{\f{1}{\delta}}}^{2}\d s\\
&\lesim&\izt
\lsr^{25\delta_0}\lsr^{-2(1-2\delta)}\d s\|U\|_{X_T}^3\lesim \|U\|_{X_T}^3
\eeno
Note we could choose $\delta<\delta_0 \leq \f{1}{50}$.

Finally, it remains for us to estimate $Y_{222}$. In this case, there is no structure for $\p_{\xi}\tphi$ can be used. Nevertheless, as noted in \cite{MR3274788}, one can
employ kind of 'partial normal form'.
We notice that:
$b(\eta)+b(\si)-2=\f{1-\ep^2|\eta-\si|^{2}}{b(\eta-\si)+1}|\eta-\si|^{2}+\f{1-\ep^2|\si|^{2}}{b(\si)+1}|\si|^{2}$,  the observation is that we could use $|\eta-\si|^{2}$ and $|\si|^2$ appearing in this quantity.
On the other hand, $b(\xi)+b(\xi-\eta)-2=\f{(1-\ep^2|\xi|^2)}{b(\xi)+1}|\xi|^2+\f{(1-\ep^2|\xi-\eta|^2)}{b(\xi-\eta)+1}|\xi-\eta|^2\geq \f{(1-\ep^2|\xi-\eta|^2)}{b(\xi-\eta)+1}|\xi-\eta|^2\geq \lsr^{-\delta_0/5}$.
We thus use identity:
\beno
e^{is\tphi}&=&e^{is(b(\xi)+b(\xi-\eta)-2)}e^{-is(b(\eta)+b(\eta-\si)-2)}\\
&=& \f{-i}{b(\xi)+b(\xi-\eta)-2}\p_{s}(e^{is(b(\xi)+b(\xi-\eta)-2)})e^{-is(b(\eta)+b(\eta-\si)-2)}
\eeno
to integrate by parts in $s$:
\beno
Y_{222}&=&-te^{itb(D)}T_{m_{4}}(\R T_{\tilde{m}_{1}}(\R w,\R w),\R w)+\izt s e^{isb(D)}e^{\ep (t-s)\Delta}\ep\Delta T_{m_{4}}(\R T_{\tilde{m}_{1}}(\R w,\R w),\R w)\d s
\\
&&+i\izt s e^{isb(D)}e^{\ep (t-s)\Delta}T_{m_{4}}(\R T_{\tilde{m}_{1}(b(\eta)+b(\eta-\si)-2)}(\R w,\R w),\R w)\d s\\
&&-\izt s e^{isb(D)}e^{\ep (t-s)\Delta}T_{m_{4}}(\R T_{\tilde{m}_{1}}(e^{isb(D)}\p_{s}\R f,\R w),\R w)\d s+ \quad similar\quad  term \\
&&+\izt s e^{isb(D)}e^{\ep (t-s)\Delta}[T_{\p_{s}m_{4}}(\R T_{\tilde{m}_{1}}(\R w,\R w),\R w)+T_{m_{4}}(\R T_{\p_{s}\tilde{m}_{1}}(\R w,\R w),\R w)]\d s\\
&\define&Y_{2221}+\cdots Y_{2225}
\eeno
where the following notations has been used:
\beno
&&m_{4}(\xi,\eta,s)=\f{\p_{\xi}\tilde{\phi}}{b(\xi)+b(\xi-\eta)-2}\f{m}{\phi}\ler\Psi_{\leq \lsr^{\delta_0}}(\xi-\eta)\Psi_{\geq 3\lsr^{-\delta_0/5
}}(\xi-\eta)\tilde{\Psi}_{\lesim \lsr^{\delta_0}}(\eta),\\
&&\tilde{m}_1(\eta,\si,s)=\Psi_{\leq \lsr^{-\delta_0}}(\eta)\Psi_{\leq \lsr^{\delta_0}}(\eta-\si)\Psi_{\leq \lsr^{\delta_0}}(\si)\Psi_{\leq 2\lsr^{-\delta_0}}(\si)%\Psi_{\leq \lsr^{-\delta_0}}(\eta)
\eeno
We have again, as in Lemma \ref{trilinear estimate}, for any $1\leq p\leq \infty$, $\f{1}{p}=\f{1}{p_1}+\f{1}{p_2},$
\beqs
\|T_{m_4}(f,g)\|_{L^p}\lesim \lsr^{9\delta_0} \|f\|_{L^{p_1}}\|g\|_{L^{p_2}},\quad
\|T_{\tilde{m}_1}(u,v)\|_{L^p}\lesim \|u\|_{L^{p_1}}\|v\|_{L^{p_2}}.
\eeqs
 For $Y_{2221}$,
 \beno
 \|Y_{2221}\|_{W^{4+2\delta,2_{\delta}}}&\lesim& t \ltr^{14\delta_0}\|\R T_{\tilde{m}_{1}}(\R w,\R w)\|_{L^{\f{1}{2\delta}}}\|\R w\|_{L^{(\f{1}{2_{\delta}}-2\delta)^{-1}}}\\
 &\lesim& \ltr^{1+14\delta_0-2(1-2\delta)}\|U\|_{X_T}^{3}\lesim \|U\|_{X_T}^{3}
 \eeno
 if $\delta\leq \delta_0\leq \f{1}{50}$.
 
For $Y_{2222}$, owing to the above estimate and the fact $e^{\ep(t-s)\Delta}\epd \chi^{L}$ is $L^{2-\delta}$ multiplier with norm less than $\ltsr^{-1}$, we get that:
\beno
\|Y_{2222}\|_{W^{4+2\delta,2_{\delta}}}\lesim \izt  \ltsr^{-1}\lsr^{-(1-18\delta_0)}\d s\|U\|_{X_T}^{3}\lesim \|U\|_{X_T}^{3}
\eeno
 
For $Y_{2223}$, since on the support of $m_{4}(\xi,\eta,s)$, $|\xi-\eta|\geq |\eta|$,
and $|\xi-\eta|\gtrsim \lsr^{-\delta_0}$,
 it is not hard to see that 
 \beqs
 \|\cF^{-1}(m_{4}\lxr^{4+3\delta}\lxmer^{-8}\ler^{-8})\|_{L^1}\lesim \|(1+\p_{\xi}^{3}+\p_{\eta}^{3})(m_{4}\lxr^{4+3\delta}\lxmer^{-8}\ler^{-8})\|_{L^2}\lesim \lsr^{\f{4}{5}\delta_0}.
 \eeqs
Write also $b(\eta)+b(\si)-2=\f{1-\ep^2|\eta-\si|^{2}}{b(\eta-\si)+1}|\eta-\si|^{2}+\f{1-\ep^2|\si|^{2}}{b(\si)+1}|\si|^{2},$
 we thus could estimate $Y_{2223}$ by
\beno
\|Y_{2223}\|_{W^{4+2\delta,2_{\delta}}}&\lesim&\izt \lsr^{1+\delta}\lsr^{\f{4}{5}\delta_0}\|\f{1-\ep^2\Delta}{b(D)}|\na|^{2}P_{\lesim \lsr^{-\delta_0}}\R w\|_{L^{\f{1}{\delta}}}\|\R w\|_{L^{\f{1}{\delta}}} \|w\|_{W^{8,(\f{1}{2_{\delta}}-2\delta)^{-1}}}\d s\\
&\lesim&\izt \lsr^{1+2\delta}\lsr^{\f{4}{5}\delta_0}\lsr^{-2\delta_0}\lsr^{-2(1-2\delta)}\d s\|U\|_{X_T}^3\lesim \|U\|_{X_T}^3
\eeno
if $ 10\delta<\delta_0$.
For $Y_{2224}$ it is not tough because it is essentially quartic. Similar to that of $Y_{2221}$, we have
\beno
\|Y_{2224}\|_{W^{4+2\delta,2_{\delta}}}&\lesim&\izt \lsr^{1+\delta}\lsr^{14\delta_0}
\|\R w\|_{L^{\f{1}{\delta}}}^2 \|e^{isb(D)}x\R f\|_{L^{(\f{1}{2_{\delta}}-2\delta)^{-1}}}\d s\\
&\lesim&\izt \lsr^{1+\delta}\lsr^{14\delta_0}\lsr^{-2(1-2\delta)}\lsr^{-1}\d s\|U\|_{X_T}^3\lesim \|U\|_{X_T}^3
\eeno
where in the above, the following identity has been used $$e^{isb(D)}\p_{s}\R f=\R\epd w+\R\lnr (\R w)^{2}+\R H$$
from which one easily gets:
$\|e^{isb(D)}\p_{s}\R f\|_{H^1}\lesim \lsr^{-1}\|U\|_{X_T}^{3}$.

For the last term $Y_{2225}$, one notices that when we take the time derivative on $m_{4}(\xi,\eta,s)$ or $\tilde{m_1}(\eta,\si,s)$, there will emerge a power $\lsr^{-1}$ which is enough for us to close the estimate.
For instance, $$\p_{s}\Psi(\lsr^{\delta_0}\eta)= \eta \cdot \na\Psi(\lsr^{\delta_0}\eta)\lsr^{\delta_0}\f{s}{\lsr}\lsr^{-1}\define \tilde{\Psi}(\lsr^{\delta_0}\eta)\f{s}{\lsr}\lsr^{-1}$$
where $\tilde{\Psi}$ has the same properties as $\psi$ that we need:
compactly supported, smooth.

\begin{lem}\label{lemvector}
Recall $b(x)=\sqrt{1+|x|^2-\ep^2|x|^{4}}$, with $x\in\mathbb{R}^2$, $\ep\in(0,1] $. There exists a $2\times 2$ matrix $S$ such that the following identity holds:
\beno
\f{(1-2\ep^2|\xi|^2)x}{b(x)}-\f{(1-2\ep^2|y|^2)y}{b(y)}=S(x,y)(x-y)
\eeno
Besides, if $\ep|x|^2\leq 3\kpz,\ep|y|^2\leq 3\kpz$ with $\kpz\leq\f{1}{200}$, then $S$ is invertible. Moreover, for any $\ep\in (0,1]$ and $\alpha,\beta\in \mathbb{N}^2$, the following uniform (in $\varepsilon$) estimate holds:
\beno
&&|\p_{x}^{\alpha}\p_{y}^{\beta}S(x,y)|\lesssim_{\alpha,\beta,\kpz} \f{1}{\langle y\rangle},\\
&&|\p_{x}^{\alpha}\p_{y}^{\beta}S^{-1}(x,y)|\lesssim_{\alpha,\beta,\kpz} ({\langle x\rangle}+{\langle y\rangle})^{3}.
\eeno
\end{lem}
\begin{proof}
\beno
&&\f{(1-2\ep^2|\xi|^2)x}{b(x)}-\f{(1-2\ep^2|y|^2)y}{b(y)}\\
&=&(1-2\ep^2|x|^2)x(\f{1}{b(x)}-\f{1}{b(y)})+\f{(1-2\ep^2|x|^2)x-(1-2\ep^2|y|^2)y)}{b(y)}\\
&=&-(1-2\ep^2|x|^2)(|x|^2-|y|^2)x\f{1-\ep^2(|x|^2+|y|^2)}{b(x)b(y)(b(x)+b(y)}+\f{(1-2\ep^2|x|^2)(x-y)-2\ep^2(|x|^2-|y|^2)y}{b(y)}\\
&=&\f{1-2\ep^2|x|^2}{b(y)}[\Id_{2\times2}-\f{\ep^2(|x|^2+|y|^2)}{b(x)(b(x)+b(y))}x\otimes(x+y)-2\f{\ep^2}{1-2\ep^2|x|^2} y\otimes(x+y)](x-y)\\
&\define& S(x-y).
\eeno
%Moreover,
We now compute $\det S$.
\beno
\det S &=& (\f{1-2\ep^2|x|^2}{b(y)})^2[1-\f{1-\ep^2(|x|^2+|y|^2)}{b(x)(b(x)+b(y))}x+\f{2\ep^2 }{1-2\ep^2|x|^2}y]\cdot(x+y)\\
&=&\f{(1-2\ep^2|x|^2)^2}{b^2(y)b(x)(b(x)+b(y))}\\
&&\qquad\bigg(1+\ep^2|x|^2|y|^2+b(x)b(y)-\big(1-\ep^2(|x|^2+|y|^2)\big)x\cdot y-\f{2\ep^2(x+y)\cdot y}{1-2\ep^2|x|^2}b(x)\big(b(x)+b(y)\big)\bigg).
\eeno
We note that if $\ep\leq 1,
\kpz\leq\f{1}{100}$, $1-2\ep^2|x|^2\geq\f{1}{2}$
and
\beqs
4\ep^2 b(x)\big(b(x)+b(y)\big)(x+y)\cdot y%{1-2\ep^2|x|^2}
\leq 9\ep^2(|x|^2+|y|^2)(b^2(x)+b^2(y))\leq %6\ep^2(2+|x|^2+|y|^2)^2\leq
108\kpz(\ep+3\kpz)\leq\f{2}{3}.
\eeqs
We thus have:
$$\det S\geq \f{1}{4b^2(y)}\f{\f{1}{3}+b(x)b(y)-x\cdot y}{b(x)(b(y)+b(x))}.$$
It is thus easy to see that:
\beq\label{lowerbdd}
|\p_{x}^{\alpha}\p_{y}^{\beta}(\f{1}{\det Q})|\lesssim_{\alpha,\beta,\kpz}\langle y \rangle^2 \langle x \rangle(\langle x\rangle+\langle y \rangle)^2.
\eeq
Besides, direct computations shows that:
\beno
|\p_{x}^{\alpha}\p_{y}^{\beta}S(x,y)|\lesssim_{\alpha,\beta,\kpz}\langle y\rangle^{-1},
\eeno
which combined with \eqref{lowerbdd}, yields:
\beno
|\p_{x}^{\alpha}\p_{y}^{\beta}S^{-1}(x,y)|\lesssim_{\alpha,\beta,\kpz}(\langle x\rangle+\langle y\rangle)^3.
\eeno
\end{proof}

\section*{Acknowledgement}{The author %would like to express great gratitude 
is indebted to his supervisor Professor Fr\'ed\'eric Rousset
for his kind guidance,
%encouragement and 
careful reading of the paper and constructive suggestions which improve the presentation. The author also benefits from the conversation with professor Corentin Audiard on the uniform stability of the Navier-Stokes-Korteweg system.

\bibliographystyle{abbrv}

\nocite{*}
\bibliography{2dNSP}

\end{document}